\documentclass[12pt]{article}

\input{settings} % settings file with all the packages

\author{R. Abedin, S. Maximov, A. Stolin*}

\title{Generalized classical Yang-Baxter equation and regular decompositions}
\newcommand\blfootnote[1]{%
  \begingroup
  \renewcommand\thefootnote{}\footnotetext{#1}%
  \addtocounter{footnote}{-1}%
  \endgroup
}

\usepackage[colorinlistoftodos]{todonotes}

\begin{document}

%%%%%%%%%%%%% ABSTRACT %%%%%%%%%%%%%%%%%%%
\maketitle
\thispagestyle{empty}
\begin{abstract}
    The focus of the paper is on constructing new solutions of the generalized classical Yang-Baxter equation (GCYBE) that are not skew-symmetric.
    Using regular decompositions of finite-dimensional simple Lie algebras, we construct Lie algebra decompositions of $\mathfrak{g}(\!(x)\!) \times \mathfrak{g}[x]/x^m \mathfrak{g}[x]$. 
    The latter decompositions are in bijection with the solutions to the GCYBE.
    Under appropriate regularity conditions, we obtain a partial classification of such solutions.
    The paper is concluded with the presentations of the Gaudin-type models associated to these solutions.
\end{abstract}

\blfootnote{(A.S.) Corresponding author. Department of Mathematical Sciences, 
Chalmers University of Technology  and  the University of Gothenburg, 412 96 Gothenburg, Sweden, astolin@chalmers.se.}
\blfootnote{(R.A.) Department of Mathematics, ETH Zürich, 8092 Zürich, Schweiz,  raschid.abedin@math.ethz.ch;}
\blfootnote{(S.M.) Institut für Mathematik, Universität Paderborn,
33098 Paderborn, Germany,\hfill \\ 
 stmax@math.upb.de;}

%%%%%%%%%%% TABLE OF CONTENTS %%%%%%%%%%%%
\newpage
\thispagestyle{empty}
\tableofcontents

%%%%%%%%%%% INTRODUCTION %%%%%%%%%%%%%%%%
\newpage
\section{Introduction}
The $r$-matrix approach is a fundamental method for constructing integrable systems; see e.g.\ \cite{chari_pressley,adler_moerbeke_vanhaecke,babelon_bernard_talon}.
Consider a mechanical system with a phase space $M$ and a Hamiltonian \(H\). If this system admits a Lax representation, i.e.\ the equations of motions with respect to \(H\) are equivalent to 
$$
\frac{dL}{dt} = [P, L]
$$
for some functions \(P\) and \(L\) on \(M\) with values in some Lie algebra \(\mathfrak{L}\), invariant polynomials of this Lie algebra automatically define constants of motion of the system. The involutivity property for these constants of motion is equivalent to the fact that Lax matrix \(L\) satisfies the relation
\begin{equation}%
\label{eq:russian_formula}
    \{ L \ot  L\} = [L \ot 1, r] - [1 \ot L,r^{21}]
\end{equation}
for some function $r$ on \(M\) with values in \(\mathfrak{L} \otimes \mathfrak{L}\). 
The statements we just made can be adjusted in a way so that they remain valid even for some infinite-dimensional Lie algebras $\mathfrak{L}$. 
For example, we consider one of such cases when
\(\mathfrak{L} = \fg(\!(x)\!)^{\oplus n}\), where $\fg(\!(x)\!)$ is the Lie algebra of formal Laurent power series with coefficients in a finite-dimensional simple complex Lie algebra \(\fg\). 
In this case, \(L = L(x)\) and \(r = r(x,y)\) depend on the formal parameters.

The Jacobi identity for the Poisson bracket in \cref{eq:russian_formula} imposes some constraints on the corresponding function $r$. If \(r\) is constant along \(M\) a sufficient condition that these constraints are satisfied is given by the generalized classical Yang-Baxter equation (GCYBE)
\begin{equation*}
    [r^{12}(x_1,x_2),r^{13}(x_1,x_3)] + [r^{12}(x_1,x_2),r^{23}(x_2,x_3)] + [r^{32}(x_3,x_2),r^{13}(x_1,x_3)]  = 0.
\end{equation*}
A key idea of the $r$-matrix method is to start with a solution to GCYBE and construct a mechanical system possessing many integrals of motion. 
For example, to every such solution, one can associate a classical integrable system by considering the Poisson-commuting Hamiltonians
\begin{equation*}
    H_i \coloneqq \sum_{k \neq i} r(u_k, u_i)^{(ki)} + \frac{1}{2}(g(u_i,u_i)^{(ii)} + \tau(g(u_i,u_i))^{(ii)}),
\end{equation*}
where $u_i$ are points inside the domain of the definition of $r$.
This expression can be understood as a quadratic polynomial on \(\fg^{*,\oplus n}\) and therefore as an element of the symmetric algebra \(S(\fg^{\oplus n}) \cong S(\fg)^{\otimes n}\).
If one replaces the space of functions \(S(\fg)^{\otimes n}\) by the universal enveloping algebra \(U(\fg)^{\otimes n}\), one obtains a quantum integrable system which generalizes the Gaudin models; see \cite{skrypnyk_spin_chains,skrypnyk_new_integrable_gaudin_type_systems}.

This motivates the search for solutions to the GCYBE. 
One well-known strategy is to consider certain Lie algebra decompositions into two subalgebras.
For example, a decomposition of the form 
$$
\fg(\!(x)\!) = \fg[\![x]\!] \oplus W
$$
for a subalgebra $W \subseteq \fg(\!(x)\!)$
leads to a generalized $r$-matrix of the form
$$
    r(x,y) = \frac{\Omega}{x-y} + g(x,y),
$$
where $\Omega$ is the quadratic Casimir element of $\fg$.
This idea is further developed in \cite{abedin_maximov_stolin_quasibialgebras}.
More formally, if
$$
r(x,y) = \frac{y^m \Omega}{x-y} + g(x,y) \in (\fg \ot \fg)(\!(x)\!)[\![y]\!]
$$
solves the GCYBE, then it corresponds uniquely to a Lie algebra decomposition
$$
L_m \coloneqq \fg(\!(x)\!) \times \fg[x]/x^{m} \fg[x] = \mathfrak{D} \oplus W,
$$
where $\mathfrak{D} = \{ (f,[f]) \mid f \in \fg[\![x]\!] \}$ is the diagonal embedding of $\fg[\![x]\!]$ into $L_m$ and $W$ is a complementary subalgebra determined by $r$.

In this paper, we consider the problem of classification of generalized classical $r$-matrices.
This problem in its full generality is too wild, so we impose some natural additional restrictions.
Formally, we consider those generalized $r$-matrices for which the corresponding $W \subseteq L_m$ satisfies the following three restrictions:
\begin{enumerate}
    \item $(x^{-1}, 0) W \subseteq W$ and $(0, [x]) W \subset W$;
    \item $[(h,h),W] \subseteq W$ for any $h$ in a Cartan subalgebra $\fh \subseteq \fg$;
    \item $W_+ \subseteq x^{N} \fg[x^{-1}]$ for some positive integer $N$, where $W_+ \coloneqq (1,0) W$ is the projection of $W$ onto the left component $\fg(\!(x)\!)$ of $L_m$.
\end{enumerate}
The first two restrictions are motivated by the utility of the resulting integrable systems \cite{skrypnyk_spin_chains,golubchik_sokolov_integrable_top_like_systems}
and the last condition allows us to apply the theory of maximal orders, developed in \cite{stolin_maximal_orders,stolin_sln}.
Subalgebras $W \subseteq L_m$ with the above-mentioned properties are called regular.

The classification of regular subalgebras is further split into several classification subproblems.
We were able to completely resolve some of them and to construct a vast set of examples for the remaining ones. 
The key idea that was used in all the cases is the reduction of regular decompositions $L_m = \mathfrak{D} \oplus W$ to regular partitions
$\Delta = \Delta_1 \oplus \dots \oplus \Delta_\ell$ of the underlying irreducible root system of $\fg$.

Additionally, we introduce a technique for constructing weakly regular subalgebras $W \subseteq L_m$, i.e.\ those for which condition 2.\ above is replaced with a weaker one
\begin{enumerate}
    \item [2'.] $[h, W_\pm] \subseteq W_\pm$.
\end{enumerate}
The construction is based on a generalization of Belavin-Drinfeld triples. 
The corresponding $r$-matrices are written explicitly.

\subsection{Structure of the paper}
Let $\fh$ be a fixed Cartan subalgebra of $\fg$.
A decomposition of $\fg$ of the form
\begin{equation*}%
\label{eq:reg_decomp_intro}
    \fg = \fg_1 \oplus \dots \oplus \fg_\ell
\end{equation*}
is called regular if
\begin{enumerate}
    \item all $\fg_i$ and $\fg_i \oplus \fg_j$ are Lie subalgebras of $\fg$ and
    \item $[\fh, \fg_i] \subseteq \fg_i$ for all $1 \le i \le \ell$.
\end{enumerate}
We consider such decompositions as building blocks for regular decompositions $L_m = \mathfrak{D} \oplus W$.
Note that the invariance under the adjoint action of Cartan subalgebra implies that each $\fg_i$ has the form 
$$
\fg_i = \fs_i \bigoplus_{\alpha \in \Delta_i} \fg_\alpha,
$$
for some vector subspace $\fs_i \subseteq \fh$ and a subset $\Delta_i$ of roots of $\fg$.
Forgetting the Cartan part of $\fg$ in \cref{eq:reg_decomp_intro} we obtain a partition of the corresponding root system 
$$
    \Delta = \Delta_1 \sqcup \dots \sqcup \Delta_\ell,
$$
where $\Delta_i$ and $\Delta_i \sqcup \Delta_j$ are closed with respect to root addition.
Such a partition of a root system is again called regular.
Regular decompositions of simple Lie algebras $\fg$ and regular partition of the corresponding root systems are completely classified in \cite{dokovic_check_hee, maximov_regular}.
We present this classification in \cref{sec:regular_decompositions}.

The theory of maximal orders, that allows us to split the classification of regular decompositions $L_m = \mathfrak{D} \oplus W$ further into smaller subproblems, is presented in \cref{subsec:maximal_orders_L0}.
More precisely, the theory states that, up to a certain equivalence, every regular subalgebra \(W \subseteq L_m\) satisfies
$$
W \subseteq \mathfrak{P}_i \times \fg[x]/x^m \fg[x],
$$
where $\mathfrak{P}_i$ is the parabolic subalgebra of $\fg[x,x^{-1}] \subseteq \fg(\!(x)\!)$ corresponding to the simple root $\alpha_i$ with $0 \le i \le \textnormal{rank}(\fg)$ of the Lie algebra $\fg[x,x^{-1}]$.
Moreover, each parabolic subalgebra $\mathfrak{P}_i$ is associated with an integer $k_i$, called the type of $\mathfrak{P}_i$.
Consequently, we have a reduction of the original problem to the classification of regular subalgebras $ W \subseteq L_m$ with $m \in \{ 0,1,2 \}$ and different types $0 \le k \le 6$.

We complete \cref{sec:preliminaries} with an explanation of the relation between regular decompositions $L_m = \mathfrak{D} \oplus W$ and solutions of the GCYBE.

A full classification of regular subalgebras $W \subseteq L_0$ of type $0$ is obtained in \cref{sec:reg_decomp_type_0}.
Precisely, we prove the following result.
\begin{maintheorem}
Let $W \subseteq \fg(\!(x)\!)$ be a regular subalgebra of type $0$.
Then we can find a regular partition of the root system of $\fg$
\begin{equation*}%
\Delta = \bigsqcup_{i = 1}^n \Delta_i
\end{equation*}
and constants $a_1, \dots, a_n \in F$ such that
\begin{equation*}%
%\label{eq:intro_main_form_0_type_0}
   W = W_\phi \bigoplus_{i=1}^n \bigoplus_{\alpha \in \Delta_i} (x^{-1} - a_i)\fg_{\alpha} [x^{-1}],
\end{equation*}
where $W_\phi \subseteq \fh[x^{-1}]$ is defined by a linear map $\phi \colon \fh \to \fh$ compatible with the regular partition of \(\Delta\); see  \cref{prop:W_from_regular_splittings}.
The converse direction is also true. 

The solution of the GCYBE associated to the subalgebra above is given by 
    \begin{equation*}
        r(x,y) = \frac{\Omega}{x-y} + \left(\frac{\phi}{y\phi - 1} \otimes 1\right)\Omega_\fh + \sum_{i = 1}^n \frac{a_i\Omega_i}{a_iy-1},
    \end{equation*}
    where \(\Omega \in (\fh \otimes \fh) \bigoplus_{\alpha \in \Delta} (\fg_\alpha \otimes \fg_{-\alpha})\) is the quadratic Casimir element and \(\Omega_\fh \in \fh \otimes \fh\) as well as \(\Omega_i \in \bigoplus_{\alpha \in \Delta_i}\fg_{\alpha} \otimes \fg_{-\alpha}\) are the corresponding components of \(\Omega\).
\end{maintheorem}

In \cref{subsec:regular_decomp_0_type_1} we look at regular subalgebras $W \subseteq \fg(\!(x)\!)$ of type $1$. We give a complete description of regular subalgebras contained inside the parabolic subalgebras corresponding to the roots presented in \cref{fig:resolved_cases}. The description is given in terms of regular decompositions $\fg = \fg_1 \oplus \fg_2$ and some additional data.

\begin{figure}[H]
    \centering
    \includegraphics[scale = 0.7]{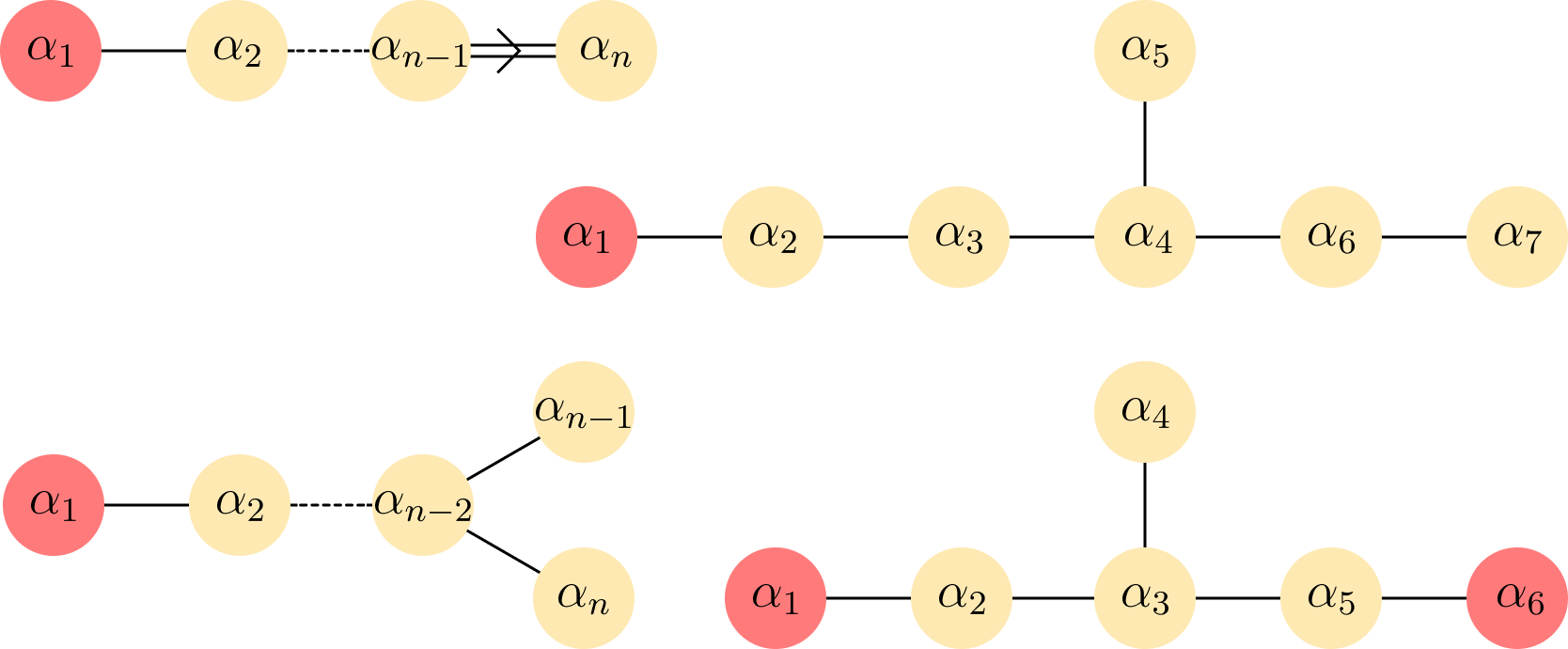}
    \caption{Regular subalgebras contained in the parabolic subalgebras corresponding to the red roots are completely classified.}
    \label{fig:resolved_cases}
\end{figure}
\noindent In order to present this result, we need to introduce the following susbets of roots:
\begin{align*}
   \Delta_\pm^{<m \alpha_i} &\coloneqq \{
   \alpha \in \Delta_\pm \mid \alpha = \pm \sum_{i=1}^n c_i \alpha_i, \ 0 \le c_i < m\}, \\
   \Delta_{\pm}^{\ge m\alpha_i} &\coloneqq \{\alpha \in \Delta_\pm \mid \alpha = \pm \sum_{i=1}^n c_i \alpha_i, \ c_i \ge m \}.
\end{align*}
Then the precise statement for the result mentioned above is the following.
\begin{maintheorem}
    Let $\Delta$ be the irreducible root system corresponding to $\fg$.
    Assume $\alpha_i$ is a simple root of $\Delta$ with $k_i = 1$.
    For any partition $\Delta^{< \alpha_i} = \Delta_1 \sqcup \Delta_2$ into two closed subsets and two constants $a_1, a_2 \in F^\times$ the formula
    
    \begin{equation*}
    \begin{aligned}
        W =W_\phi 
        &\bigoplus_{\alpha \in \Delta_1}
        (x^{-1} - a_1) \fg_\alpha[x^{-1}]
        \bigoplus_{\alpha \in \Delta_2} 
        (x^{-1} - a_1) \fg_\alpha^{a_2} [x^{-1}] \\
        &\bigoplus_{\beta \in \Delta_+^{\ge \alpha_i}} x(x^{-1} - a_1)(x^{-1} - a_2) 
        \fg_\beta [x^{-1}]  \bigoplus_{\alpha \in \Delta_-^{\ge \alpha_i}}
     x^{-1}\fg_{\alpha}[x^{-1}],
    \end{aligned}
    \end{equation*}
    defines a regular subalgebra of type $1$. 
    Here $W_\phi \subseteq \fh[x^{-1}]$ is a subalgebra defined by a certain linear map $\phi \colon \fh \to \fh$; see \cref{ex:type_1_from_2_regular}.
    Moreover, every regular $W \subseteq L_0$ of type $1$ associated to a node from \cref{fig:resolved_cases} is of this form. \newline 
    The corresponding solution of the GCYBE is given by
    \begin{equation*}
    \begin{split}
        r(x,y) = \frac{\Omega}{x-y} + \left(\frac{\phi}{y\phi-1} \otimes 1\right)\Omega_\fh +\frac{a_1 \Omega_1}{a_1 y-1} + \frac{a_2 \Omega_2}{a_2 y-1} + \frac{ a_1a_2(x+y)- a_1 - a_2}{(a_1y-1)(a_2y-1)}\Omega_{\fc},
    \end{split}
    \end{equation*}
    where \(\Omega \in (\fh \otimes \fh) \bigoplus_{\alpha \in \Delta} (\fg_\alpha \otimes \fg_{-\alpha})\) is the quadratic Casimir element and 
    $$
        \Omega_i \in (\oplus_{\alpha \in \Delta_i} \fg_\alpha) \otimes \fg, \ \ 
        \Omega_{\fc} \in (\oplus_{\alpha \in \Delta_+^{\ge \alpha_i}} \fg_\alpha) \otimes \fg
    $$
    are components of \(\Omega\) lying in the corresponding subspaces.
\end{maintheorem}
\noindent In \cref{subsec:regular_decomp_0_type_1} we also present explicit constructions of regular subalgebras in $\fg(\!(x)\!)$ for the remaining type $1$ cases starting with a regular decomposition of $\fg$.

Regular subalgebras in $\fg(\!(x)\!)$ of type $k \ge 2$ are considered in
\cref{subsec:reg_0_type_2}.
The structure of such subalgbras is even wilder and their classification seems unfeasible. 
However, we present a general algorithm for constructing such objects. It is demonstrated in \cref{ex:B_n_type_2}.

\Cref{sec:regular_1_2} is devoted to regular decompositions $L_m = \mathfrak{D} \oplus W$ with $m \ge 1$.
We prove that regular subalgebras of $L_1$ of type $0$ can be reduced to regular subalgebras of $\fg \times \fg$. The latter are classified in \cref{sec:reg_constant_decomp}.

\begin{maintheorem}%
\label{mainthm:gxg_decomposition}
Let $\mathfrak{w} \subseteq \fg \times \fg$ be a subalgebra such that
\begin{itemize}
    \item $[(h,h), \mathfrak{w}]  \subseteq \mathfrak{w}$ for all $h \in \fh$ and
    \item $\Delta \oplus \mathfrak{w} = \fg \times \fg$.
\end{itemize}
Then there is a regular partition $\Delta = S_+\sqcup S_-$ and subspaces $\fs_\pm = \ft_\pm \oplus \fr_\pm \subseteq \fh$, having the properties
$\fh = \fs_+ + \fs_-$ and $\ft_+ \cap \ft_- = \{ 0 \}$, such that
\begin{equation*} 
\begin{split}
        \mathfrak{w} = \left((\ft_+ \bigoplus_{\alpha \in S_+} \fg_\alpha) \times \{0\}\right) 
        \oplus 
        \left(\{ 0 \} \times (\ft_- \bigoplus_{\beta \in S_-} \fg_\beta) \right)
        \oplus 
        \textnormal{span}_F \{ (h, \phi(h)) \mid h \in \fr_+ \},
         \end{split}
\end{equation*}
    where $\phi \colon \fr_+ \to \fr_-$ is a vector space isomorphism with no non-zero fixed points. The converse direction is also true.
    In other words, all regular subalgebras are described by partitions $\Delta = S_1 \sqcup S_2$ and the extra datum $(\ft_\pm, \fr_\pm, \phi)$ describing the gluing of the Cartan parts.
\end{maintheorem}
\noindent Similar to the $L_0$ case, in \cref{subsec:L_1_reduction_to_gxg}, we classify regular subalgebras $W \subseteq L_1$ of type $0$ in terms of regular subalgebras of $\fg \times \fg$ and some additional datum.

\begin{maintheorem}%
\label{mainthm:thm_L1}
    Let \( W \subseteq L_1 \) be a regular subalgebra of type $0$.
    Then
    \begin{equation*}
        W = \mathfrak{w} \oplus \left(\left(W_{\psi}
        \bigoplus_{i = 0}^n\bigoplus_{\alpha \in \Delta_i}(x^{-1} - a_i)\fg_\alpha[x^{-1}]\right)\times \{0\}\right),
    \end{equation*}
    where
    \begin{itemize}
        \item \(\mathfrak{w}  \subseteq \fg \times \fg\) is described by the datum \(\Delta = S_+ \sqcup S_-\) and  \((\ft_\pm,\fr_\pm,\phi)\);

        \item \(\Delta = \sqcup_{i = 0}^n \Delta_i\) is a regular decomposition with \(S_+ \subseteq \Delta_0\) and \(a_0,\dots,a_n \in F\) are distinct constants such that \(a_0 = 0\);

        \item \(\psi \colon \fh \to \fh\) is a linear map defining $W_\psi$ and compatible with the previous root space data; see \cref{prop:form_W_L_1} for details.
    \end{itemize}
    The converse holds as well.

    Furthermore, the \(r\)-matrix of \(W\) above is
    \begin{equation*}
        r(x,y) = \frac{y\Omega}{x-y} + r_{(S_\pm,\ft_\pm,\fr_\pm,\phi)} +\frac{1}{2}\Omega+\left(\frac{\psi}{y\psi - 1} \otimes 1\right)\Omega_\fh + \sum_{i = 1}^n \frac{a_i\Omega_i}{a_iy-1},
    \end{equation*}
    where \(\Omega_i \in \bigoplus_{\alpha \in \Delta_i}(\fg_\alpha \otimes \fg_{-\alpha})\) is the component of the Casimir element and \(r_{(S_\pm,\ft_\pm,\fr_\pm,\phi)}\) is a certain tensor in $\fg \ot \fg$ determined by $\mathfrak{m}$ in  \cref{eq:rmatrix_from_constant_gxg}.
\end{maintheorem}

Replacing the condition $[(h,h),W] \subseteq W$ with the weaker condition $[h, W_\pm] \subseteq W_\pm$, i.e.\ only the projections of $W$ are required to be $\fh$-invariant, we obtain the notions of a weakly regular subalgebra of $L_m$.
We show in \cref{subsec:weakly_regular_decomp} that such subalgebras can be effectively constructed using a generalized version of Belavin-Drinfeld triples combined with \cref{mainthm:thm_L1}.

In the remaining part of \cref{sec:regular_1_2} we prove that regular subalgebras $W \subseteq L_m$ with \(m > 0\) of type $k > 0$ admit a certain standard form.
Since all regular subalgebras $W \subseteq L_0$ are trivially included into the set of regular subalgebras of $L_m$, with $m \ge 1$, 
the classification of the latter objects is even wilder. 
However, we present methods for constructing such non-trivial subalgebras. 
See \cref{table:reductions} for a short summary of the results.

The paper is concluded with \cref{sec:gaudin_models}, where we relate the decompositions mentioned above to Gaudin models through the corresponding solutions to the GCYBE. We give some explicit examples of new generalized Gaudin Hamiltonians.

In \cref{sec:notations}, we placed a table with the most used notation to simplify the reading. 

\begin{table}[H]%
\centering
\begin{tabular}{|c|c|c|}
\hline
\multicolumn{3}{|c|}{$W \subseteq L_m$} \\
\hline
  $m = 0$ & type $= 0$ & \makecell{Classified by regular partitions of $\Delta$\\ and compatible linear maps $\phi \colon \fh \to \fh$. } \\
\hline
    $m = 0$ & type $> 0$ & \makecell{Cases from \cref{fig:resolved_cases} are completely classified using  decompositions \\ $\Delta = \Delta_1 \sqcup \Delta_2$ and $\phi \colon \fh \to \fh$. \\ For the remaining cases explicit constructions are presented.} \\
\hline
    $m = 1$ & type $= 0$ & \makecell{Classified by regular decompositions of $\fg \times \fg$ \\ and compatible linear maps $\phi, \psi \colon \fh \to \fh$.} \\
\hline
    $m > 0$ & type $k_i> 0$ & \makecell{Always of the form \(W = W_\fh \oplus (I_+ \times \{0\}) \oplus  (\{0\} \times I_-)\) \\ for \(I_+ \subseteq \mathfrak{P}_i\), \(I_- \subseteq \fg[x]/x^m\fg[x]\) and \(W_\fh \subseteq \fh[x] \times \fh[x]/x^m\fh[x]\). \\ Explicit constructions are presented.} \\
\hline
\end{tabular}
\caption{Overview of classification results.}
\label{table:reductions}
\end{table}

\section*{Acknowledgements}
\addcontentsline{toc}{section}{Acknowledgements}
The authors are grateful to the SUSTech International Center for Mathematics and E. Zelmanov for hosting our visit to China during the finalization of this paper and creating an environment for fruitful collaboration.\\
The authors would also like to thank T. Skrypnyk for his valuable comments.\\
The work of R.A.\ was supported by the DFG grant AB 940/1--1. It was also supported by the NCCR SwissMAP, a National
Centre of Competence in Research, funded by the Swiss
National Science Foundation (grant number 205607). \\
The work of S.M. is funded by DFG – SFB – TRR 358/1 2023 – 491392403. \\

%%%%%%%%%%%%% PRELIMIRAIES %%%%%%%%%%%%
\section{Preliminaries}%
\label{sec:preliminaries}
We fix once and for all an algebraically closed field $F$ of characteristic $0$.
Let $\fg$ be a finite-dimensional simple Lie algebra over $F$ with a Cartan subalgebra $\fh$.
We write $\Delta$ for the set of (non-zero) roots and $\pi = \{\alpha_1,\dots,\alpha_n\} \subset \Delta$ for a chosen set of simple roots. 
The choice of simple roots gives us the polarization $\Delta = \Delta_+ \sqcup \Delta_-$. 
Later we use $\alpha_0$ to denote the maximal root of $\Delta$.
Furthermore, we fix a basis 
$$\{H_{\alpha_i}, E_{\pm\alpha} \mid 1 \le i \le n, \ \alpha \in \Delta_+\}$$ of $\fg$, such that $\kappa(E_\alpha, E_{-\alpha}) = 1$ and $\kappa(H_{\alpha}, H) = \alpha(H)$, where $\kappa$ is the Killing form of $\fg$.
In particular, we have $[E_\alpha, E_{-\alpha}] = H_\alpha = \sum_{i=1}^n c_i H_{\alpha_i}$ for any $\alpha = \sum_1^n c_i \alpha_i \in \Delta_+$.

\subsection{Regular decompositions of simple Lie algebras}%
\label{sec:regular_decompositions}
An \emph{$m$-regular decomposition of $\fg$} is a decomposition of the form
\begin{equation*}%
    \fg = \bigoplus_{i = 1}^m \fg_i, \ m \ge 2,
\end{equation*}
satisfying the conditions:
\begin{enumerate}
    \item all $\fg_i$ as well as $\fg_i \oplus \fg_j$ are Lie subalgebras of $\fg$;
    \item each $\fg_i$ has the form $ \fs_i  \oplus_{\alpha \in \Delta_i} \fg_{\alpha}$ for some subspace $\fs_i \subseteq \fh$ and some subset $\Delta_i \subseteq \Delta$.
\end{enumerate}
Subalgebras of the form described in 2.\ are called \emph{regular}, motivating the name.
Equivalently, one can say that $\fg_i$ and $\fg_i \oplus \fg_j$ are subalgebras of $\fg$ invariant under the action of $\fh$.

We say that a subset $S \subseteq \Delta$ is \emph{closed} if for all $\alpha, \beta \in S$ the containment $\alpha + \beta \in \Delta$ implies $\alpha + \beta \in S$.
Regular decompositions of simple Lie algebras are closely related with \emph{regular partitions} of irreducible root systems of finite type, i.e.\ partitions
\begin{equation}%
\label{eq:reg_partition_def}
    \Delta = \bigsqcup_{i=1}^m \Delta_i, \ m \ge 2,
\end{equation}
with the property that all $\Delta_i$ and $\Delta_i \sqcup \Delta_j$ are closed.
More precisely, the results of \cite{maximov_regular} give us the following correspondence.
\begin{proposition}%
\label{prop:m_partititons}
    Given an $m$-regular partition $\Delta = S_1 \sqcup \dots \sqcup S_m$ with $m \ge 2$, we can find subspaces $\fs_1, \dots, \fs_m \subseteq \fh$ (not unique in general) such that
    $$
    \fg = \bigoplus_{i=1}^m \left( \fs_i \bigoplus_{\alpha \in S_i} \fg_\alpha \right)
    $$ 
    is a regular decomposition.
    Conversely, by forgetting the Cartan part of an $m$-regular decomposition of $\fg$ we obtain an $m$-regular partition of $\Delta$.
\end{proposition}

Regular partitions of irreducible finite root systems into two parts were classified in \cite{dokovic_check_hee}.
Later in \cite{maximov_regular}, it was shown that $m$-regular partitions with $m \ge 3$ exist only for root systems of type $A_n$, $n \ge 2$.
Such partitions were completely described in the same paper.

We now shortly recall these classifications starting with the $2$-regular case.
Let $S$ be a subset of simple roots $\pi$. 
We denote by $\Delta^S = \Delta^S_+ \sqcup \Delta^S_-$ the root subsystem of $\Delta$ generated by $S$.
Given two subsets $S, T \subseteq \pi$ we define
$$
P(S,T) \coloneqq \Delta^S \cup (\Delta_+ \setminus \Delta^T_+).
$$
The following theorem describes all $2$-partitions.
\begin{theorem}[\cite{dokovic_check_hee}, Theorem 4]%
\label{thm:invertible_sets}
Let $S \subseteq T \subseteq \pi$ be two subsets of simple roots such that $S$ is orthogonal to $T \setminus S$. 
Then 
$$
\Delta = P(S,T) \sqcup (\Delta \setminus P(S,T))
$$
is a $2$-regular partition of $\Delta$.
Moreover, up to the action of the Weyl group $W(\Delta)$, any $2$-regular partition is of this form.
\end{theorem}

Let us order the simple roots
$\pi = \{\alpha_1, \dots, \alpha_n \}$ of the system $A_n$ in such a way, that
$$
  \beta_i \coloneqq \alpha_1 + \dots + \alpha_i 
  \ \text{ and } \ \beta_i - \beta_j 
$$ 
are roots for all $1 \le i \neq j \le n$.
For convenience, we put $\beta_0 \coloneqq 0$. 
\begin{theorem}[\cite{maximov_regular}, Theorem A]
\label{thm:root_partitions}
If $\Delta = \Delta_1 \sqcup \dots \sqcup \Delta_m$, $m \ge 3$ is a regular partition, then $\Delta$ is necessarily of type $A_n$. Moreover, the following statements are true:
    \begin{enumerate}
        \item Up to swapping positive and negative roots, re-numbering elements $\Delta_i$ of the partition and action of $W(A_n)$, there is a unique maximal $(n+1)$-partition with
        $$
    \Delta_i = \{-\beta_i + \beta_j \mid 0 \le i \neq j \le n \}, \ \ 0 \le i \le n;
    $$
    \item Any other $(3 \le m < n+1)$-regular partition is obtained from the maximal one above by combining several subsets $\Delta_i$ together;
    \item Up to equivalences mentioned above, all $m$-regular partitions are described by $m$-partitions $\lambda = (\lambda_1, \dots, \lambda_m)$ of $n+1$. 
    \end{enumerate}
\end{theorem}

To pass from a regular partition $\Delta = \sqcup_{i = 1}^m \Delta_i$ to a regular decomposition $\fg = \oplus_{i = 1}^m \fg_i$, we let 
$$
\fg_i \coloneqq \fs_i \oplus \textnormal{span}_F \{E_{\alpha} \mid \alpha \in \Delta_i \}
$$
for some (non-unique) disjoint subspaces $\fs_i$ of $\fh$.
In case $m=2$, we first put $$\fs_i' \coloneqq \textnormal{span}_F \{H_\alpha \mid \pm\alpha \in \Delta_i \}.$$ Then we represent $\fh = (\fs_1' \oplus \fs_2') \oplus \fh'$ and, finally, obtain $\fs_i$ by distributing $\fh'$ into $\fs_i'$ in an arbitrary way.
In case $m \ge 3$ there is even less freedom.
\begin{theorem}[\cite{maximov_regular}, Theorem B]
Let $\mathfrak{sl}(n+1, F) = \oplus_{i = 1}^m \fg_i$ be an $m$-regular partition.
Up to swapping positive and negative roots, re-numbering $\fg_i$'s and the action of $W(A_n)$, it has one of the following forms

\vspace{1em}
\begin{tabular}{rl}
    1. & $\fg_1 = \textnormal{span}_F \{
            E_{\beta_i} \mid 1 \le i \le n \}$, \\
    ~ & $\fg_\ell = \textnormal{span}_F \Bigg\{E_{-\beta_i + \beta_j}, H_{\beta_i} \bigg| \sum_{t=1}^{\ell-2} \lambda_t < i \le \sum_{t=1}^{\ell-1} \lambda_t, \ 0 \le j \neq i \le n  \Bigg\}$,
\end{tabular}

where  $2 \le \ell \le k+1$ and $(\lambda_1, \dots, \lambda_k)$ is a $k$-partition of $n$;

\vspace{1em}
\begin{tabular}{rl}
    2. & $\fg_1 = \textnormal{span}_F \{ 
       E_{-\beta_i + \beta_j}, H_{\beta_i}, X \mid 0 \le i \le \lambda_1, \ 0 \le j \neq i \le n
       \}$, \\
    ~ & $\fg_\ell = \textnormal{span}_F \Bigg\{
       E_{-\beta_i + \beta_j}, H_{\beta_i} - X \bigg|
       \sum_{t = 1}^{\ell - 1} \lambda_t < i 
       \le \sum_{t = 1}^{\ell} \lambda_t, \ 0 \le j \neq i \le n
       \Bigg\}$,
\end{tabular}

where $2 \le \ell \le k$,  $(\lambda_1, \dots, \lambda_k)$ is a $k$-partition of $n$ and $X$ is an arbitrary vector in
$$
    (FH_{\beta_1} \oplus \dots \oplus FH_{\beta_{\lambda_1}}) \cup \left\{ H_{\beta_p}\, \Bigg|\, 2 \le m \le k, \ \lambda_m > 1, \ \sum_{t=1}^{m-1} \lambda_t < p \le \sum_{t=1}^{m} \lambda_t  \right\}.
$$
\end{theorem}

\subsection{Maximal upper-bounded subalgebras in \(\fg(\!(x)\!)\)}%
\label{subsec:maximal_orders_L0}
Subalgebras $W$ of $\fg(\!(x)\!)$ are called \emph{upper-bounded}, if they satisfy the condition
\begin{equation*}%
     W \subseteq x^N \fg[x^{-1}] \ \text{ for some } \ N \in \bZ_+,
\end{equation*}
and \emph{bounded}, if they satisfy the stronger condition 
\begin{equation}%
\label{eq:boundedness_property}
     x^{-N} \fg[x^{-1}] \subseteq W \subseteq x^N \fg[x^{-1}] \ \text{ for some } \ N \in \bZ_+.
\end{equation}
Among all upper-bounded algebras, we distinguish those that are both maximal, with respect to inclusion, and sum up with $\fg[\![x]\!]$ to the whole $\fg(\!(x)\!)$, i.e.\ 
\begin{equation}%
\label{eq:whole_algebra_property}
    \fg[\![x]\!] + W = \fg(\!(x)\!).
\end{equation}
\begin{example}
    The trivial maximal bounded algebra is given by
    \begin{equation*}%
        \mathfrak{P} \coloneqq \mathfrak{P}_0 \coloneqq \fg[x^{-1}].
    \end{equation*}
    It is clearly a maximal proper subalgebra of $\fg(\!(x)\!)$ satisfying \cref{eq:boundedness_property,eq:whole_algebra_property} with $N=1$.
\end{example}

\begin{example}
    For each simple root $\alpha_i \in \{\alpha_1, \dots, \alpha_n \}$, there is a maximal bounded algebra
    \begin{equation}%
    \label{eq:maximal_order_i}
        \mathfrak{P}_i \coloneqq 
        \fh[x^{-1}]  \bigoplus_{\alpha \in \Delta} x^{\lfloor \alpha(h_i) \rfloor} \fg_\alpha[x^{-1}],
    \end{equation}
    where \(h_i \in \fh\) are defined by \(\alpha_i(h_j) = \delta_{ij}/k_i\) with \(\alpha_0 = \sum k_i\alpha_i\), $k_i \in \mathbb{Z}_+$. 
    The notation is motivated by the fact that \(\mathfrak{P}_i\) are the standard maximal parabolic subalgebras of \(\fg(\!(x)\!)\), if the latter is considered as a completed affine Lie algebra modulo its center.
\end{example}

Maximal bounded subalgebras, which sometimes are also called maximal orders, were thoroughly studied in \cite{stolin_maximal_orders}, where the following result was proven.

\begin{proposition}
    Let $W \subseteq \fg(\!(x)\!)$ be a proper bounded subalgebra
    such that 
    $$\fg[\![x]\!] + W = \fg(\!(x)\!).$$
    Then there is an automorphism
    $\varphi \in \textnormal{Aut}_{F[\![x]\!]\textnormal{-LieAlg}}(\fg[\![x]\!]) $
    such that
    $\varphi(W) \subseteq \mathfrak{P}_i$ for some $0 \le i \le n$.
\end{proposition}

\begin{remark}
    To be precise, paper \cite{stolin_maximal_orders}
    studies subalgebras \(W \subseteq \fg(\!(x^{-1})\!)\) with the property
    \begin{equation*}%
        x^{-N} \fg[\![x^{-1}]\!] \subseteq W \subseteq x^N \fg[\![x^{-1}]\!].
    \end{equation*}
    These are called orders. It is straight-forward to transport the results from \cite{stolin_maximal_orders} to our setting, since both our setting and the one in \cite{stolin_maximal_orders} can be reduced to the study of bounded subalgebras of \(\fg[x,x^{-1}]\).
\end{remark}

It turns out, that upper-bounded subalgebras $W$ of $\fg(\!(x)\!)$ can also be placed into the maximal bounded ones.

\begin{proposition}%
\label{prop:maximal_upper_bounded}
    Let \(W\) be a subalgebra of  \(\fg(\!(x)\!)\) such that 
    $$
    W \subseteq x^{N}\fg[x^{-1}] \textnormal{ and } \dim(x^{N}\fg[x^{-1}]/W) < \infty
    $$
    for some $N \in \mathbb{Z}_+$.
    Then there exists \(g \in G(F[x,x^{-1}])\), where \(G\) is the connected semisimple affine algebraic group associated to \(\fg\), such that
    \begin{equation*}
        \textnormal{Ad}(g)W \subseteq \mathfrak{P}_i 
    \end{equation*}
    for some $0 \le i \le n$.
\end{proposition}
\begin{proof}
By enlarging $W$ to $F[x^{-1}]W$, we can assume without loss of generality that $$F[x^{-1}]W \subseteq W.$$
Let \(I \coloneqq x^{-2N-1}W \subseteq W\), then $I$ is an ideal in $W$ and 
\(I \subseteq x^{-N-1}\fg[x^{-1}]\). 
Consider the algebra $W \oplus Fc$ as a subalgebra of the affine Lie algebra \(\widehat{\fg}\), i.e.\ the central extension
\begin{equation}\label{eq:hatg}
    \widehat{\fg}= \fg[x,x^{-1}] \oplus F c,    
\end{equation}
endowed with the bracket
\[
[ax^k,bx^\ell]_{\widehat{\fg}} 
= [a,b]x^{k+\ell} + \delta_{k,-\ell}k \kappa(a,b)c.
\]
Because \(W \subseteq x^{N}\fg[x^{-1}]\) and \(I \subseteq x^{-N-1}\fg[x^{-1}]\), the algebra $I$ is also an ideal in $W \oplus Fc \subseteq \widehat{\fg}$.
Furthermore, \(I\) has finite codimension in the subalgebra 
$$
\fn_- \oplus x^{-1}\fg[x^{-1}] \subseteq \widehat{\fg},
$$
since \(W \subseteq x^N\fg[x^{-1}]\) is of finite codimension by assumption. 
Therefore, copying the proof from \cite[Proposition 3.10]{abedin_burban_geometrization_trigonometric}, we see that  \cite[Proposition 2.8]{kac_wang} implies the existence of $
g \in G(F[x,x^{-1}])$
such that \(\textnormal{Ad}(g)W \subseteq \mathfrak{P}_i\) for some \(0\le i\le n\).
\end{proof}

Adding the extra condition on $W$ to sum up with $\fg[\![x]\!]$ to the whole $\fg(\!(x)\!)$, the result of \cref{prop:maximal_upper_bounded} can be refined as follows.

\begin{corollary}%
\label{cor:inner_aut_upper_bounded_subalgebra}
    Let \(W\) be a subalgebra of  \(\fg(\!(x)\!)\) with the properties
    \begin{enumerate}
        \item \(W \subseteq x^{N}\fg[x^{-1}]\) for some $N \in \mathbb{Z}_+$ and
        \item $\fg[\![x]\!] + W = \fg(\!(x)\!)$.
    \end{enumerate}
    Then there exists \(g \in G(F[x])\) such that
    \begin{equation*}
        \textnormal{Ad}_g(W) \subseteq \mathfrak{P}_i  
    \end{equation*}
    for some $0 \le i \le n$.
\end{corollary}
\begin{proof}
    By virtue of \cref{prop:maximal_upper_bounded}, we have \(\textnormal{Ad}_g(W) \subseteq \mathfrak{P}_i\) for some element \(g \in G(F[x,x^{-1}])\). The Kac-Moody group $\widehat{G}$ associated to the affine Lie algebra $\widehat{\fg}$ (see \cref{eq:hatg}) is the central extension of $G(F[x,x^{-1}])$ by \(F^\times\) via the rescaling of \(x\); see \cite{peterson_kac}. 
    Let $B_\pm \coloneqq H U_\pm \subseteq \widehat{G}$, where the subgroups \(U_\pm\) have the property
    $$
    \textnormal{Ad}(U_\pm) = \langle e^{x^{\pm k}\textnormal{ad}(a)}\mid a \in \fg_\alpha, \alpha \in \Delta_\pm, k \in \bZ_{\ge 0}\rangle \subseteq \textnormal{Ad}(G(F[x,x^{-1}]))
    $$
    and $H$ is a maximal torus of $\widehat{G}$. We note that \(\textnormal{Ad}(H)\) stabilizes all root spaces and the Cartan subalgebra of \(\widehat{\fg}\).
    By \cite[Corollary 2]{peterson_kac} we can decompose
    $g$ inside $\widehat{G}$ as
    $$
        g = b_+ w b_-,
    $$
    for some $b_\pm \in B_\pm$ and an element $w$ inside the Weyl group $W$ of $\widehat{\fg}$. 
    Using that \(\textnormal{Ad}(b^{-1}_-)\mathfrak{P}_i = \mathfrak{P}_i\) we obtain
    $$
        \textnormal{Ad}(b_+)W \subseteq  \textnormal{Ad}(w^{-1}) \mathfrak{P}_i.
    $$
    Applying $\textnormal{Ad}(w^{-1})$ to the equality 
    $\fg[\![x]\!] + \mathfrak{P}_i = \fg(\!(x)\!)$
    we get $\fg[\![x]\!] + \textnormal{Ad}(w^{-1})\mathfrak{P}_i = \fg(\!(x)\!)$.
    The argument right before 
    \cite[Lemma 9.2.2]{abedin_thesis} shows that this is possible if and only if 
    $$
    \textnormal{Ad}(w^{-1})\mathfrak{P}_i = \textnormal{Ad}(w')\mathfrak{P}_j
    $$
    for some (possibly different) integer $0 \le j \le n$ and some element \(w'\) of the Weyl group of \(\fg\), viewed as an element of \(G\).
    Taking $g = w^{\prime,-1}b_+$ and changing $i$ to $j$ if necessary, we get the desired statement.
\end{proof}

\subsection{Regular decompositions}
Motivated by \cite{abedin_maximov_stolin_quasibialgebras,golubchik_sokolov_integrable_top_like_systems},
we focus on Lie algebras
$$
 L_m \coloneqq \fg(\!(x)\!) \times \fg[x]/x^k\fg[x], \ m \ge 0
$$
and their Lie algebra decompositions.
More precise, we study \emph{regular decompositions}, i.e.\ decompositions of the form $L_m = \mathfrak{D} \oplus W$, where
\begin{itemize}
    \item $\Delta$ is the diagonal embedding of $\fg[\![x]\!]$ into $L_m$ with a complementary subalgebra $W\subseteq L_m$;
    \item \(W\) is invariant under the adjoint action by \(\{(h,h)\mid h \in  \fh\}\);
    \item \(W\) is invariant under the multiplication by \((x^{-1},0)\) and \((0,[x])\) and
    \item The projection $W_+$ of $W$ onto the left component $\fg(\!(x)\!)$ is upper-bounded.
\end{itemize}
We call a subalgebra $W \subseteq L_m$ itself \emph{regular} if $L_m = \mathfrak{D} \oplus W$ is a regular decomposition.

Projecting a regular subalgebra $W$ onto the left component $\fg(\!(x)\!)$ of $L_m$ gives an upper-bounded subalgebra $W_+ \subseteq \fg(\!(x)\!)$ such that $W_+ + \fg[\![x]\!] = \fg(\!(x)\!)$.
By \cref{cor:inner_aut_upper_bounded_subalgebra} we can find an automorphism 
$\varphi$ such that 
$$\varphi(W_+) \subset \mathfrak{P}_i, \ 0 \le i \le \textnormal{rank}(\fg).$$
Extending this automorphism to $L_m$ we get the inclusion 
\begin{equation}\label{eq:inclusion}
    (\varphi \times [\varphi])W \subseteq \mathfrak{P}_i \times \fg[x]/x^k\fg[x].
\end{equation}
Allowing such equivalences, we can replace the last requirement on the projection of $W$ with the requirement on $W$ to be contained in 
$\mathfrak{P}_i \times \fg[x]/x^k\fg[x]$ for some $0 \le i \le \textnormal{rank}(\fg)$.
When the latter inclusion is satisfied, we say that $W$ is regular and has \emph{type $k_i$}.
For consistency, we let $k_0 \coloneqq 0$.

\begin{remark}\label{rem:restriction_on_k}
The inclusion \cref{eq:inclusion} restricts the set of integers \(m\) we need to consider. Formally, since \(\mathfrak{P}_i \subseteq x\fg[x^{-1}]\), combining \cref{eq:inclusion}  with \(L_m = \mathfrak{D} \oplus W\) we get 
\begin{equation*}
    \{0\} \times x^{2}\fg[x] / x^m \fg[x] \subseteq (\varphi \times [\varphi])W.
\end{equation*}
Therefore, \((\varphi \times [\varphi])W\) is completely determined by its image in \(L_m/(\{0\} \times x^{2}\fg[x] /x^m \fg[x] ) = L_2\) and we can assume without loss of generality that \(0 \le m \le 2\).    
\end{remark}

\subsection{Connection to the classical Yang-Baxter equation}%
\label{subsec:connection_GCYBE}
In view of \cite{abedin_maximov_stolin_quasibialgebras},
subalgebras $W \subseteq L_m$ complementary to $\mathfrak{D}$ are in bijection with formal power series \(r\) of the form 
\begin{equation*}
     r(x,y) = \frac{y^k \Omega}{x-y} + p(x,y) \in (\fg \ot \fg)(\!(x)\!)[\![y]\!],
\end{equation*}
where $\Omega \in \fg \ot \fg$ is the quadratic Casimir element of $\fg$, which satisfy the
\emph{generalized classical Yang-Baxter equation} (GCYBE)
\begin{equation}%
\label{eq:gcybe}
    [r^{12}(x_1,x_2),r^{13}(x_1,x_3)] + [r^{12}(x_1,x_2),r^{23}(x_2,x_3)] + [r^{32}(x_3,x_2),r^{13}(x_1,x_3)]  = 0.
\end{equation}
Such a series can be viewed as a generating series of $W$. 
For a regular subalgebra \(W \subseteq L_m\), the upper-boundedness of $W$ guarantees that $p$
is a polynomial in $(\fg \ot \fg)[x,y]$
and $\fh$-invariance gives the identity
$$
[h \otimes 1 + 1 \otimes h,p(x,y)] = 0
$$ 
for all $h \in \fh$.
The fact that $W$ is an $F[x^{-1}]$-module, on the other hand, has no nice interpretation on the side of \(r\).

\section{Regular decompositions \(\fg(\!(x)\!) = \fg[\![x]\!] \oplus W\)}%
\label{sec:reg_decomp_type_0}
We start the study of regular subalgebras $W \subseteq L_m$ with the simplest case, namely $m=0$ and hence $L_0 \cong \fg(\!(x)\!)$.

\subsection{Decompositions \(\fg(\!(x)\!) = \fg[\![x]\!] \oplus W\)}
Let us first consider splittings \(\fg(\!(x)\!) = \fg[\![x]\!] \oplus W\) for subalgebras \(W\) of \(\fg(\!(x)\!)\) which satisfy \(x^{-1}W \subseteq W\) but are not-necessarily \(\fh\)-invariant. 
It is not hard to see that in this case we can write
\begin{equation*}%
    W = W_A \coloneqq A(x^{-1}\fg[x^{-1}])
\end{equation*}
for some series
\begin{equation}%
\label{eq:A_general_form}
    A = 1 + Rx + Sx^2 + \dots 
\end{equation}
which is considered as an \(F(\!(x)\!)\)-linear map \(A \colon \fg(\!(x)\!) \to \fg(\!(x)\!)\).
Clearly, the subspace $W_A$ for an arbitrary series \cref{eq:A_general_form} is not a subalgebra in general.
The next statement provides necessary and sufficient conditions on $A$ for $W_A$ to be a subalgebra.
\begin{lemma}%
\label{eq:normal_form_W_A}
    The vector space \(W_A \subseteq \fg(\!(x)\!)\) is a subalgebra if and only if for all \(a,b \in \fg\)
    \begin{equation}%
    \label{eq:condition_on_A}
        [Aa,Ab] = A([a,b] + x([Ra,b] + [a,Rb]-R[a,b]))
    \end{equation}
    holds.
\end{lemma}
\begin{proof}
    Clearly, \(W_A\) is a subalgebra if \cref{eq:condition_on_A} is satisfied. On the other hand, if \(W_A\) is a subalgebra, we have that
    \begin{equation*}
        x^{-2}[Aa,Ab] = Ac
    \end{equation*}
    for some unique \(c \in x^{-2}\fg[\![x]\!]\). 
    On one side, the form of endomorphism $A$ gives the containment
    \begin{equation}%
    \label{eq:commutator_Aa_Ab}
        x^{-2}[Aa,Ab] \in x^{-2}[a,b] + x^{-1}([Ra,b] + [a,Rb]) + \fg[\![x]\!].
    \end{equation}
    On the other side, since \(W_A \cap \fg[\![x]\!] = \{0\}\), the coefficients in front of $x^{-2}$ and $x^{-1}$ determine completely the $\fg[\![x]\!]$ part of the element
    \cref{eq:commutator_Aa_Ab}.
    Therefore, \(c = c_2x^{-2} + c_1x^{-1}\) for 
    \begin{equation*}
        c_2 = [a,b] \textnormal{ and } c_1 + Rc_2 = [Ra,b] + [a,Rb]
    \end{equation*}
    and \cref{eq:condition_on_A} holds true.
\end{proof}

We can see from \cref{prop:maximal_upper_bounded} that after applying a polynomial automorphism, every upper bounded subalgebra \(W\) complementary to \(\fg[\![x]\!]\) is contained in \(\mathfrak{P}_i\subseteq x\fg[x^{-1}]\). 
Looking at coefficients in front of higher powers of $x$ in \cref{eq:condition_on_A} we get the following statement.

\begin{lemma}%
\label{lem:R_S_relations_type_i}
    The vector subspace \(W_A \subseteq \fg(\!(x)\!)\) is an upper-bounded subalgebra if and only if, after applying an \(F[x]\)-linear automorphism of \(\fg[x]\), the following conditions hold for all $a,b \in \fg$:
    \begin{enumerate}
        \item \(A = 1 + Rx+Sx^2\);
        \item \(R([Ra,b] + [a,Rb]-R[a,b]) - [Ra,Rb] = [Sa,b]+[a,Sb] - S[a,b]\);
        \item \([Sa,Sb] = 0\) and
        \item $[Sa,Rb] + [Ra,Sb] = S([Ra,b] + [a,Rb]-R[a,b])$.
    \end{enumerate}
    ~
\end{lemma}

When the subalgebra \(W_A\) is \(\fh\)-invariant,
% in \cref{eq:normal_form_W_A}
we have \([\fh,A] = 0\) and therefore \(A(x^{-1}\fg_\alpha) \subseteq f_\alpha\fg_\alpha\) for some \(f_\alpha \in x^{-1} + F[\![x]\!]\). In particular, \cref{eq:normal_form_W_A} specializes to \(\fh\)-invariant decompositions as follows.

\begin{lemma}%
\label{lemm:regular_subalgebra_standard_form}
    Let \(\fg(\!(x)\!) = \fg[\![x]\!] \oplus W\) be a decomposition for a subalgebra \(W\) satisfying \(x^{-1}W \subseteq W\) and \([\fh,W]\subseteq W\).
    Then we can write
    \begin{equation*}%
        W = W_\fh \bigoplus_{\alpha \in \Delta} f_{\alpha}\, \fg_\alpha[x^{-1}],
    \end{equation*}
    where $W_\fh$ is an $F[x^{-1}]$-invariant subspace of $\fh(\!(x)\!)$ and $f_{\alpha} \in x^{-1} + F[\![x]\!]$ satisfy the relations
    \begin{equation}%
    \label{eq:commutators_regular_W_0}
        f_\alpha f_\beta = (x^{-1} + c_{\alpha \beta}) f_{\alpha + \beta}, \quad \forall \alpha, \beta, \alpha+\beta \in \Delta,
    \end{equation}
    for some constants $c_{\alpha \beta} \in F$.
\end{lemma}

\begin{remark}
    Any Lie algebra automorphism \(A\) of \(\fg(\!(x)\!)\) given by \cref{eq:A_general_form} satisfy \cref{eq:condition_on_A}, since for such automorphisms \([Ra,b] + [a,Rb] - R[a,b] = 0\). 
    In this case \(W_A \cong x^{-1}\fg[x^{-1}]\).

    In the setting of \cref{lemm:regular_subalgebra_standard_form} this observation can be interpreted as follows.
    There is a trivial choice of series $f_\alpha$ satisfying the relations \cref{eq:commutators_regular_W_0}.
    More precisely, we take arbitrary series $f_{\alpha_i} \in x^{-1} + F[\![x]\!]$
    for simple roots $\alpha_i$
    and for any root of the form
    $\alpha = \sum_{i = 1}^{n} k_i \alpha_i$ put
    \begin{equation}
        f_{\alpha} \coloneqq 
        x^{k-1}
        \prod_{i = 1}^n 
        f_{\alpha_i}^{k_i} \ \text{ and } \ 
        f_{-\alpha} \coloneqq 
        x^{-k-1}
        \prod_{i = 1}^n 
        f_{\alpha_i}^{-k_i},
    \end{equation}
    where $k \coloneqq \sum_{i=1}^n k_i$.
    In this case, the relations
    \cref{eq:commutators_regular_W_0}
    hold true with $c_{\alpha,\beta} = 0$.
    However, this choice is not interesting, because the assignment \(E_{\alpha_i} \mapsto xf_{\alpha_i}E_{\alpha_i}\) already defines an automorphism \(A\) of \(\fg(\!(x)\!)\) of the form \cref{eq:A_general_form} that maps \(W_0\) to \(W\) defined by the \(f_\alpha\) as above.
\end{remark}

\begin{example}
    Following \cite{skrypnyk_new_integrable_gaudin_type_systems,skrypnyk_spin_chains}, let $\fg$ be matrix Lie algebra and \(D\) be a matrix such that \(Y \mapsto DY + YD\) defines an endomorphism of \(\fg\). Then one can define a series $A$ by
    $$
        A(Y) \coloneqq \sqrt{1 + Dx} \, \, Y \, \sqrt{1+Dx}.
    $$
    This endomorphism of $\fg(\!(x)\!)$ 
    satisfies the condition of \cref{eq:normal_form_W_A} with
    $$
    R(Y) = \frac{1}{2}(DY + YD).
    $$
    Therefore, \(W_A\) is an $F[x^{-1}]$-invariant subalgebra of \(\fg(\!(x)\!)\) complementary to \(\fg[\![x]\!]\).
    When \([\textnormal{ad}(h),R] = 0\) 
    for all \(h \in \fh\), 
    the corresponding \(W_A\) is additionally \(\fh\)-invariant. 
    For instance, take \(\fg = \mathfrak{sp}(2n)\) whose Cartan subalgebra is given by diagonal matrices. 
    Then one can take \(D = \textnormal{diag}(a_1, \dots, a_n, a_1,\dots, a_n)\). 
    A similar construction is possible for other classical Lie algebras.

    Note that the series $A$, in general, has more than $3$ non-zero terms and hence do not give rise to an upper-bounded subalgebra of $\fg(\!(x)\!)$.
    It is unclear however if these subalgebras are gauge equivalent to upper-bounded ones.
\end{example}

Combining previous results, we can derive a standard form for regular subalgebras \(W \subseteq L_0\) of any type \(k \ge 0\).
To simplify the presentation of this normal form, we introduce the following subsets of roots
\begin{align*}
   \Delta_\pm^{<m \alpha_i} &\coloneqq \{
   \alpha \in \Delta_\pm \mid \alpha = \pm \sum_{i=1}^n c_i \alpha_i, \ 0 \le c_i < m\}, \\
   \Delta_{\pm}^{\ge m\alpha_i} &\coloneqq \{\alpha \in \Delta_\pm \mid \alpha = \pm \sum_{i=1}^n c_i \alpha_i, \ c_i \ge m \}
\end{align*}
and let
\begin{align}%
\label{eq:l_c_r}
    \fl \coloneqq \bigoplus_{\alpha \in \Delta_+^{< k_i \alpha_i}} \fg_\alpha
    \bigoplus_{\alpha \in \Delta_-^{< \alpha_i}}
    \fg_{\alpha}, \ \
    \fc \coloneqq \bigoplus_{\alpha \in \Delta_+^{\ge k_i \alpha_i}}
    \fg_\alpha \ \text{ and } \ \
    \fr \coloneqq \bigoplus_{\alpha \in \Delta_-^{\ge \alpha_i}}
    \fg_{\alpha}.
\end{align}
Furthermore, we define the following subspaces of $\fg(\!(x)\!)$
\begin{equation*}%
\begin{split}
    V^a &\coloneqq (x^{-1}-a)V[x^{-1}],  \\
    V^{a,b} &\coloneqq x(x^{-1}-a)(x^{-1}-b)V[x^{-1}]
\end{split}    
\end{equation*}
for any subspace \(V \subseteq \fg\) and \(a,b \in F\).

If $\alpha_0 = \sum_{i=1}^n k_i \alpha_i$ is the decomposition of the maximal root into a sum of simple roots, then the parabolic $\mathfrak{P}_i$ (see \cref{eq:maximal_order_i}) can be written in the form
\begin{equation}%
\label{eq:max_order_i_roots}
    \mathfrak{P}_i = (\fh \oplus \fl \oplus x \fc \oplus x^{-1} \fr) [x^{-1}]
\end{equation}
and by \cref{lemm:regular_subalgebra_standard_form} any regular subalgebra $W \subseteq \mathfrak{P}_i$ is necessarily of the form
\begin{equation}%
\label{eq:reg_subalgebra_type_i}
\begin{aligned}
    W = W_\fh  
    \bigoplus_{\alpha \in \Delta_+^{< k_i \alpha_i}} \fg_\alpha^{a_\alpha}
    \bigoplus_{\alpha \in \Delta_-^{< \alpha_i}}
    \fg_\alpha^{a_\alpha}
    \bigoplus_{\alpha \in \Delta_+^{\ge k_i \alpha_i}}
    \fg_\alpha^{c_\alpha, d_\alpha}
    \bigoplus_{\alpha \in \Delta_-^{\ge \alpha_i}}
    \fg_{\alpha}^0.  
\end{aligned}
\end{equation}
In other words, defining a regular subalgebra $W \subseteq \mathfrak{P}_i$ is equivalent to defining constants $a_\alpha, c_\alpha, d_\alpha$ in \cref{eq:reg_subalgebra_type_i} in a consistent way and finding a compatible \(W_\fh \subseteq \fh[x^{-1}]\).

If we write 
$$
    W_\fh = W_\phi \coloneqq \{x^{-n}(x^{-1}h - \phi(h)) \mid h \in \fh,n\in\bZ_{\ge 0}\}
$$
for some linear map \(\phi \colon \fh \to \fh\), the associated \(A = 1 + Rx + Sx^2\) is given by 
\begin{equation*}
        Rv = \begin{cases}
            -\phi(v) & v\in\fh,\\
            -a_\alpha v &
            v\in\fg_\alpha,\alpha \in \Delta_+^{<k_i\alpha_i} \cup \Delta_-^{<\alpha_i}, \\
            -(c_\alpha + d_\alpha)v& v \in \fg_\alpha, \alpha \in \Delta_+^{\ge k_i\alpha_i},\\
            0& v \in \fg_\alpha, \alpha \in \Delta_-^{\ge \alpha_i} 
        \end{cases}\textnormal{ and } \
        Sv = \begin{cases}
            c_\alpha d_\alpha v& v \in \fg_\alpha,\alpha \in \Delta_+^{\ge k_i\alpha_i}\\
            0 & v \in \fh \bigoplus_{\alpha \in \Delta \setminus \Delta_+^{\ge k_i\alpha_i}}\fg_\alpha,
        \end{cases}
    \end{equation*}

According to \cite[Theorem 3.1]{skrypnyk_spin_chains}, the $r$-matrix associated to \(W_A\) is given by 
\begin{equation}
    \begin{split}
        r(x,y) &= \frac{(A(x) \otimes \overline{A}(y))\Omega}{x-y} =
        \frac{((A(x)A(y)^{-1}-1) \otimes 1 + 1 \otimes 1)\Omega}{x-y}
        \\&
        =\frac{\Omega}{x-y} + \left(\frac{A(x)-A(y)}{x-y}A(y)^{-1} \otimes 1\right)\Omega,  
    \end{split}
\end{equation}
where \(\overline{A}\) is uniquely determined by \(\kappa(Aa,\overline{A}b) = \kappa(a,b)\). To obtain the expression above we have used \((1 \otimes \overline{A})\Omega = (A^{-1}\otimes 1)\Omega\).

Therefore, if \(A = 1 + Rx + Sx^2\) we obtain
    \begin{equation*}
        \begin{split}
            r(x,y) = \frac{\Omega}{x-y} + \left((R+ (x+y)S)(1+Ry+Sy^2)^{-1}\otimes 1\right)\Omega.
        \end{split}
    \end{equation*}
    Here we used that
    \begin{equation*}
        (1 + Tx)^{-1} = \sum_{k = 0}^n (-1)^n T^n x^n.
    \end{equation*}
    holds for all linear maps \(T \colon \fg[\![x]\!] \to \fg[\![x]\!]\).

Summarized, the \(r\)-matrix of the subspace \cref{eq:reg_subalgebra_type_i} is given by
\begin{equation}%
\label{eq:rmatrix_type_i}
    \begin{split}
        r(x,y) = \frac{\Omega}{x-y} + \left(\frac{\phi}{y\phi-1} \otimes 1\right)\Omega_\fh &+\sum_{\alpha \in \Delta_+^{< k_i\alpha_i} \cup \Delta_-^{<\alpha_i}}\frac{a_\alpha}{a_\alpha y-1}E_\alpha \otimes E_{-\alpha} \\ &+ \sum_{\alpha \in \Delta_+^{\ge k_i\alpha_i}}\frac{c_\alpha d_\alpha(x+y) - c_\alpha - d_\alpha}{(c_\alpha y-1)(d_\alpha y-1)}E_\alpha \otimes E_{-\alpha}.
    \end{split}
\end{equation}
Note that \(\frac{\phi}{y\phi-1} = \phi (y\phi-1)^{-1} = (y\phi-1)^{-1}\phi\) is unambiguous since \(\phi\) and \(y\phi - 1\) commute.

\subsection{Regular decompositions \(\fg(\!(x)\!) = \fg[\![x]\!]\oplus W\) of type \(0\)}
Let us now discuss regular subalgebras \(\fg(\!(x)\!) = \fg[\![x]\!] \oplus W\) of type \(0\), i.e.\ such that \(W \subseteq \fg[x^{-1}]\). Using \cref{lem:R_S_relations_type_i} with \(S= 0\) yields the following result.

\begin{corollary}%
\label{cor:W_A_nijenhuis}
    The vector space \(W_A = A(x^{-1}\fg[x^{-1}]) \subseteq \fg(\!(x)\!)\) is a subalgebra satisfying \(W_A \subseteq \fg[x^{-1}]\) if and only if \(A = 1 + Rx\) and the so-called Nijenhuis tensor 
    \begin{equation*}%
    %\label{eq:nijenhuis}
        N_R(a,b) \coloneqq R([Ra,b] + [a,Rb]-R[a,b]) - [Ra,Rb]
    \end{equation*}
    of \(R\) vanishes: \(N_R = 0\).
\end{corollary}

The vanishing condition for the Nijenhuis tensor can be expressed using the Jordan decomposition of $R$.

\begin{proposition}
\label{prop:solutions_R_give_filtration_alt}
    Let \(R \colon \fg \to \fg\) be a linear map, \(R = R_s + R_n\) be its Jordan decomposition and \(\lambda_1,\dots,\lambda_m\in F\) be the eigenvalues of the semi-simple linear map \(R_s\) with associated eigenspaces \(\fg_i \coloneqq \textnormal{Ker}(R_s - \lambda_i)\). Then 
    \begin{enumerate}
        \item \(N_{R_s} = 0\) if and only if \([\fg_i,\fg_j] \subseteq \fg_i + \fg_j\);

        \item $N_R = 0$ if and only if \(N_{R_s} = 0\) and 
        \begin{equation}%
        \label{eq:i_j_compatibility}
        \begin{split} &\pi_i\big(N_{R_n}(a,b) - (\lambda_i - \lambda_j)(R_n[a,b]-[R_na,b])\big)= 0,\\&
        \pi_j\big(N_{R_n}(a,b) + (\lambda_i - \lambda_j)(R_n[a,b]-[a,R_nb])\big)= 0
        \end{split}\end{equation}
        holds for all \(a \in \fg_i, b \in \fg_j\) and \(1\le i,j\le m\), where \(\pi_k\colon \fg = \bigoplus_{i = 1}^m\fg_i \to \fg_k\) is the canonical projection.
    \end{enumerate}
\end{proposition}

\begin{proof}
    By direct calculation we see that the equality $N_R(a,b) = 0$ implies 
    \begin{equation}%
    \label{eq:induction_equation_R}
    \begin{aligned}
        (R-\mu)(R-\lambda)[a,b] &= (R-\mu)[(R-\lambda)a,b] + (R-\lambda)[a,(R-\mu)b] \\
        & \quad - [(R-\lambda)a,(R-\mu)b]
    \end{aligned}
    \end{equation}
    for all $\lambda, \mu \in F$. Assume $\lambda$ and $\mu$ are two eigenvalues of $R$ with generalized eigenvectors $v$ and $w$ of ranks $r$ and $t$ respectively.
    We prove now by induction that the equalities $(R-\lambda)^r v = (R-\mu)^t w = 0$ imply the equality
    $$
    (R - \lambda)^r (R-\mu)^t [v,w] = 0.
    $$
    If $v$ and $w$ are eigenvectors, then, using \cref{eq:induction_equation_R}, we get the base case
    $(R-\lambda)(R-\mu)[v,w] = 0$.
    Assume the statement is true for
    $r = 1$ and $1 \le t \le k-1$.
    Then,
    \begin{align*}
        (R-\lambda)(R-\mu)^k[v,w] &=
        (R-\mu)^{k-1}(R-\lambda)(R-\mu)[v,w] \\
        &=  (R-\lambda)(R-\mu)^{k-1}[v,(R-\mu)w] \\
        &= 0,
    \end{align*}
    where the last equality follows from the induction hypothesis and the fact that the vector 
    $(R-\mu)w$ has rank $k-1$.
    Assume now that the statement is true for all $1 \le r \le k-1$ and $t \ge 1$.
    We then have
    \begin{align*}
        (R-\lambda)^{k}(R-\mu)^t[v,w] &=
        (R-\lambda)^{k-1}(R-\mu)^{t-1}(R-\lambda)(R-\mu)[v,w] \\
        &= (R-\lambda)^{k-1}(R-\mu)^{t-1} \left\{ (R-\mu)[(R-\lambda)v,w] + (R-\lambda)[v,(R-\mu)w] \right. \\
        & \hspace{20.5em} 
        \left.  - [(R-\lambda)v,(R-\mu)w] \right\} \\
        &= (R-\lambda)^{k-1}(R-\mu)^t[(R-\lambda)v,w] \\
        & \quad + 
        (R-\lambda)^{k-1}(R-\lambda)(R-\mu)^{t-1}[v,(R-\mu)w] \\
        & \quad +
        (R-\lambda)^{k-1}(R-\mu)^{t-1} [(R-\lambda)v,(R-\mu)w] \\
        &= 0.
    \end{align*}
    This shows that \([\fg_i,\fg_j] \subseteq \fg_i + \fg_j\) for \(\lambda = \lambda_i\) and \(\mu = \lambda_j\), since \(\fg_i = \bigcup_{n = 1}^\infty \textnormal{Ker}((R-\lambda_i)^n)\) and \((R-\lambda_i)|_{\fg_j}\) has trivial kernel for all \(i \neq j\).

    It is now easy to see from \cref{eq:induction_equation_R} that \([\fg_i,\fg_j] \subseteq \fg_i + \fg_j\) implies \(N_{R_s} = 0\). 
    In particular, we see that these two conditions are equivalent. And if one of these equivalent conditions is satisfied, \cref{eq:induction_equation_R} for \(\lambda = \lambda_i\) and \(\mu = \lambda_j\) implies \(N_R(a,b) = N_{R_n}(a,b)\) for \(a,b \in \fg_i\) and
    \begin{equation*}    
    \begin{split}
        N_R(a,b) &= (R-\lambda_i)(R-\lambda_j)[a,b] - (R-\lambda_i)[a,R_nb] - (R-\lambda_j)[R_na,b] + [R_na,R_nb]
    \end{split} 
    \end{equation*}
    for \(a \in \fg_i,b \in \fg_j\).
    The latter can be rewritten as \cref{eq:i_j_compatibility} since 
    \begin{equation*}
        \begin{split}
            (R-\lambda_i)v &= (R_s + R_n-\lambda_i)v = (R_s-\lambda_i) v + R_n v = R_nv + (R_s - \lambda_i)\pi_j(v) \\&= R_nv + (\lambda_j - \lambda_i)\pi_j(v) 
            % =  R_nv - (\lambda_i-\lambda_j)\pi_j(v)
        \end{split}
    \end{equation*}
    and similarly
    \begin{equation*}
        (R-\lambda_j)v = R_nv + (\lambda_i-\lambda_j)\pi_i(v)
    \end{equation*}
    holds for all \(v \in \fg_i \oplus \fg_j\), so
    \begin{equation*}   
    \begin{split}
        (R-\lambda_i)(R-\lambda_j)v &= (R-\lambda_i)(R_nv - (\lambda_i-\lambda_j)\pi_i(v)) \\&= R_n^2v -(\lambda_i-\lambda_j)R_n \pi_j(v) + (\lambda_i-\lambda_j)R_n \pi_i(v).
    \end{split}
    \end{equation*}
\end{proof}

\begin{remark}
    Observe that in particular, the Nijenhuis tensor of \(R_n|_{\fg_i}\) vanishes for all \(1 \le i \le m\). Therefore, the construction of all linear maps satisfying \(N_R = 0\) can be split into two steps: first find a decomposition \(\fg = \bigoplus_{i = 1}^m\fg_i\) satisfying \([\fg_i,\fg_j] \subseteq \fg_i + \fg_j\) and then find nilpotent linear maps \(\{R_{n,i} \colon \fg_i \to \fg_i\}_{i = 1}^m\) satisfying  \cref{eq:i_j_compatibility}. 
\end{remark}

The following corollary in case that \(R_n = 0\) was already noticed in \cite{golubchik_sokolov_one_more_type}.

\begin{corollary}\label{cor:solutions_R_give_filtration_alt}
    Let \(R \colon \fg \to \fg\) be a diagonalizable linear map and \(a_1,\dots,a_m\in F\) be its eigenvalues with associated eigenspaces \(\fg_i \coloneqq \textnormal{Ker}(R - a_i)\). Then $N_R = 0$ if and only if \([\fg_i,\fg_j] \subseteq \fg_i + \fg_j\) for $1 \le i,j \le m$.
\end{corollary}

\begin{remark}%
\label{rem:diagonalizable_R_from_decomposition}
    In view of \cref{cor:solutions_R_give_filtration_alt},
    we can construct a solution $R$ to $N_R = 0$ by first taking a decomposition $\fg = \bigoplus_{i = 1}^m \fg_i$ with the property that $\fg_i$ and $\fg_i \oplus \fg_j$ are subalgebras of $\fg$
    and then defining 
    $$
    R \coloneqq \sum_{i = 1}^m \lambda_i \pi_i,
    $$ 
    where $\pi_i$ is the projection of $\fg$ onto $\fg_i$ and $\lambda_i \in F$ is an arbitrary constant.
    In particular, we can start with a regular decomposition of $\fg$ defined in \cref{sec:regular_decompositions}.
\end{remark}

Let us now specify \cref{prop:solutions_R_give_filtration_alt} for regular partitions.

\begin{theorem}%
\label{prop:W_from_regular_splittings}
Let 
$$
    \Delta = \bigsqcup_{i = 1}^n \Delta_i
$$
be a regular partition,
\(a_1,\dots,a_n\) be some constants in \(F\) 
and \(\phi \colon \fh \to \fh\) be a linear map such that
\begin{equation}%
\label{eq:condition_on_phi}
(\phi-a_i)(\phi-a_j)H_\alpha = 0 \ \text{ for } \alpha \in \Delta_i, \ -\alpha \in \Delta_j.   
\end{equation} 
Then
\begin{equation}%
\label{eq:regular_subalgebra_of_type_0_standard_form}
        W \coloneqq W_\phi \bigoplus_{i = 1}^n
        \bigoplus_{\alpha \in \Delta_i}\fg^{a_i}_{\alpha},    
    \end{equation}
    is a regular subalgebra of type $0$, where
    \begin{equation}%
    \label{eq:W_phi_construction}
        W_\phi \coloneqq \{x^{-n}(x^{-1}h - \phi(h)) \mid h \in \fh,n\in\bZ_{\ge 0}\}
    \end{equation}
    This assignment is a bijection between the set of data \(( \Delta = \bigsqcup_{i = 1}^n \Delta_i, \{a_1,\dots,a_n\}, \phi)\) and regular decompositions 
    \(\fg(\!(x)\!) = \fg[\![x]\!] \oplus W\) of type $0$.

    Furthermore, the \(r\)-matrix of \(W\) in \cref{eq:regular_subalgebra_of_type_0_standard_form} has the form 
    \begin{equation*}
        r(x,y) = \frac{\Omega}{x-y} + \left(\frac{\phi}{y\phi - 1} \otimes 1\right)\Omega_\fh + \sum_{i = 1}^n \frac{a_i\Omega_i}{a_iy-1},
    \end{equation*}
    where \(\Omega_i = \sum_{\alpha \in \Delta_i}E_\alpha \otimes E_{-\alpha}\).
\end{theorem}

\begin{proof}
    Let \(W = W_A\) for \(A = 1 + Rx\),
    \(R = R_s + R_n\) be the Jordan decomposition of $R$ and 
    \(\lambda_1,\dots,\lambda_m\in F\) be the eigenvalues of the semi-simple part \(R_s\) with associated eigenspaces 
    \(\fg_i \coloneqq \textnormal{Ker}(R_s - \lambda_i)\). 
    By virtue of \cref{cor:W_A_nijenhuis}, \(W \subseteq \fg(\!(x)\!)\) is a subalgebra if and only if \(N_R = 0\).

    Let us observe that \(W\) is \(\fh\)-invariant if and only if \([\textnormal{ad}(\fh),R] = 0\). 
    Therefore, for every \(i \in \{1,\dots,m\}\) we get \(\fg_i = \fs_i \bigoplus_{\alpha \in \Delta_i}\fg_\alpha\) for some \(\Delta_i \subseteq \Delta\). 
    The decomposition 
    \(\fg = \bigoplus_{i = 1}^n\fg_i\) is regular if and only if \([\fg_i,\fg_j] \subseteq \fg_{i} + \fg_j\). 
    However, by \cref{prop:solutions_R_give_filtration_alt}
    the latter is equivalent to \(N_{R_s} = 0\), which is a necessary condition for $N_R = 0$.

    Since \(\fg_\alpha\) is one-dimensional for all \(\alpha \in \Delta \cup \{0\}\), we have \(R|_{\fg_\alpha} = R_s|_{\fg_\alpha}\). This implies that \(N_R(a,b) = N_{R_s}(a,b) = 0\) for all \(a\in\fg_\alpha\), \(b \in \fg_\beta\) with \(\alpha,\beta \in \Delta \cup \{0\}\) and \(\alpha + \beta \neq 0\). Here, \(\fg_0 = \fh\) was used. Moreover, \(R(\fh) \subseteq \fh\) and \([\fh,\fh] = 0\) implies \(N_R(a,b) = 0\) for all \(a,b\in \fh\).
    
    Put \(\phi \coloneqq R|_{\fh} \colon \fh \to \fh\) and assume \(\alpha \in \Delta_i\) and \(-\alpha \in \Delta_j\). Then
    \begin{equation*}
        N_R(E_\alpha,E_{-\alpha}) = (-\phi^2 + (a_i + a_j)\phi - a_ia_j)H_\alpha = -(\phi-a_i)(\phi-a_j)H_\alpha
    \end{equation*}
    holds. In particular, we see that \(N_R = 0\) if and only if \cref{eq:condition_on_phi} holds.

    The form of the \(r\)-matrix follows immediately from \cref{eq:rmatrix_type_i}.
\end{proof}

\begin{remark}
    At first glance, it is unclear whether the system of equations \cref{eq:condition_on_phi} is consistent.
    However, by \cref{prop:m_partititons} for any regular partition \(\Delta = \bigsqcup_{i = 1}^n \Delta_i\) we can find a decomposition 
    \(\fh = \bigoplus_{i = 1}^n\fs_i\) 
    such that 
    \(\fg = \bigoplus_{i = 1}^n \fs_i \bigoplus_{\alpha \in \Delta_i}\fg_\alpha\) 
    is a regular decomposition.
    Then we can define \(\phi\) satisfying 
    \cref{eq:condition_on_phi} by letting 
    \(\phi(v) = a_i v\) for \(v \in \fs_i\).
\end{remark}

\begin{remark}%
\label{rem:h_invariant_constant_gcybe}
    Note that solutions $r \in \fg \ot \fg$ to the constant generalized classical Yang-Baxter equation
    \begin{equation}%
    \label{eq:gcybe_for_r}
        [r^{12},r^{13}] + [r^{12},r^{23}] + [r^{32},r^{13}]
        =
        0
    \end{equation}
    are in bijection with subalgebras
    $\mathfrak{w} \subseteq \fg[x^{-1}]/x^{-2}\fg[x^{-1}]$,
    such that $\fg \oplus \mathfrak{w} = \fg[x^{-1}]/x^{-2}\fg[x^{-1}]$.
    Indeed, this follows from the fact that tensors \(r \in \fg \otimes \fg\) satisfying \cref{eq:gcybe_for_r} are in bijection with solutions of the GCYBE of the form \(\frac{\Omega}{x-y} + r\)
    and \cite{abedin_maximov_stolin_quasibialgebras}.
    
    We could also consider the classification of such tensors \(r\) under the assumption of \(\fh\)-invariance. It turns out that this classification is trivial: these tensors are precisely arbitrary tensors in \(\fh \otimes \fh\). Let us briefly explain why. The \(\fh\)-invariance is again equivalent to the \(\fh\)-invariance of $\mathfrak{w}$ and it is easy to see that there exists a linear map \(\phi \colon \fh \to \fh\) such that
    \begin{equation*}
        \mathfrak{w} = \textnormal{span}_F\{x^{-1} h - \phi(h)\mid h \in \fh\} \oplus \sum_{\alpha \in \Delta} (x^{-1}-a_\alpha)\fg_\alpha
    \end{equation*}
    for a set \(\{a_\alpha\}_{\alpha \in \Delta} \subset F\). Assume that some \(a_\alpha \neq 0\). Then 
    \[(x^{-1}-a_\alpha)(x^{-1}-a_{-\alpha})(x^{-1}-a_\alpha)\fg_\alpha \subseteq (x^{-1}-a_\alpha)\fg_\alpha.\]
    This implies that \(2a_\alpha a_{-\alpha} + a_\alpha^2 = \lambda\) and \(a_\alpha^2 a_{-\alpha} = \lambda a_\alpha\) so \(a_\alpha a_{-\alpha} = \lambda\) and \(2\lambda + a_{\alpha}^2 = \lambda\), so \(a_\alpha^2 = -\lambda\). 
    Plugging this back into \(a^2_\alpha a_{-\alpha} = \lambda a_\alpha\) gives \(a_{\alpha} = -a_{-\alpha}\). But this implies that 
    \[[(x^{-1}-a_\alpha)\fg_{\alpha},(x^{-1}+a_{\alpha})\fg_{-\alpha}] = F H_\alpha \subseteq \mathfrak{w}\cap \fg\]
    which is a contradiction. 
    We conclude \(a_\alpha = 0\) for all \(\alpha \in \Delta\). Therefore, every \(\fh\)-invariant subalgebra \(\mathfrak{w} \subseteq \fg[x^{-1}]/x^{-2}\fg[x^{-1}]\) complementary to $\fg$ is of the form
    \begin{equation*}
        \mathfrak{w} = \textnormal{span}_F\{x^{-1} h - \phi(h)\mid h \in \fh\} \oplus x^{-1}(\fn_+ \oplus \fn_-)
    \end{equation*}
    for some linear map \(\phi \colon \fh \to \fh\). 
    The \(r\)-matrix corresponding to \(\mathfrak{w}\) is of the form 
    $$
    r = \phi(h_i) \ot h_i \in \fh \otimes \fh
    $$
for an orthonormal basis \(\{h_i\}_{i = 1}^\ell\) of \(\fh\). 
Since there are no restrictions on \(\phi\), $r$ can be an arbitrary tensor in \(\fh \otimes \fh\).
\end{remark} 

\subsection{Regular decompositions \(\fg(\!(x)\!) = \fg[\![x]\!] \oplus W\) of type \(1\)}%
\label{subsec:regular_decomp_0_type_1}
Let us now consider regular splittings \(\fg(\!(x)\!) = \fg[\![x]\!] \oplus W\) of type \( 1\). 
In view of \cref{eq:max_order_i_roots}, we have
\begin{equation}
    W \subseteq 
    \mathfrak{P}_i = (\fh \oplus \fl \oplus x \fc \oplus x^{-1} \fr) [x^{-1}]
\end{equation}
for \(i \in \{1,\dots,n\}\) such that the simple root \(\alpha_i\) has multiplicity \(k_i = 1\).
In this case the inclusion $[\fl, \fl] \subseteq \fh \oplus \fl$ holds. Using this observation we can immediately concoct the following example.
\begin{example}%
\label{ex:type_1_trivial}
    For two different constants $a_1, a_2 \in F^\times$ define
    \begin{equation*}
    \begin{aligned}
        W = (\fh \oplus \fl)^{a_1}
        \oplus  \fc^{a_1,a_2}  \oplus \fr^0. 
    \end{aligned}
    \end{equation*}
    Clearly, it is a regular subalgebra of $\mathfrak{P}_i$ of type $1$.
    By \cref{lem:R_S_relations_type_i}, it must be of the form \( W = (1+Rx+Sx^2)x^{-1}\fg[x^{-1}]\) for some endomorphisms $R$ and $S$ of $\fg$.
    Indeed, take 
    \begin{equation*}
        Rv = \begin{cases}
            -a_1v & v \in 
            \fh \oplus \fl, \\
            -(a_1 + a_2)v& v \in \fc, \\
            0 & v \in \fr
        \end{cases}\textnormal{ and } \
        Sv = \begin{cases}
            a_1a_2v& v \in \fc,\\
            0 & v \in \fh \oplus \fl \oplus \fr.
        \end{cases}
    \end{equation*}
    In particular, by \cref{lem:R_S_relations_type_i}, these endomorphisms solve $N_R = dS$.
    We do not allow $a_1$ or $a_2$ be equal to $0$, because then the type of $W$ becomes $0$.
\end{example}

We can also make an example of a regular subalgebra of type $1$ using a 2-regular partition of $\Delta$.

\begin{example}%
\label{ex:type_1_from_2_regular}
Let $\Delta^{< \alpha_i}$ be the root system one obtains from $\Delta$ by removing all the roots containing $\alpha_i$, with $k_i = 1$.
In general, it is a union of two irreducible root systems.
A partition
$$
\Delta^{< \alpha_i} = \Delta_1 \sqcup \Delta_2
$$ 
into two closed subsets gives another two-parameter example of a type $1$ regular subalgebra:
    \begin{equation*}
    \begin{aligned}
        W =W_\phi 
        \bigoplus_{\alpha \in \Delta_1}
        \fg_\alpha^{a_1} 
        \bigoplus_{\alpha \in \Delta_2} 
        \fg_\alpha^{a_2} 
        \oplus  \fc^{a_1,a_2}  \oplus \fr^0, 
    \end{aligned}
    \end{equation*}
    where $a_1, a_2 \in F^\times$ and $W_\phi$ is given by the same formula \cref{eq:W_phi_construction}, but we require the linear map \(\phi \colon \fh \to \fh\) to satisfy
    \begin{equation*}
    \begin{split}
        (\phi-a_i)(\phi-a_j)H_\alpha = 0 & \textnormal{ for }\alpha \in \Delta_i, \  -\alpha \in \Delta_j, \\
        (\phi - a_1)(\phi-a_2)H_\alpha = 0 &
        \textnormal{ for } \alpha \in \Delta_+^{\ge \alpha_i}.
    \end{split}
    \end{equation*}
    Indeed, this follows from \cref{lem:R_S_relations_type_i} with endomorphisms $R$ and $S$ given by
    \begin{equation*}
        Rv = \begin{cases}
            -a_iv & v \in 
            \bigoplus_{\alpha \in \Delta_i}\fg_\alpha, \\
            -(a_1 + a_2)v& v \in \fc, \\
            0 & v \in \fr, \\
            -\phi(v)& v \in \fh 
        \end{cases}\textnormal{ and } \
        Sv = \begin{cases}
            a_1a_2v& v \in \fc,\\
            0 & v \in \fh \oplus \fl \oplus \fr.
        \end{cases}
    \end{equation*}
    Observe, that by letting $a_1 = a_2$ and $\phi = a_1$ we obtain a specification of \cref{ex:type_1_trivial}.
\end{example}

We now present more sophisticated constructions using $m$-regular decompositions. 
As we will see, we can construct type-1 regular subalgebras with an arbitrary number of parameters $a_i \in F$ for systems of types $A_n, C_n$ and $D_n$.
In all other cases, regular subalgebras of type 1 are necessarily of the form presented in \cref{ex:type_1_from_2_regular}.

\paragraph{Type $A_n$.}
If we want to obtain a regular subalgebra of type $1$ with $m \ge 3$ parameters, it is natural to start with an $m$-regular decomposition of $\fg$ and try extending it.
The only irreducible root systems, admitting $m$-regular partitions with $m \ge 3$, are precisely $A_n$, $n \ge 2$. 
Therefore, the following construction will be the main building block for the latter examples.

Assume $W \subseteq \mathfrak{P}_i$ is a type-$1$ regular subalgebra. Denote by $R_1$ and $R_2$ the subsystems of $\Delta$ generated by 
$\{ \alpha_1, \dots, \alpha_{i-1} \}$ and
$\{ \alpha_{i+1}, \dots, \alpha_{n}\}$.
When $i = 1$ or $n$, we allow $R_1$ or $R_2$ respectively to be the empty set. 
Such a $W$ must have the following form
\begin{equation}%
\label{eq:regular_type_1_An}
\begin{aligned}
    W = W_\fh
    &\overbrace{\bigoplus_{\alpha \in R_1}  \fg_\alpha^{a_\alpha} \bigoplus_{\alpha \in R_2} \fg_\alpha^{b_\alpha}}^{\text{partition of } R_1 \textnormal{ and } R_2}
    \hspace{-0.5em}\bigoplus_{\alpha \in \Delta_+^{\ge  \alpha_i}} \fg_\alpha^{c_\alpha,d_\alpha}
    \bigoplus_{\alpha \in \Delta_-^{\ge \alpha_i}}\fg_{\alpha}^0.  
\end{aligned}
\end{equation}
Note that if we group together the roots in $R_j$ having the same constant $a_\alpha$, we obtain a regular partition of $R_j$ in the sense of \cref{eq:reg_partition_def}.

Define $c \coloneqq c_{\alpha_i}$, $d = d_{\alpha_i}$ and let
$\Lambda \coloneqq \{ a_\alpha \mid \alpha \in \pi \setminus \{ \alpha_i \} \} \cup \{ c,d\}$
be the set of constants in \cref{eq:regular_type_1_An} associated with simple roots.
The following statement says, that these constants determine constants of all other roots. 
\begin{lemma}%
\label{lem:parameters_of_An_simple_roots}
    Let $C \coloneqq \{a_\alpha \mid \alpha \in \Delta_{\pm}^{< \alpha_i}\} \cup \{  c_\beta, d_\beta \mid \beta \in \Delta_+^{\ge \alpha_1} \}$ be the set of all constants in \nolinebreak\cref{eq:regular_type_1_An}.
    Then $C \subseteq \Lambda$.
\end{lemma}
\begin{proof}
    If $\gamma$ is a positive root in $R_1$, then it can be written as $\gamma = \alpha_{k_1} + \alpha_{k_2} + \dots + \alpha_{k_m}$ for some simple roots
    $\alpha_{k_j} \in \{\alpha_1, \dots, \alpha_{i-1} \}$.
    Consequently, the constant $a_\gamma$ lies in the set $\{a_{\alpha_{k_j}} \mid 1 \le j \le m\} \subseteq \Lambda$.
    The same argument shows that $a_\gamma \in \Lambda$ for any positive $\gamma \in R_2$.

    Any root $\gamma \in \Delta_+^{\ge \alpha_i}$ is necessariy of the form 
    $\gamma = \alpha + \alpha_i + \beta$, for some non-negative roots $\alpha \in R_1 \cup \{ 0\}$ and $\beta \in R_2\cup \{ 0\}$.
    Assume first, that $\alpha \neq 0$ and $\beta = 0$.
    Then by commuting 
    $$\fg_{\alpha_i}^{c,d} = x(x^{-1}-c)(x^{-1}-d)\fg_{\alpha_i}[x^{-1}] \ \text{ with } \ 
    \fg_\alpha^{a_\alpha} = (x^{-1}-a_\alpha)\fg_{\alpha}[x^{-1}]$$ 
    we get the equality
    $$
      (x^{-1}-q)(x^{-1} - c_\gamma)(x^{-1} - d_{\gamma})
      =
      (x^{-1} - c)(x^{-1} - d)(x^{-1} - a_{\alpha})
    $$
    for some $q \in F$.
    Consequently $\{ c_\gamma, d_\gamma \} \subseteq \{ a_{\alpha}, c, d \} \subseteq \Lambda$.
    Similarly, for $\alpha = 0, \beta \neq 0$
    we have $c_\gamma, d_\gamma \in \Lambda$.
    Finally, writing an arbitrary root $\gamma \in \Delta_+^{\ge \alpha_i}$
    as $\gamma = \alpha + (\alpha_i + \beta)$
    and using the previous results we get the containments $c_\gamma, d_\gamma \in \Lambda$.

    It now remains to prove, that $a_{-\gamma} \in \Lambda$ for negative $-\gamma \in R_i$.
    Assume $-\alpha_k \in R_1$, then we can write it as 
    $$-\alpha_k = \underbrace{(\alpha_{k+1} + \dots + \alpha_i)}_{\in \Delta_+^{\ge \alpha_i}} + \underbrace{(-\alpha_k - \dots - \alpha_i)}_{\in \Delta_-^{\ge \alpha_i}}.
    $$
    By commuting the corresponding subalgebras of $W$ we get $a_{-\alpha_{k}} \in \Lambda$.
    Similarly, $a_{-\alpha_k} \in \Lambda$ for any $-\alpha_k \in R_2$.
    Since any negative root $-\gamma \in R_i$ is a sum of $-\alpha_k \in R_i$ we get the desired statement.
\end{proof}

\begin{corollary}
    In case of type $A_n$, a regular subalgebra $W \subseteq \mathfrak{P}_i$ can have at most $n+1$ parameters.
    In other words, $|C| \leq |\Lambda| = n+1$.
    \end{corollary}
\noindent The upper bound $n+1$ can always be achieved as it is seen from the following examples.

\begin{example}%
\label{ex:type_1_An_middle}
Fix an integer $n \ge 3$ and a simple root $\alpha_i$ with $1 < i < n$. 
To simplify the description, we introduce the following notation:
\begin{equation}%
\label{eq:definition_beta_sum}
\begin{aligned}
  \beta_k &\coloneqq \alpha_1 + \dots + \alpha_{k}, \ 1 \le k \le n, \\
  \overrightarrow{\beta_k} &\coloneqq \alpha_{i+1} + \dots + \alpha_{i+k}, \ 1 \le k \le n - i, \\
  \overleftarrow{\beta_k} &\coloneqq \alpha_{i-1} + \dots + \alpha_{i-k}, \ 1 \le k \le i-1.
\end{aligned}
\end{equation}
For convenience we also set $\beta_0 = \overrightarrow{\beta_0} = \overleftarrow{\beta_0} = 0$.
Both $R_1$ and $R_2$ are subsystems of type $A$; see \cref{fig:An_diagram_middle_point}.
\begin{figure}[H]
    \centering
    \includegraphics[scale = 0.7]{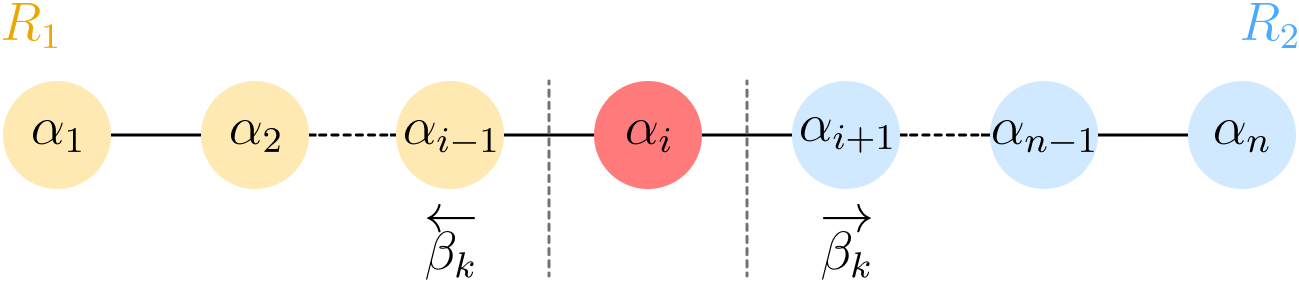}
    \caption{$\overleftarrow{\beta_k}$ and $\overrightarrow{\beta_k}$ sum $k$ simple roots to the left and right from $\alpha_i$ respectively.}
    \label{fig:An_diagram_middle_point}
\end{figure}
\noindent By \cref{thm:root_partitions}, up to equivalences, they have unique finest partitions.
One representative for the finest partition of $R_1$ is
\begin{align*}
  \Delta_1^0 &= \{- \overleftarrow{\beta_k} \mid \ 1 \le k \le i-1 \}, \\
  \Delta_1^{m} &= \{\overleftarrow{\beta_m}, \overleftarrow{\beta_m} - \overleftarrow{\beta_k}  \mid \ 1 \le m \neq k \le i-1 \}.
\end{align*}
Similarly,
\begin{align*}
  \Delta_2^0 &= \{- \overrightarrow{\beta_k} \mid 1 \le k \le n-i \}, \\
  \Delta_2^{j} &= \{\overrightarrow{\beta_j}, \overrightarrow{\beta_j} - \overrightarrow{\beta_k}  \mid 1 \le j \neq k \le n-i \}
\end{align*}
is one of the finest partitions of $R_2$.
Define $H_0 \coloneqq 0$.
For all $0 \le j \le n-i$ and $ 0 \le m \le i-1$ we define the following subalgebras of $\fg$:
\begin{align*} 
    \fg_{1,m} &\coloneqq F H_{\overleftarrow{\beta_m}} \bigoplus_{\alpha \in \Delta_1^{m}} \fg_\alpha, \\
    \fg_{2,j} &\coloneqq F H_{\overrightarrow{\beta_j}} \bigoplus_{\alpha \in \Delta_2^{j}} \fg_\alpha.
\end{align*}
These subalgebras can be now ''glued'' into the following regular subalgebra of $\fg(\!(x)\!)$:
\begin{equation*}
\begin{aligned}
    W = 
    \bigoplus_{m=0}^{i-1} \left( \fg_{1,m}^{a_m} \oplus \fg^{a_m,b_0}_{\overleftarrow{\beta_m} + \alpha_i}
    \right)\bigoplus_{j=0}^{n-i} \left( \fg_{2,j}^{b_j} \oplus \fg_{\alpha_i + \overrightarrow{\beta_j} }^{a_0,b_j}
    \right)\bigoplus_{m=1}^{i-1} \bigoplus_{j=1}^{n-i} \fg_{\overleftarrow{\beta_m} + \alpha_i + \overrightarrow{\beta_j}}^{a_m,b_j}
    \bigoplus_{\alpha \in \Delta_-^{\ge \alpha_i}} \fg_\alpha^0
\end{aligned}
\end{equation*}
where $a_m$ and $b_j$ are distinct elements in $F$.

The corresponding endomorphisms $R$ and $S$ of $\fg$ are given by
\begin{equation}%
\label{eq:endomorphisms_R_S_An_1}
        Rv = \begin{cases}
            -av & (Fv)^{a} \subseteq W, \\
            -(a + b)v & (Fv)^{a,b} \subseteq W, \\
            0 & \text{otherwise}
        \end{cases}
        \ \textnormal{ and } \
        Sv = \begin{cases}
            ab v& (Fv)^{a,b} \subseteq W,\\
            0 & \text{otherwise},
        \end{cases}
\end{equation}
where $v \in \{ H_{\overleftarrow{\beta_m}}, H_{\overrightarrow{\beta_j}}, E_{\pm \alpha} \}$. 
\end{example}

\begin{example}%
\label{ex:type_1_An_edge}
    When $n \ge 2$ and $i=1$ or $n$, we can use the same approach as in \cref{ex:type_1_An_middle}, keeping only $R_2$ or $R_1$ respectively.
    More precisely, when $i=1$ the above construction leads to the following regular subalgebra
    \begin{equation*}
    \begin{aligned}
    W = 
    &\bigoplus_{j=0}^{n-1} \left( \fg_{j}^{a_j} \oplus \fg_{\alpha_i + \overrightarrow{\beta_j} }^{a_j,b}
    \right) \bigoplus_{\alpha \in \Delta_-^{\ge \alpha_i}} \fg^0_\alpha,
\end{aligned}
\end{equation*}
where 
$$
  \Delta_0 = \{- \overrightarrow{\beta_k} \mid 1 \le k \le n-1 \}, \ \
  \Delta_{j} = \{\overrightarrow{\beta_j}, \overrightarrow{\beta_j} - \overrightarrow{\beta_k}  \mid 1 \le j \neq k \le n-1 \}
$$
and
$$
\fg_{j} \coloneqq F H_{\overrightarrow{\beta_j}} \bigoplus_{\alpha \in \Delta_{j}} \fg_\alpha, \ 0 \le j \le n-1
$$
with the convention $H_0 = 0$.

The endomorphisms $R$ and $S$ are given by the same equations \cref{eq:endomorphisms_R_S_An_1}, but with 
$v \in \{ H_{\overrightarrow{\beta}_j}, E_{\pm \alpha} \}$.
\end{example}

\paragraph{Type $B_n$.}
There is only one vertex with degree $1$, namely $\alpha_1$
The remaining simple roots $\{\alpha_2, \dots, \alpha_n \}$ generate a subsystem $R\subseteq \Delta$ of type $B_{n-1}$; see
\cref{fig:Bn_vertex_type_1}.
\begin{figure}[H]
    \centering
    \includegraphics[scale = 0.7]{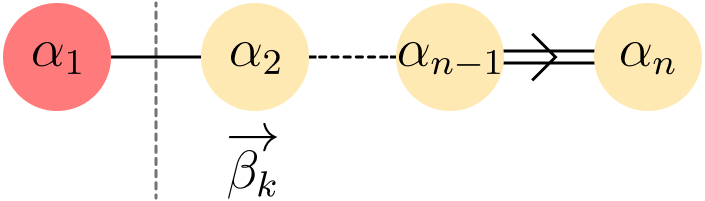}
    \caption{Unique vertex with $k_i = 1$.}
    \label{fig:Bn_vertex_type_1}
\end{figure}
\noindent Similarly to \cref{eq:regular_type_1_An}, any regular subalgebra of type $1$ can then be written in the form 
\begin{equation}%
\label{eq:regular_type_1_Bn}
    W = W_\fh
    \bigoplus_{\alpha \in R}  \fg_\alpha^{a_\alpha}
    \bigoplus_{\alpha \in \Delta_+^{\ge  \alpha_1}} \fg_\alpha^{c_\alpha,d_\alpha}
    \bigoplus_{\alpha \in \Delta_-^{\ge \alpha_1}}\fg_{\alpha}^0.  
\end{equation}
Since there are no regular partitions of $R$ into more than two parts \cite{maximov_regular}, we can immediately conclude that $|\{a_\alpha \mid \alpha \in R \}| \le 2$.
When there are exactly two different constants, call them $a$ and $b$, we have a regular partition $R = R_1 \sqcup R_2$ of the root system $R$.
Note that a system of type $B_n$ has the following property: for any two $\eta, \gamma \in \Delta_+^{\ge \alpha_1}$ we can find two roots $\mu, \nu \in R \sqcup \{ 0\} $ such that
$\eta = (\gamma + \mu) + \nu$.
Consequently, by fixing a pair $\{ c_\gamma, d_\gamma \}$ of constants for an arbitrary $\gamma \in \Delta_+^{\ge \alpha_1}$, we fix all other constants as well.
From this easy observation follows, that we cannot get more than 4 different parameters in \cref{eq:regular_type_1_Bn}.
The next result implies, that actually the number of parameters is at most 2.

\begin{lemma}%
\label{lem:completion_of_roots_Bn}
    There is a root $\gamma \in \Delta_+^{\ge \alpha_1}$ and two roots $\mu, \nu \in \Delta_-^{\ge \alpha_1}$ such that $\gamma + \mu \in R_1$ and $\gamma + \nu \in R_2$.
\end{lemma}

\begin{proof}
    Assume the statement is false.
    Then the roots  
    $$\beta_n - \beta_{n-1} = \alpha_{n} \  \text{ and } \ \beta_n - (\beta_n + \alpha_n) = - \alpha_n$$ 
    always lie together either in $R_1$ or $R_2$. Here, we used the notation defined in \cref{eq:definition_beta_sum}.
    Without loss of generality let $\pm \alpha_{n} \in R_1$.
    Similarly, all three roots
    \begin{align*}
        -\alpha_n &= \beta_{n-1} - \beta_n, \\
        \alpha_{n-1} &= \beta_{n-1} - \beta_{n-2}, \\
        -\alpha_{n-1} - 2\alpha_n &= \beta_{n-1} - (\beta_n + \alpha_n + \alpha_{n-1})
    \end{align*}
    must be contained in one of the sets $R_1$ or $R_2$. Since $-\alpha_n \in R_1$, the same is true for the other two roots.
    Consequently, $\pm \alpha_{n-1} \in R_1$.
    Continuing in this way, we prove that all $\pm \alpha_i$, $2 \le i \le n$ lie inside $R_1$. This contradicts the fact that we started with a $2$-regular partition $R = R_1 \sqcup R_2$.
\end{proof}

Taking 
$\gamma \in  \Delta_+^{\ge \alpha_1}$, 
satisfying the condition of 
\cref{lem:completion_of_roots_Bn}, 
we get the inclusions 
$$\{ c_\eta, d_\eta \mid \eta \in \Delta_+^{\ge \alpha_1} \} \subseteq \{ c_\gamma, d_\gamma \}\cup \{a,b \} \subseteq \{ a,b\}.$$
Note that none of the constants can be $0$, because this would imply the type of $W$ being $0$.

In case $\{a_\alpha \mid \alpha \in R \} = \{ a \neq 0\}$, we know that $a \in \{c_\gamma, d_\gamma \}$.
If $c_\gamma = a$, then the remaining constant $d_\gamma \neq 0$ can be chosen arbitrarily.
Summarizing everything above, we get the following statement.

\begin{proposition}%
\label{prop:Bn_type_1}
    Let $\fg$ be of type $B_n$, $n \ge 2$.
    Then any regular subalgebra $W \subseteq \mathfrak{P}_1$ is necessarily of the form
    presented in \cref{ex:type_1_from_2_regular}, where $\Delta^{< \alpha_1}$ is a root system of type $B_{n-1}$ and $\Delta^{< \alpha_1} = \Delta_1 \sqcup \Delta_2$ is its $2$-regular partition.
\end{proposition}

\paragraph{Type $C_n$.}
As one can see from \cref{fig:Cn_type_1}, removing the vertex of degree one from the Dynkin diagram of type $C_n$ leads to a subdiagram of type $A_{n-1}$.
\begin{figure}[H]
    \centering
    \includegraphics[scale = 0.7]{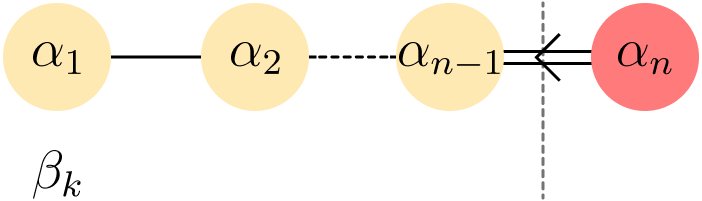}
    \caption{Subdiagram $A_{n-1}$ inside $C_n$. Here $\beta_k$ is the sum of $k$ consecutive simple roots starting from $\alpha_1$}
    \label{fig:Cn_type_1}
\end{figure}
\noindent 
This suggests that we can expect more than just $2$ parameters.
The following example shows that we can obtain a regular subalgebra $W \subseteq \mathfrak{P}_n$ with $n$ parameters.

\begin{example}%
\label{ex:Cn_decomposition}
    Start with one of the finest regular decompositions of $A_{n-1}$, namely
    \begin{align*}
        \Delta_0 \coloneqq \{ \beta_{i} \mid 1 \le i \le n-1 \} \ \text{ and } \ 
        \Delta_{j} \coloneqq \{ -\beta_j, -\beta_j + \beta_i \mid 1 \le i \neq j \le n-1 \},
    \end{align*}
    where $\beta_k \coloneqq \alpha_1 + \dots + \alpha_k$ and $\beta_0 = 0$.
    As before, we put $H_0 = 0$ and for $0 \le j \le n-1$ define
    $$
      \fg_{j} \coloneqq F H_{\beta_j} \bigoplus_{\alpha \in \Delta_{j}} \fg_\alpha.
    $$
    We can extend this regular splitting of $\mathfrak{sl}(n, F)$ to a regular subalgebra $W \subseteq \mathfrak{P}_n$ in the following way
    \begin{equation*}
        W \coloneqq 
        \bigoplus_{i=0}^{n-1} \left( \fg_i^{a_i} 
        \oplus \fg_{\beta_n - \beta_i}^{a_i, a_{n-1}} \right) \bigoplus_{0 \le i \le j < n-1} \fg_{\beta_n + \beta_{n-1} - \beta_j - \beta_i}^{a_i, a_j} 
        % \bigoplus \fg_{\beta_n + \beta_{n-1}}^{a_0, a_0} 
        \bigoplus_{\alpha \in \Delta_-^{\ge \alpha_n}} \fg_{\alpha}^0
    \end{equation*}
    where $a_i \in F$ are arbitrary constants, such that at least one of the pairs $(a_i, a_{n-1}), (a_i, a_j)$ or $(a_0, a_0)$ has both non-zero entries.
    The corresponding endomorphisms $R$ and $S$ are given using the same formula \cref{eq:endomorphisms_R_S_An_1} with $v \in \{H_{\beta_j}, E_{\pm \alpha} \mid 1 \le j \le n, \alpha \in \Delta \}$.
\end{example}

\paragraph{Type $D_n$.}
When $\fg$ is of type $D_n$, we have three different roots of degree $1$ and, essentially, two different cases to consider; see \cref{fig:Dn_splitting_d_1}.
Removing $\alpha_1$ in case $n \ge 5$ leads to a subsystem of type $D_{n-1}$.
On the other hand, removing $\alpha_1$ when $n = 4$ and removing $\alpha_{n-1}$ or $\alpha_n$ gives a subdiagram of type $A_{n-1}$.
\begin{figure}[H]
    \centering
    \includegraphics[scale = 0.7]{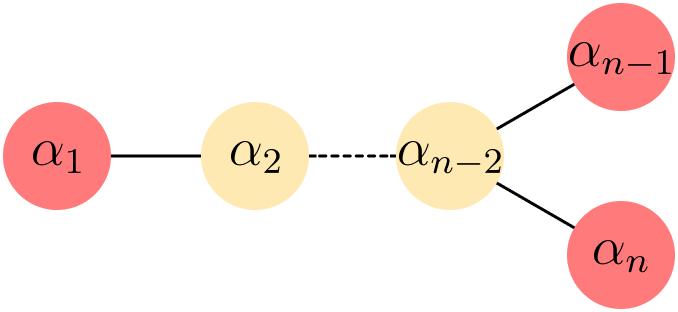}
    \caption{Roots of degree 1 in $D_n$}
    \label{fig:Dn_splitting_d_1}
\end{figure}
\noindent 
Consequently, in the first case we anticipate a rigid system with only two parameters, while in the second case, we can expect to find a regular subalgebra with many parameters.

Let us first consider the case $i=1$ and $n \ge 5$.
It is similar to $B_n$-type case. 
We define $R$ to be the subsystem of $\Delta$ generated by $\{\alpha_2, \dots, \alpha_n \}.$ 
It is a subsystem of type $D_{n-1}$ and hence it has no regular partitions into more than 2 parts.
Therefore, for any regular subalgebra $W \subseteq \mathfrak{P}_1$ we can find two constants $a,b \in F$ and a regular partition $R = R_1 \sqcup R_2$ such that
\begin{equation*}
    W = W_\fh \bigoplus_{\alpha \in R_1} \fg_\alpha^a \bigoplus_{\alpha \in R_2} \fg_\alpha^b 
    \bigoplus_{\alpha \in \Delta_+^{\ge \alpha_1}} \fg_\alpha^{c_\alpha, d_\alpha} 
    \bigoplus_{\alpha \in \Delta_-^{\ge \alpha_1}} \fg_\alpha^0
\end{equation*}
% \begin{align*}
% W = W_\fh &\bigoplus_{\alpha \in R_1} (x^{-1} - a) \fg_{\alpha}[x^{-1}]\bigoplus_{\alpha \in R_2} (x^{-1} - b) \fg_{\alpha}[x^{-1}] \\
% &\bigoplus_{\alpha \in \Delta_+^{\ge  \alpha_1}}
%     x(x^{-1}-c_\alpha)(x^{-1} - d_\alpha) \fg_\alpha[x^{-1}]
%     \bigoplus_{\alpha \in \Delta_-^{\ge \alpha_1}}
%     x^{-1} \fg_{\alpha}[x^{-1}]. 
% \end{align*}
Assuming that $R_1$ and $R_2$ are non-trivial, we can again prove a similar result.
\begin{lemma}%
\label{lem:completion_of_roots_Cn}
    There is a root $\gamma \in \Delta_+^{\ge \alpha_1}$ and two roots $\mu, \nu \in \Delta_-^{\ge \alpha_1}$ such that $\gamma + \mu \in R_1$ and $\gamma + \nu \in R_2$.
\end{lemma}
\begin{proof}
    Assume the opposite.
    Then for any fixed $\gamma \in \Delta_+^{\ge \alpha_1}$ the set
    $$
      O_\gamma \coloneqq \{\gamma + \mu \in \Delta \mid \mu \in \Delta_{-}^{\ge \alpha_1}  \} \subseteq R,
    $$
    called the orbit of $\gamma$, must be contained entirely in one of the sets $R_1$ or $R_2$.
    If two orbits have a non-zero intersection $O_\gamma \cap O_\eta \neq \emptyset$, then both $O_\gamma$ and
    $O_\eta$ are contained in the same $R_i$.
    We say that two orbits $O_\gamma$ and
    $O_\eta$ are connected by a root $\beta \in R$, if there are $\mu, \nu \in \Delta_-^{\ge \alpha_1}$ such that $\gamma + \mu = \eta + \nu$.
    Connected orbits, in particular, are contained in the same closed set $R_1$ or $R_2$.
    Consider the following graph
    \begin{enumerate}
        \item vertices are the roots in $\Delta_+^{\ge \alpha_1}$;
        \item there is an edge between $\gamma$ and $\eta$ if there is a simple root $\alpha_i \in R$ such that $\alpha_i$ or $-\alpha_i$ is contained in $O_\gamma \cap O_\eta$. 
        We label such an edge with $\alpha_i$ or $-\alpha_i$ respectively.
    \end{enumerate}
    It is not hard to see that in our particular case the graph has the form presented in \cref{fig:intersection_of_orbits_Dn}.
    \begin{figure}[H]
    \centering
    \includegraphics[scale = 0.7]
    {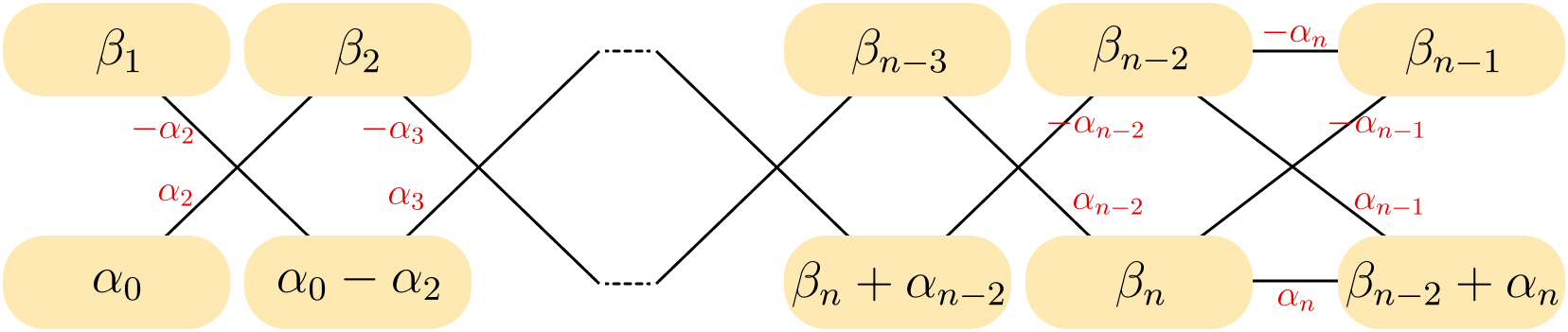}
    \caption{Orbits related by simple roots.}
    \label{fig:intersection_of_orbits_Dn}
\end{figure}
\noindent Consequently, all orbits are connected and they contain the set $\{\pm \alpha_2, \dots, \pm \alpha_n \}$.
This means that either $R \subseteq R_1$ or $R \subseteq R_2$ contradicting our assumption that $R_1$ and $R_2$ were non-trivial.
\end{proof}
Note that, as in the case of $B_n$, any two roots $\gamma, \eta \in \Delta_{+}^{\ge \alpha_1}$ can be connected by a chain of roots $\mu_1, \dots, \mu_\ell \in R$, i.e.\ we can write
$\eta = \gamma + \mu_1 + \dots + \mu_\ell$.
Repeating the argument proceeding \cref{prop:Bn_type_1} we obtain the following statement.
% When some $R_i = \emptyset$, we can again argue as for the $B_n$ case. 
% Namely, any two roots $\gamma, \eta \in \Delta_{+}^{\ge \alpha_1}$ can be connected by a chain of roots $\mu_1, \dots, \mu_\ell \in R$, i.e.\ we can always write
% $\eta = \gamma + \mu_1 + \dots + \mu_\ell$.
% This implies that 
% Repeating the argument for type $B_n$ we get the following result.
\begin{proposition}%
\label{prop:Dn_type_1}
    Let $\fg$ be of type $D_n$, $n \ge 5$.
    Then any regular subalgebra $W \subseteq \mathfrak{P}_1$ is necessarily of the form
    given in \cref{ex:type_1_from_2_regular}.
    The system $\Delta^{< \alpha_1}$ is a system of type $D_{n-1}$ and 
    $\Delta^{< \alpha_1} = \Delta_1 \sqcup \Delta_2$ is its regular partition.
\end{proposition}

Now we turn our attention to the case $i = n$.
As it was mentioned before, the subsystem $R \subseteq \Delta$ generated by $\{\alpha_1, \dots, \alpha_{n-1} \}$ has type $A_{n-1}$.
We now present a construction of a regular subalgebra $W \subseteq \mathfrak{P}_n$ with $n$ parameters.

\begin{example}
We start exactly as in \cref{ex:Cn_decomposition} with the finest partition of $R$:
\begin{align*}
        \Delta_0 \coloneqq \{ \beta_{i} \mid 1 \le i \le n-1 \} \ \text{ and } \ 
        \Delta_{j} \coloneqq \{ -\beta_j, -\beta_j + \beta_i \mid 1 \le i \neq j \le n-1 \},
\end{align*}
where $\beta_k \coloneqq \alpha_1 + \dots + \alpha_k$ and $\beta_0 = 0$.
We put $H_0 = 0$ and define
    $$
      \fg_{j} \coloneqq F H_{\beta_j} \bigoplus_{\alpha \in \Delta_{j}} \fg_\alpha
    $$
for each $0 \le j \le n-1$.
Such a decomposition can be glued to a regular subalgebra of $\mathfrak{P}_n$ in the following way:
\begin{equation*}
    W \coloneqq \bigoplus_{i=0}^{n-1} \fg_i^{a_i} 
    \bigoplus_{i=0}^{n-3} \fg_{\beta_n - \beta_i}^{a_i, a_{n-2}}
    \bigoplus_{i=0}^{n-2} \fg_{\beta_n - \alpha_{n-1} - \beta_i}^{a_i, a_{n-1}}
    \bigoplus_{0 \le i < j < n-3} \fg_{\beta_n + \beta_{n-2} - \beta_i - \beta_j}^{a_i, a_j}
    \bigoplus_{\alpha \in \Delta_-^{\ge \alpha_n}} \fg_{\alpha}^{0},
\end{equation*}
where $a_i \in F$.
\end{example}

The endomorphisms $R$ and $S$ are obtained using \cref{eq:endomorphisms_R_S_An_1} with the restriction $v \in \{H_{\beta_j}, E_{\pm \alpha} \mid 1 \le j \le n, \alpha \in \Delta \}$.

\paragraph{Type $E_6, E_7$.}
There are only two vertices of degree $1$ in $E_6$. Removing any of them leads to a subsystem of type $D_5$; see \cref{fig:E6_E7_type_1}.
\begin{figure}[H]
    \centering
    \includegraphics[scale = 0.7]{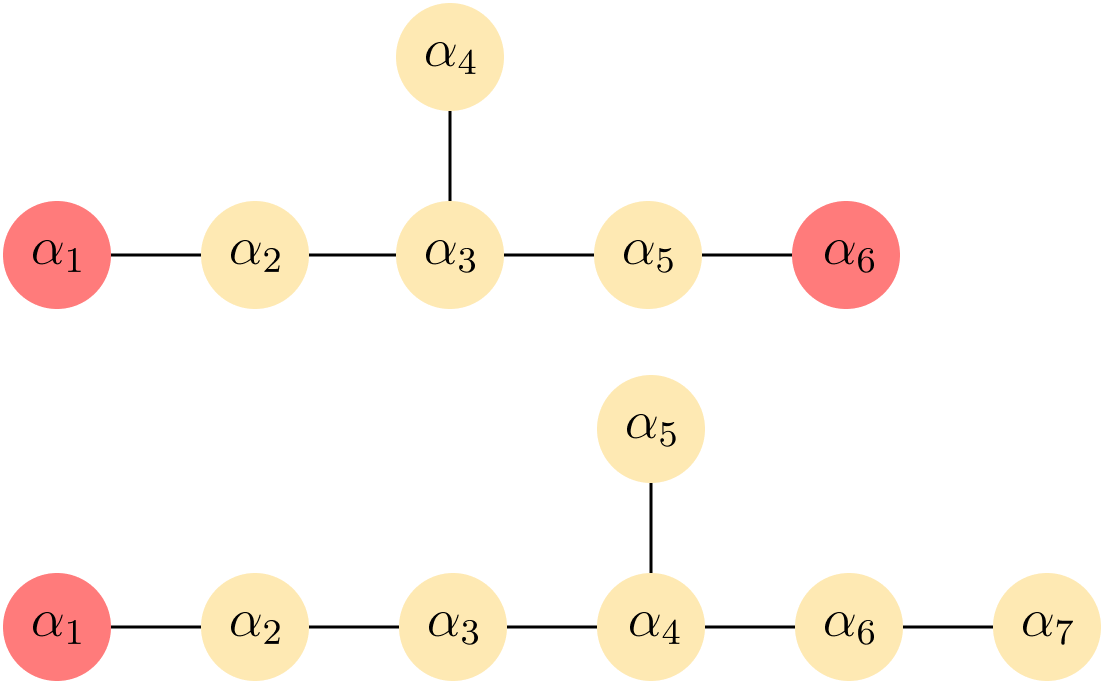}
    \caption{Vertices of degree 1 in $E_6$ and $E_7$.}
    \label{fig:E6_E7_type_1}
\end{figure}
\noindent Furthermore, removing the only node of degree 1 from $E_7$ gives rise to a subsystem of type $E_6$.
In both cases, we can repeat the arguments for $D_n$ and prove the following. 

\begin{proposition}%
\label{prop:E6_E7_type_1}
    Let $\fg$ be of type $E_6$ or $E_7$.
    Then any regular subalgebra $W \subseteq \mathfrak{P}_1$ (and $\mathfrak{P}_6$ in $E_6$ case) is necessarily of the form
    given in \cref{ex:type_1_from_2_regular}.
\end{proposition}
\begin{remark}
    The only difference between proofs for \cref{prop:Dn_type_1,prop:E6_E7_type_1} is that in the latter case we need a more cumbersome calculation to show that all orbits intersect and contain all the roots $\{\pm \alpha_1, \dots, \pm \alpha_n \} \setminus \{\pm \alpha_i \}$.
    However, this can still be done effectively by hand, using the corresponding Hasse diagrams.
\end{remark} 

\subsection{Regular decompositions \(\fg(\!(x)\!) = \fg[\![x]\!] \oplus W\) of type \(> 1\)}%
\label{subsec:reg_0_type_2}
When the degree $k_i$ of a simple root $\alpha_i$ is greater than $1$, we can still use the same two building blocks
\cref{ex:type_1_An_middle,ex:type_1_An_edge}
to construct examples of regular subalgebras $W \subseteq \mathfrak{P}_i$.
The difference lies only in the ``gluing'' complexity. 
We consider $\fg$ of type $B_n$ as an example.

\begin{example}%
\label{ex:B_n_type_2}
    Let $\fg$ be a simple Lie algebra of type $B_n$, $n \ge 3$, and fix some $1 < i < n$.
    By \cref{lemm:regular_subalgebra_standard_form} any regular subalgebra $W \subseteq \mathfrak{P}_i$ is of the form
    \begin{equation*}
    \begin{aligned}
        W =
        W_\fh 
        \bigoplus_{\alpha \in R} \fg_\alpha^{a_\alpha} 
        \bigoplus_{\alpha \in \Delta_+^{\ge k_i \alpha_i}} \fg_{\alpha}^{c_\alpha, d_\alpha}
        \bigoplus_{\alpha \in \Delta_-^{\ge  \alpha_i}} \fg_\alpha^0,
    \end{aligned}
    \end{equation*}
    where $R = \Delta_+^{< k_i \alpha_i} \sqcup \Delta_-^{< \alpha_i}$.
    Note that $R$ is not a subsystem in general, because it may not be closed under root addition.
    However, $R$ does contain \emph{sets of type} $A_n$ and $B_{n-i+1}$; see \cref{fig:Bn_type_2}.
    \begin{figure}[H]
    \centering
    \includegraphics[scale = 0.7]{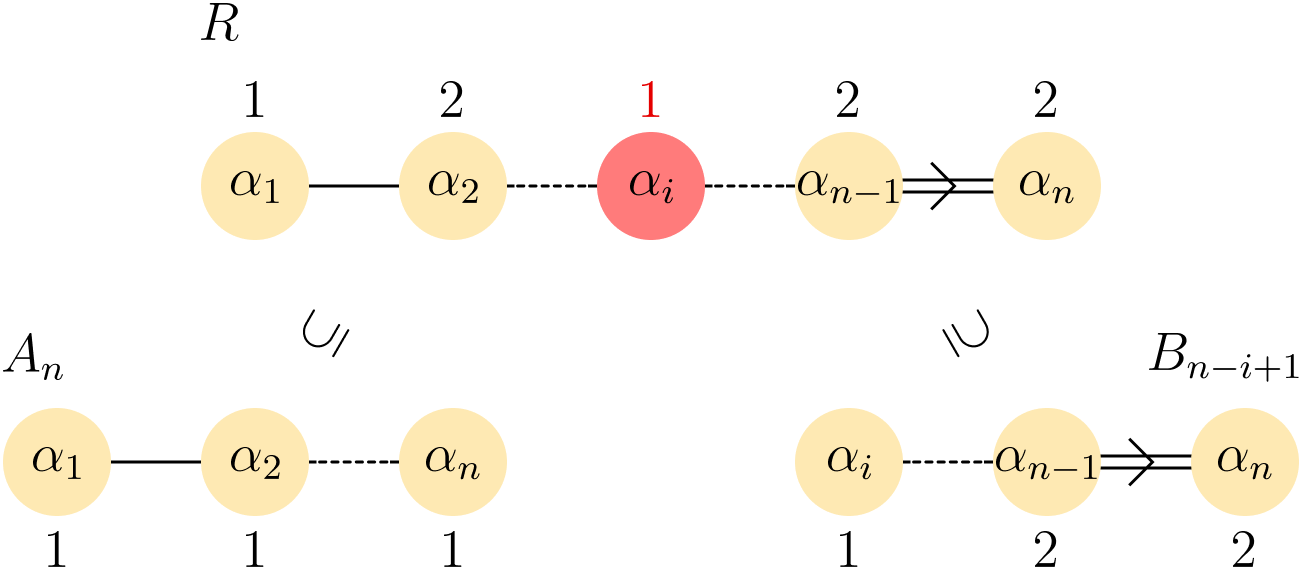}
    \caption{Sets of type $A_n$ and $B_{n-i+1}$ inside $R$.}
    \label{fig:Bn_type_2}
\end{figure}
\noindent By that we mean that $R_1 \coloneqq \{ \beta_k - \beta_j \mid 0 \le k \neq j \le n\}$ is a subset of $R$ and for any two roots $\mu, \nu \in R_1$ we have $\mu + \nu \in R_1$ if and only if $\mu + \nu$ is a root in $A_n$.
Here we again used the notation \cref{eq:definition_beta_sum}.
Similarly, the set 
$$
R_2 \coloneqq \{\beta_k - \beta_j, \pm ( 2\beta_n - \beta_\ell - \beta_p) \mid i \le k \neq j \le n, \ i-1 \le p < \ell \le n-1 \}
$$
has the property $\mu + \nu \in R_2$ for $\mu, \nu \in R_2$ if and only if $\mu + \nu$ is a root of $B_{n-i+1}$.
A regular subalgebra $W$ must induce regular partitions of the subsets $R_1$ and $R_2$ respectively. 
In building an example, we take the opposite path: we decompose the subsets $R_1$ and $R_2$ and then ``glue'' them into a regular subalgebra $W \subseteq \mathfrak{P}_i$.

Start with a finest partition of $R_1$:
\begin{align*}
        \Delta_0' \coloneqq \{\beta_{k} \mid 1 \le k \le n \} \ \text{ and } \ 
        \Delta_{j}' \coloneqq \{ -\beta_j, -\beta_j + \beta_k \mid 1 \le k \neq j \le n \}.
\end{align*}
Now we extend these sets with the remaining roots of $R$ in the following way
\begin{align*}
        \Delta_0 &\coloneqq \Delta_0' \sqcup \{ 2\beta_n - \beta_k \mid i \le k \le n-1 \} \\
        \Delta_j &\coloneqq \Delta_j' \sqcup \{
        2\beta_n - \beta_j - \beta_k \mid i < k \le n-1
        \}, \ 1 \le j \le i-1,
\end{align*}
and
$$
  \Delta_i \coloneqq (\Delta_+^{\ge 2\alpha_n}\cap \Delta_+^{< \alpha_i}) \sqcup(\Delta_-^{\ge 2 \alpha_n}\cap \Delta_-^{< 2\alpha_i})\bigsqcup_{j=i}^{n} \Delta_j'.
$$
In this way, we get a partition $R = \bigsqcup_{k=0}^i \Delta_k$ with the property, that all the roots of $R_2$ lie within two subsets $\Delta_{i-1}$ and $\Delta_i$.
We now add the Cartan part of $\fg$:
\begin{align*}
        \fg_0 &\coloneqq FH_{\beta_n} \bigoplus_{\alpha \in \Delta_0}\fg_\alpha, \\
        \fg_j &\coloneqq FH_{\beta_n-\beta_j} \bigoplus_{\alpha \in \Delta_j} \fg_\alpha, \\
        \fg_i &\coloneqq \bigoplus_{k=i}^{n-1} FH_{\beta_n-\beta_k} \bigoplus_{\alpha \in \Delta_i} \fg_\alpha
\end{align*}
and extend everything to a regular subalgebra $W \subseteq \mathfrak{P}_i$
\begin{equation*}
    W \coloneqq \bigoplus_{k=0}^{i-1} \fg_k^{a_k}
    \bigoplus_{0 \le \ell < m \le i-1} \fg_{2 \beta_n - \beta_m - \beta_\ell}^{a_m, a_\ell}
    \bigoplus \fg_i^0  
    \bigoplus_{\alpha \in \Delta_-^{\ge 2\alpha_i}} \fg_\alpha^0.
\end{equation*}
with $i$ parameters.
The corresponding endomorphisms $R$ and $S$ are given by \cref{eq:endomorphisms_R_S_An_1} with $v \in \{H_{\beta_n - \beta_k}, E_{\pm \alpha} \mid 0 \le k \le n-1, \ \alpha \in \Delta \}$.
\end{example}

Therefore, a general approach to constructing regular subalgebras of type $> 1$ is:
\begin{enumerate}
    \item Find a subset $R_1 \subseteq R$ of type $A_n$ and consider its finest partition $R_1 = \sqcup_{i=0}^{n} \Delta_i'$ (see \cref{thm:root_partitions});
    \item Complete sets $\Delta_i'$ with the remaining roots of $R\setminus R_1$ in such a way that any subset $R_2 \subseteq R$ of a type different from $A_n$ is contained in at most two elements of the partition;
    \item Add Cartan part and extend the resulting decomposition of $\fg$ to a regular subalgebra $W$.
\end{enumerate}

\section{Regular decompositions \(L_m = \mathfrak{D} \oplus W\) for \(m > 0\)}%
\label{sec:regular_1_2}
In view of \cref{rem:restriction_on_k}, to study regular subalgebras $W \subseteq L_m$ it is enough to consider cases  \(m \le 2\).
Lie algebra $L_0$ was considered in the previous section. 
Now we turn our attention to $m=1$ and $m=2$.

It turns out, that regular $W  \subseteq L_1$ of type $0$ admit a complete classification. 
Similar to \cref{prop:W_from_regular_splittings}, the classification is obtained by reducing the problem to decompositions of $\fg \times \fg$.

\subsection{Regular decompositions \(\fg \times \fg = \mathfrak{d} \oplus \mathfrak{w}\)}%
\label{sec:reg_constant_decomp}
We say that $\mathfrak{w} \subseteq \fg \times \fg$
is regular if it is invariant under $\{ (h, h) \mid h \in \fh \}$ and $\mathfrak{w} \oplus \mathfrak{d} = \fg \times \fg $, where $\mathfrak{d} = \{ (a,a) \mid a\in \fg\}$.
Regular subalgebras are completely described in the following proposition.

\begin{theorem}%
\label{prop:gxg_splittings}
Let $\mathfrak{w} \subseteq \fg \times \fg$ be a regular subalgebra.
Then there is a $2$-regular partition $\Delta = S_+\sqcup S_-$ and subspaces of the Cartan subalgebra $\fs_\pm = \ft_\pm \oplus \fr_\pm$, having the properties
$\fh = \fs_+ + \fs_-$ and $\ft_+ \cap \ft_- = \{ 0 \}$, such that
\begin{equation}\label{eq:w_gxg_const} 
\begin{split}
        \mathfrak{w} = \left((\ft_+ \bigoplus_{\alpha \in S_+} \fg_\alpha) \times \{0\}\right) 
        \oplus 
        \left(\{ 0 \} \times (\ft_- \bigoplus_{\beta \in S_-} \fg_\beta) \right)
        \oplus 
        \textnormal{span}_F \{ (h, \phi(h)) \mid h \in \fr_+ \},
        \end{split}
    \end{equation}
    where $\phi \colon \fr_+ \to \fr_-$ is a vector space isomorphism with no non-zero fixed points.
    In other words, all regular subalgebras are described by $2$-regular partitions of $\Delta$ and the extra datum $(\ft_\pm, \fr_\pm, \phi)$.
\end{theorem}
\begin{proof}
Consider a regular subalgebra $\mathfrak{w} \subseteq \fg \times \fg$.
Write $p_\pm$ for the canonical projections of $\mathfrak{w}$ onto the first and second components of $\fg \times \fg$ respectively.
Define the following spaces:
$$
\mathfrak{w}_\pm \coloneqq p_\pm (\mathfrak{w}), \ \
I_{+} \coloneqq \mathfrak{w} \cap (\fg \times \{ 0 \}) \ \text{ and } \ I_- \coloneqq \mathfrak{w} \cap (\{ 0\}\times \fg).
$$
It is clear that $I_\pm$ is an ideal in $\mathfrak{w}_\pm$. 
This allows to define an isomorphism $\varphi$ of Lie algebras:
\begin{equation}%
\label{eq:isomorphism_quotionts}
\begin{aligned}
    \varphi \colon &\mathfrak{w}_+ / I_+ \to \mathfrak{w}_- / I_-  , \\
    &[a_+] \mapsto [a_-] \coloneqq [(p_- \circ p^{-1}_+)(a_+)].
\end{aligned}
\end{equation}
\noindent
The invariance of $\mathfrak{w}$ under the adjoint action of $\{ (h,h) \mid h \in \fh \}$ implies the invariance of $\mathfrak{w}_\pm$ under the adjoint action of $\fh$.
Indeed, applying projections $p_\pm$ to the element
$$
[(h,h), (a_+, a_-)] = ([h,a_+], [h, a_-]) \in W
$$
we see that $[h, a_\pm] \in \mathfrak{w}_\pm$ for all $h \in \fh$ and all $(a_+, a_-) \in W$, giving the desired invariance.
Moreover, since $[\fh, I_\pm] \subseteq I_\pm$, the isomorphism $\varphi$ is an $\fh$-module isomorphism.
Consequently, we can write
$$
    \mathfrak{w}_\pm = \fs_\pm \bigoplus_{\alpha \in S_\pm} \fg_\alpha \ \text{ and } \ 
    I_\pm = \ft_\pm \bigoplus_{\alpha \in R_\pm} \fg_\alpha
$$
for some subspaces $\ft_\pm \subseteq \fs_\pm \subseteq \fh$ and subsets $R_\pm \subseteq S_\pm \subseteq \Delta$.
Note that since $I_+ \cap I_- = \{ 0\}$ we have $\ft_+ \cap \ft_- = \{0 \}$ and $R_+ \cap R_- = \emptyset$.
The isomorphism $\varphi$ then means 
$$
\mathfrak{w}_+ / I_+ \cong \fs_+ / \ft_+ \bigoplus_{\alpha \in S_+ \setminus R_+} \fg_\alpha \ \cong \ \fs_- / \ft_- \bigoplus_{\alpha \in S_- \setminus R_-} \fg_\alpha \cong \mathfrak{w}_- / I_-.
$$
Since $\varphi$ intertwines the $\fh$-action, it must act as multiplication by a non-zero scalar on each non-zero root space and hence $S_+ \setminus R_+ = S_- \setminus R_- \subseteq S_+ \cap S_-$.
Furthermore, since $\mathfrak{w}_+ + \mathfrak{w}_- = \fg$ we get 
$$
\Delta = S_+ \cup S_- = R_+ \cup (S_+ \setminus R_+) \cup R_- \cup (S_- \setminus R_-) = R_+ \cup R_- \cup (S_+ \setminus R_+).
$$
Assume $S_+ \setminus R_+ \neq S_+ \cap S_-$, meaning that there is a root $\alpha \in R_+ \cap (S_+ \cap S_-)$ and $$\alpha \not\in S_+ \setminus R_+ = S_- \setminus R_-.$$
Then we must have the containment $\alpha \in R_-$, leading to a contradiction $R_+ \cap R_- \neq \emptyset$.
Therefore, $S_+ \setminus R_+ = S_+ \cap S_-$ and $R_\pm = S_\pm \setminus S_\mp$.

Let us write $\fs_\pm = \ft_\pm \oplus \fr_\pm$ for some vector subspaces $\mathfrak{r}_\pm \subseteq \fs_\pm$.
Combining the results above, we can write 
\begin{align*}
\mathfrak{w} = &\left((\ft_+ \bigoplus_{\alpha \in \Delta_1} \fg_\alpha) \times \{0\}\right) 
        \oplus 
        \left(\{ 0 \} \times (\ft_- \bigoplus_{\beta \in \Delta_2} \fg_\beta) \right)
        \bigoplus_{\gamma \in S_+ \cap S_-} \textnormal{span}_F \{  
(E_\gamma, \lambda_{\gamma} E_\gamma)
\} \\[1em] 
&\oplus 
\textnormal{span}_F \{ (h_1, h_2) \in \fr_+ \times \fr_- \mid \phi([h_1]) = [h_2] \},
\end{align*}
where we put $\Delta_1 \coloneqq S_+ \setminus S_-$ and $\Delta_2 \coloneqq S_- \setminus S_+$.
To complete the proof we need to show that the intersection $S_+ \cap S_-$ is empty.
For that assume the opposite and take $\gamma \in S_+ \cap S_-$.
Then either $-\gamma \in S_+ \cap S_-$ or $-\gamma \not \in S_+ \cap S_-$.
In the first case the following elements must be contained in $\mathfrak{w}$: 
\begin{align*}
    &(E_{\pm \gamma}, \lambda_{\pm \gamma} E_{\pm \gamma}), \ \ 
    (H_\gamma, \lambda_\gamma \lambda_{-\gamma} H_\gamma), \\
    &(E_\gamma, \lambda^2_\gamma \lambda_{-\gamma} E_\gamma ), \ \
    ( E_{-\gamma}, \lambda_\gamma \lambda^2_{-\gamma}  E_{-\gamma} ).
\end{align*}
This is possible if and only if $\lambda_\gamma \lambda_{-\gamma} = 1$.
However, this would imply $(H_\gamma, H_\gamma) \in \mathfrak{w} \cap \mathfrak{d}$, which is a contradiction.
Therefore, we can without loss of generality assume $-\gamma \in S_+ \setminus S_-$.
Then, by taking appropriate commutators inside $\mathfrak{w}$, we see that the following elements are also in $\mathfrak{w}$:
\begin{align*}
    &(E_\gamma, \lambda_\gamma E_\gamma), \ \
    (E_{-\gamma}, 0), \\
    &(H_{\gamma}, 0), \ \ (\gamma(H_\gamma) E_\gamma, 0),
\end{align*}
giving rise to the contradiction $\gamma \in S_+ \setminus S_-$.
We can now conclude that such a $\gamma$ does not exist implying $S_+ \cap S_- = \emptyset$.

The isomorphism $\varphi$ cannot have fixed points, because otherwise we get a non-trivial intersection of $\mathfrak{w}$ with the diagonal. 
Let $\phi \colon \fr_+ \to \fr_-$ be the isomorphism that maps $h_1 \in \fr_+$ to a representative of \(\varphi([h_1]) \in \fs_-/\ft_-\) inside $\fr_-$.
This completes the proof.
\end{proof}

\begin{example}
    One straight-forward example of a regular subalgebra of $\fg \times \fg$ is
    $$
    \mathfrak{w} = (\fn_+ \times \{0\}) \oplus (\{ 0 \} \times \fn_-) \oplus \{ (h, \phi(h)) \mid h \in \fh \},
    $$
    where $\phi$ is defined on simple roots by $\phi(H_{\alpha_i}) = \lambda_{\alpha_i} H_{\alpha_i}$ for $1 \neq \lambda_{\alpha_i} \in F^\times $.
\end{example}

\begin{example}
    The example above can be easily generalized to an arbitrary $2$-regular partition.
    More precisely, if 
    $\Delta = \Delta_1 \sqcup \Delta_2$, then
    we can define
    $$
    \fg_i \coloneqq \textnormal{span}_F \{E_\alpha, H_\beta \mid \alpha \in \Delta_i, \beta \in \Delta_i \cap (-\Delta_i) \}.
    $$
    It is not hard to see that $\fh \not\subset \fg_1 \oplus \fg_2$.
    In other words, there is always a missing Cartan part $\fh'$.
    This allows to define
    $$
    \mathfrak{w} = (\fg_1 \times \{0\}) \oplus (\{0\} \times \fg_2) \oplus \{ (h, \phi(h)) \mid h \in \fh' \},
    $$
    where $\phi$ is any linear isomorphism without non-zero fixed points.
\end{example}

\begin{remark}
Let us make a connection of the proposition above to the GCYBE similar to \cref{rem:h_invariant_constant_gcybe}.
It is not hard to see that any subspace $\mathfrak{w} \subseteq \fg \times \fg$ complementary to $\mathfrak{d}$
can be written in the form
$$
\mathfrak{w} = \{ (Ra + a, Ra - a) \mid a \in \fg \},
$$
for some unique endomorphism $R \in \textnormal{End}_{F}(\fg)$.
The property of $\mathfrak{w}$ being a subalgebra is equivalent to the relation
\begin{equation}%
\label{eq:gcybe_for_R}
    [Ra, Rb] - R([Ra,b] + [a, Rb]) = -[a,b] \ \text{ for all } a,b \in \fg.
\end{equation}
The Killing form $\kf$ on $\fg$ allows us to identify $\fg \ot \fg$ with $\textnormal{End}_{F}(\fg)$, by sending
$x \ot y$ to $\kf(-, y)x$.
Under this identification, the \cref{eq:gcybe_for_R} becomes
\begin{equation}%
\label{eq:mgcybe_for_r}
    [r^{12},r^{13}] + [r^{12},r^{23}] + [r^{32},r^{13}] = -[\Omega^{12}, \Omega^{13}],
\end{equation}
where $\Omega \in \fg \ot \fg$ is the quadratic Casimir element of $\fg$.

By \cref{prop:gxg_splittings}, all $\fh$-invariant solutions to \cref{eq:mgcybe_for_r} are classified by $2$-regular partitions and additional datum $(S_\pm, \ft_\pm, \fr_\pm, \phi)$. 
Explicitly, consider the tensor  
\begin{equation*}
    \begin{split}
        \widetilde{r} &= \sum_{\alpha \in S_+} (E_\alpha,0) \otimes (E_{-\alpha},E_{-\alpha}) - \sum_{\alpha \in S_-}(0,E_{\alpha}) \otimes (E_{-\alpha},E_{-\alpha}) \\&+ \sum_{i = 1}^n (\phi(h_i) + h_i,\phi(h_i)-h_i) \otimes (h_i,h_i)
    \end{split}
\end{equation*}
where \(\{h_i\}_{i = 1}^n\subseteq \fh\) is an orthonormal basis and \(\phi \colon \fh \to \fh\) is the unique extension of \(\phi\) satisfying \(\phi(\ft_+ \oplus \ft_-) = \{0\}\). 
The element $\widetilde{r}$  
corresponds to the map \((a,a) \mapsto (Ra+a,Ra-a)\) and hence the solution of \cref{eq:mgcybe_for_r} associated with $R$ is given by 
\begin{equation}%
\label{eq:rmatrix_from_constant_gxg}
    r_{(S_\pm, \ft_\pm, \fr_\pm, \phi)} \coloneqq \frac{1}{2}(r_+ + r_-) = \frac{1}{2}\left(\sum_{\alpha \in S_+}E_\alpha \otimes E_{-\alpha} - \sum_{\alpha \in S_-}E_\alpha \otimes E_{-\alpha}\right) + \sum_{i = 1}^n \phi(h_i) \otimes h_i,    
\end{equation}
where \(p_\pm \colon \fg \times \fg \to \fg\), \((a_+,a_-) \mapsto a_\pm\) are the canonical projections and
\(r_\pm \coloneqq (p_\pm \otimes p_\pm)\widetilde{r}\).
\end{remark}
 
\subsection{Regular decomposition \(L_m = \mathfrak{D} \oplus W\) with \(m > 0\) of type \(0\)}%
\label{subsec:L_1_reduction_to_gxg}
Let us start with a regular decomposition
$$
L_2= \fg(\!(x)\!) \times \fg[x]/x^2\fg[x] = \mathfrak{D} \oplus W.
$$
By definition, the projection $W_+$ of $W$ onto the left component is contained in a maximal order $\mathfrak{P}$; see \cref{eq:max_order_i_roots}.
When \(W\) is of type 0, we have $W_+ \subseteq \mathfrak{P}_0 = \fg[x^{-1}]$ and, consequently, 
$$
0 \times x \fg \subseteq W.
$$
Therefore, we can quotient $W$ by the ideal
$0 \times x \fg$ and reduce the problem to a regular decomposition of $L_1$ of type 0.

Consider now a regular subalgebra $W \subseteq L_1$ of type $0$.
This means that $W \subseteq \fg[x^{-1}] \times \fg$.
Intersecting it with $\fg \times \fg$ we get a regular subalgebra $\mathfrak{w} \subseteq \fg \times \fg$.
The latter subalgebras were classified in \cref{prop:gxg_splittings}.
The following result shows that we can also extend regular subalgebras of $\fg \times \fg$ back to type $0$ subalgebras of $L_1$.

\begin{theorem}%
\label{prop:form_W_L_1}
    Let \(\fg(\!(x)\!) \times \fg = \mathfrak{D} \oplus W\) be a regular decomposition such that \(W \subseteq \fg[x^{-1}] \times \fg\), i.e.\ $W$ has type $0$. 
    Then
    \begin{equation*}
        W = \mathfrak{w} \oplus \left(\left(W_{\psi}
        \bigoplus_{i = 0}^n\bigoplus_{\alpha \in \Delta_i}\fg_\alpha^{a_i}\right)\times \{0\}\right),
    \end{equation*}
    where
    \begin{itemize}
        \item \(\mathfrak{w} \subseteq \fg \times \fg\) is given by \cref{eq:w_gxg_const} with a 2-splitting \(\Delta = S_+ \sqcup S_-\) and the additional datum \((\ft_\pm,\fr_\pm,\phi)\);

        \item \(\Delta = \sqcup_{i = 0}^n \Delta_i\) is a regular decomposition 
        such that \(S_+ \subseteq \Delta_0\) and
        $S_+ + \Delta_i$ are closed;
        
        \item \(a_0,\dots,a_n \in F\) are distinct constants with \(a_0 = 0\);

        \item \(\psi \colon \fh \to \fh\) satisfies 
        \(\psi(\fr_+ \oplus \ft_+) = \{0\}\) and 
        \begin{equation*}
            \begin{cases}
                (\psi-a_i)H_\alpha = 0& \alpha \in S_+,-\alpha \in \Delta_i,\\
                (\psi-a_i)(\psi-a_j)H_\alpha = 0 & \alpha \in \Delta_i,-\alpha \in \Delta_j.
            \end{cases}
        \end{equation*}
    \end{itemize}
    Moreover, the above datum always defines a regular subalgebra \(W \subseteq L_1\) of type 0.
    
    The \(r\)-matrix of \(W\) has the form 
    \begin{equation*}
        r(x,y) = \frac{y\Omega}{x-y} + r_{(S_\pm,\ft_\pm,\fr_\pm,\phi)} +\frac{1}{2}\Omega+\left(\frac{\psi}{y\psi - 1} \otimes 1\right)\Omega_\fh + \sum_{i = 1}^n \frac{a_i\Omega_i}{a_iy-1},
    \end{equation*}
    where \(\Omega_i = \sum_{\alpha \in \Delta_i}E_\alpha \otimes E_{-\alpha}\) and \(r_{(S_\pm,\ft_\pm,\fr_\pm,\phi)}\) is given by formula \cref{eq:rmatrix_from_constant_gxg}.
\end{theorem}
\begin{proof}
    Let $W \subseteq L_1$ be a regular subalgebra contained in $\fg[x^{-1}]\times \fg$.
    Define 
    $$\mathfrak{w} \coloneqq W \cap (\fg \times \fg).
    $$
    Regularity of $W$ implies regularity of $\mathfrak{w}$.
    By \cref{prop:gxg_splittings} we can find
    \begin{enumerate}
        \item a regular partition $\Delta = S_+ \sqcup S_-$;
        \item Cartan subspaces $\fs_\pm = \ft_\pm \oplus \fr_\pm$ such that
        $\fh = \fs_+ + \fs_-$ and 
        $\ft_+ \cap \ft_- = \{0\}$;
        \item a linear isomorphism $\phi \colon \fr_+ \to \fr_-$ without non-zero fixed points
    \end{enumerate}
    such that 
    $$
    \mathfrak{w} = \left((\ft_+ \bigoplus_{\alpha \in S_+} \fg_\alpha) \times \{0\}\right) 
        \oplus 
        \left(\{ 0 \} \times (\ft_- \bigoplus_{\beta \in S_-} \fg_\beta) \right)
        \oplus 
        \textnormal{span}_F \{ (h, \phi(h)) \mid h \in \fr_+ \}.
    $$
    By $(x^{-1}, 0)$-invariance of $W$ we have
    $$
    x^{-k}(\fs_+ \bigoplus_{\alpha \in S_+} \fg_\alpha)  \times \{ 0 \} \subseteq W \ \text{ for all } \ k \ge 1. 
    $$
    Furthermore, for each $\alpha \in S_-$
    we can find unique $f \in \fg[\![x]\!]$ and $w \in W$ such that
    $$
    (x^{-1}E_\alpha, 0) = (f, [f]) + w.
    $$
    Since $W$ is of type $0$ we necessarily have the containment $f \in \fg$ forcing
    $(x^{-1}E_\alpha - f, -[f]) \in W$.
    The $\fh$-invariance of $W$ now implies that $f = a_\alpha E_\alpha$. 
    Consequently,
    $$
    \bigoplus_{\alpha \in S_-} (x^{-1} - a_\alpha)\fg_\alpha[x^{-1}]  \times \{ 0 \} \subseteq W.
    $$
    Let $\{ a_\alpha \mid \alpha \in S_-\} = \{a_1, \dots, a_m\}$ and define $\Delta_i \coloneqq \{ \alpha \in S_- \mid a_\alpha = a_i \}.$
    Regularity of $W$ implies that $S_- = \sqcup_i \Delta_i$ is a regular partition of $S_-$.
    Summarizing, $W$ must have the following form
    \begin{align*}
      W &= \left((\ft_+ \bigoplus_{\alpha \in S_+} \fg_\alpha) \times \{0\}\right) 
        \oplus 
        \left(\{ 0 \} \times (\ft_- \bigoplus_{\alpha \in S_-} \fg_\alpha) \right)
        \oplus 
        \textnormal{span}_F \{ (h, \phi(h)) \mid h \in \fr_+ \}  \\
        &\oplus \left( 
        %x^{-1}(\fs_+ 
        \bigoplus_{\alpha \in S_+} \fg_\alpha
        %)
        [x^{-1}]  \times \{ 0 \} \right)
        \oplus 
        \left(\bigoplus_{i=1}^m \bigoplus_{\alpha \in \Delta_i} \fg_\alpha^{a_i} \times \{ 0 \} \right) \oplus W_\psi,
    \end{align*}
    for some $W_\psi \subseteq \fh[x^{-1}] \times \fh$.
    To conclude the regularity of $\Delta = S_+ \sqcup_i \Delta_i$ we need to 
    prove the closedness of $S_+ \sqcup \Delta_i$. 
    For that choose two roots $\alpha \in S_+$ and $\beta \in \Delta_i$ such that 
    $\alpha + \beta \in \Delta_j$.
    This, in particular, would imply the following inclusion
    $$
     [\fg_\alpha \times \{ 0 \}, \fg^{a_i}_\beta \times \{0 \}] = \fg^{a_i}_{\alpha + \beta} \times \{ 0 \} \subseteq \fg^{a_j}_{\alpha + \beta} \times \{ 0 \} \subseteq W,
    $$
    which is possible if and only if $a_i = a_j$ and hence $i = j$.
    Therefore, $\Delta = S_+ \sqcup_i \Delta_i$ is a regular partition of $\Delta$. 
    % Observe that even the stronger condition \((S_+ + \Delta_i)\cap \Delta \subseteq \Delta_i\) holds.

    It now remains to understand the Cartan part $W_\psi = \{x^{-n} (x^{-1}h - \psi(h), 0) \mid h \in \fh\}$.
    The arguments above imply that 
    $$
        x^{-1}\fs_+[x^{-1}] \times \{ 0\} \subseteq W_\fh
    $$
    giving $\psi(\fs_+) = 0$.
    Now take an arbitrary root $\alpha \in \Delta$ and consider three cases:
    \begin{enumerate}
        \item For $\pm \alpha \in S_+$ we must have $H_\alpha \in \fs_+$ and hence $\psi(H_\alpha) = 0$;
        \item When $\alpha \in S_+$ and $-\alpha \in \Delta_i$ the relation
        $[\fg_\alpha,\fg_{-\alpha}^{a_i}] = (x^{-1}-a_i)FH_\alpha$ 
        implies $\psi(H_\alpha) = a_i$;
        \item Finally, when $\alpha \in \Delta_i$ and $-\alpha \in \Delta_j$ the equality
        $$[\fg^{a_i}_\alpha,\fg_{-\alpha}^{a_j}] = (x^{-1} - a_i)(x^{-1}-a_j)FH_\alpha$$
        leads to $(\psi-a_i)(\psi-a_j)H_\alpha = 0$.
    \end{enumerate}
    Retracing the arguments above also shows that \(W\) defined by \(\mathfrak{w}\), \(\psi\) and \(\Delta = \bigsqcup_{i = 1}^n\Delta_i\) as above provides a regular subalgebra of \(L_1\) of type 0. The formula for the \(r\)-matrix is deduced similarly to \cref{eq:rmatrix_type_i}.
\end{proof}

\begin{example}
    The subalgebra \(W = W_+ \times W_-\) with 
    \begin{equation*}
        W_+ = (\fh \oplus \fl \oplus x^{-1}\fr)[x^{-1}] \oplus \bigoplus_{\alpha \in \Delta_+^{\ge k_i\alpha_i}}\fg_\alpha^{a_\alpha}    
    \end{equation*}
    and \(W_- = \fr\) defines a regular decomposition \(L_1 = \Delta \oplus W\) of type \(k_i \ge 0\).
\end{example}

\subsection{Weakly regular decompositions for $L_1$}%
\label{subsec:weakly_regular_decomp}
For subalgebras inside \(L_m\), \(m \ge 1\), there exists another natural notion of $h$-invariance. 
Precisely, instead of assuming the invariance of $W$ with respect to the action of $(h,h)$,
one can require the projections \(W_+\) and $W_-$ onto \(\fg(\!(x)\!)\) and \(\fg[x]/x^m\fg[x]\) respectively to be \(\fh\)-invariant.
For simplicity, we call upper-bounded and $F[x^{-1}]$-invariant subalgebras $W \subseteq L_m$ with the requirement above \emph{weakly regular}.
% , without being necessarily invariant \(W\) under \(\{(h,[h])\mid h \in \fh\}\). 
Using a slight generalization of Belavin-Drinfeld construction \cite{belavin_drinfeld_solutions_of_CYBE_paper}
we can create examples of weakly regular subalgebras $W \subseteq L_1$. 
The construction can be extended to loop algebras to produce weakly regular subalgebras in $L_m$, $m \ge 1$; See \cref{rem:weakly_regular_Lm}.

Let $A = (a_{i,j})_{i,j = 1}^n$ be the Cartan matrix of $\fg$.
We call \((\Gamma_1,\Gamma_2,\tau)\) \emph{a generalized Belavin-Drinfeld triple} if \(\Gamma_1, \Gamma_2 \subseteq \pi\) are subsets of simple roots and \(\tau \colon \Gamma_1 \to \Gamma_2\) is a bijection such that:
    \begin{enumerate}
        \item \(a_{\tau(i), \tau(j)} = a_{i,j}\) for all $\alpha_i, \alpha_j \in \Gamma_1$. In other words, $\tau$ preserves the entries of the Cartan matrix corresponding to the roots in $\Gamma_1$, i.e.\
        $$
            \frac{\langle\alpha_i, \alpha_j\rangle}{\langle\alpha_i, \alpha_i\rangle} = \frac{\langle\tau(\alpha_i), \tau(\alpha_j) \rangle}{\langle \tau(\alpha_i), \tau(\alpha_i) \rangle}; 
        $$

        \item \(\tau(\alpha),\dots,\tau^{k-1}(\alpha) \in \Gamma_2\) but \(\tau^k(\alpha) \notin \Gamma_1\)  holds for every \(\alpha \in \Gamma_1\) and some \(k \in \bN\).
    \end{enumerate}
Then \(\tau\) defines an isomorphism \(\theta_\tau\colon\fg_{\Gamma_1} \to \fg_{\Gamma_2}\) of the Lie subalgebras \(\fg_{\Gamma_1},\fg_{\Gamma_2} \subseteq \fg\) generated by \(\{E_{\alpha},E_{-\alpha}\mid \alpha \in \Gamma_1\}\) and \(\{E_{\alpha},E_{-\alpha}\mid \alpha \in \Gamma_2\}\) respectively. 
Precisely, one defines the isomorphism by  \(E_{\pm \alpha} \mapsto E_{\pm \tau(\alpha)}\).

Let us write $\Delta_{\Gamma_i}$ for the root subsystem of $\Delta$ generated by $\Gamma_i$ and put \(\Delta_{\Gamma_i}^\pm \coloneqq \Delta_{\Gamma_i} \cap \Delta_\pm\). 
We denote the restriction of \(\theta_\tau\)  to 
\(\bigoplus_{\alpha \in \Delta_{\Gamma_1}^\pm}\fg_\alpha\)
with $\theta_\tau^+$ and the restriction of its inverse $\theta_\tau^{-1}$ to \(\bigoplus_{\alpha \in \Delta_{\Gamma_2}^\pm}\fg_\alpha\) 
with $\theta_\tau^-$. 
Both restrictions can be extended by $0$ on the remaining root vectors to produce two endomorphisms
$\theta_\tau^\pm \colon \fn_\pm \to \fn_\pm$, that we denote with the same letters.
Due to condition 2.\ on $\tau$, the endomorphisms $\theta_\tau^\pm$ are nilpotent.
Consequently, 
$$
\rho_\tau^\pm \coloneqq \theta_\tau^\pm/(1-\theta_\tau^\pm) \coloneqq \sum_{k = 1}^\infty \theta_\tau^{\pm, k}\colon \fn_\pm \to \fn_\pm
$$
are well-defined endomorphisms as well.

We say that \(\phi \colon \fh \to \fh\) is compatible with \((\Gamma_1,\Gamma_2,\tau)\) if 
\begin{equation}%
\label{eq:compatible_linear_map}
    \{(h_\alpha,h_{\tau(\alpha)})\mid \alpha \in \Gamma_1\} \subseteq \{(\phi(h)+h,\phi(h)-h)\mid h \in \fh\}.
\end{equation}
%Denote by $\fh_{\Gamma_i}^{\perp}$ the orthogonal complement of the Cartan subalgebra $\fh_{\Gamma_i} \subseteq \fh$ with respect to the Killing form. 
%We say that a linear isomorphism \(\phi \colon \fr_1 \to \fr_2\), for some subspaces $\fr_i \subseteq \fh_{\Gamma_i}^{\perp}$ is compatible with a generalized Belavin-Drinfeld triple \((\Gamma_1,\Gamma_2,\tau)\) if
%\begin{enumerate}
    %\item \(\phi\) has no non-zero fixed points;
    %\item \(\{(h,h)\mid h \in \fh\} \oplus \{(h_\alpha,h_{\tau(\alpha)})\mid \alpha \in \Gamma_1\} \oplus \{(h,\phi(h))\mid h \in \fr_1\} = \fh \times \fh.\)
%\end{enumerate}
Given a generalized Belavin-Drinfeld triple with a compatible linear map \(\phi\), we can construct a weakly regular subalgebra $W \subseteq L_1$ and the associated \(r\)-matrix in a way similar to \cref{prop:form_W_L_1}. 
To simplify the notation, let us define
\begin{equation}
\label{eq:constant_rmatrix_from_BD}
    r_{(\Gamma_1,\Gamma_2,\tau,\phi)} \coloneqq (\phi \otimes 1)\Omega_\fh + 2\left(\sum_{\alpha \in \Delta_+}(\rho^+_\tau-1)(E_\alpha) \otimes E_{-\alpha}-\sum_{\alpha \in \Delta_+}(\rho^-_\tau-1)(E_{-\alpha}) \otimes E_{\alpha}\right). 
\end{equation}
This is a solution to \cref{eq:mgcybe_for_r}.

\begin{proposition}%
\label{prop:weakly_regular_construction}
    For every linear isomorphism \(\phi \colon \fh \to \fh \) compatible with a generalized Belavin-Drinfeld triple \((\Gamma_1,\Gamma_2,\tau)\) we can define 
    \begin{equation*}
        \begin{split}
        \mathfrak{w}_{(\Gamma_1,\Gamma_2,\tau,\phi)} 
            \coloneqq 
            &\left(\bigoplus_{\alpha \in \Delta_+ \setminus \Delta_{\Gamma_1}} \fg_\alpha \times \{0\}\right) \oplus \left(\bigoplus_{\alpha \in \Delta_-\setminus \Delta_{\Gamma_2}}\{0\} \times \fg_\alpha\right)
            \\
            & \ \ \oplus \{(E_\alpha,E_{\tau(\alpha)})\mid \alpha \in \Delta_{\Gamma_1}\} \oplus \{(\phi(h)+h,\phi(h)-h)\mid h \in \fh\}.
        \end{split}
    \end{equation*}
    Then 
    $$
    W \coloneqq \mathfrak{w}_{(\Gamma_1,\Gamma_2,\tau,\phi)} \oplus \left(\left(W_{\psi}
        \bigoplus_{i = 0}^n\bigoplus_{\alpha \in \Delta_i}\fg_\alpha^{a_i}\right)\times \{0\}\right),
    $$
    where
    \begin{itemize}
        \item $\Delta = \sqcup_{i=0}^n \Delta_i$ is a regular partition such that $\Delta_{\Gamma_1} \subseteq \Delta_0$ and $\Delta_{\Gamma_1} + \Delta_i$ are closed;
        \item \(a_0,\dots,a_n \in F\) are distinct constants with \(a_0 = 0\);

        \item \(\psi \colon \fh \to \fh\) is a linear map satisfying 
        \(\psi(\textnormal{Im}(\phi+1)) = \{0\}\) and 
        \begin{equation*}
            \begin{cases}
                (\psi-a_i)H_\alpha = 0& \alpha \in \Delta_{\Gamma_1}, -\alpha \in \Delta_i,\\
                (\psi-a_i)(\psi-a_j)H_\alpha = 0 & \alpha \in \Delta_i,-\alpha \in \Delta_j,
            \end{cases}
        \end{equation*}
    \end{itemize}
    is a weakly regular subalgebra of \(L_1\).

    Furthermore, the associated \(r\)-matrix is given by
    \begin{equation*}
        r(x,y) = \frac{y\Omega}{x-y} + r_{(\Gamma_1,\Gamma_2,\tau,\phi)} +\frac{1}{2}\Omega+\left(\frac{\psi}{y\psi - 1} \otimes 1\right)\Omega_\fh + \sum_{i = 1}^n \frac{a_i\Omega_i}{a_iy-1},
    \end{equation*}
    where \(\Omega_i = \sum_{\alpha \in \Delta_i}E_\alpha \otimes E_{-\alpha}\) and \(r_{(\Gamma_1,\Gamma_2,\tau,\phi)}\) is given by \cref{eq:constant_rmatrix_from_BD}.
\end{proposition}

\begin{remark}%
\label{rem:weakly_regular_Lm}
    It was shown in \cite{KPSST,abedin_maximov_classical_twists} that quasi-trigonometric $r$-matrices, i.e.\ formal solutions of the classical Yang-Baxter equation of the form
    $$
        r(x,y) = \frac{y \Omega}{x-y} + p(x,y)
    $$
    for some $p \in (\fg \ot \fg)[x,y]$ are in bijection with Lagrangian Lie subalgebras of $L_1$ with respect to a certain form.
    Furthermore, in \cite{abedin_maximov_classical_twists} these solutions were classified in terms of triples $(\Gamma_1, \Gamma_2, \tau)$, where 
    \begin{itemize}
        \item $\Gamma_i$ are subsets of simple roots $\Pi$ of the loop algebra $\fg[x,x^{-1}]$;
        \item the minimal root $\alpha_0$ is not in $\Gamma_1$ and
        \item $\tau \colon \Gamma_1 \to \Gamma_2$ is a bijection such that
        $\langle \alpha, \beta \rangle = \langle \tau(\alpha), \tau(\beta) \rangle $  and $\tau^k(\alpha) \not \in \Gamma_1$ for all $\alpha, \beta \in \Gamma_1$ and some $k \in \mathbb{N}$.
    \end{itemize}
    If the map $\tau$ in the triple $(\Gamma_1, \Gamma_2, \tau)$ does not preserve the form, but only the extended affine Cartan matrix, i.e.\ $a_{\tau(i), \tau(j)} = a_{i,j}$ for $\alpha_i \in \Gamma_1$,
    we can use the same construction with minimal adjustments to produce weakly regular subalgebras of $L_1$ of type $0$ or $1$ respectively.

    More precisely, if $\alpha_0 \not \in \Gamma_2$, then we get the generalized triple described above. 
    In case $\alpha_0 \in \Gamma_2$ we can use the same formulas for $r_{(\Gamma_1,\Gamma_2,\tau,\phi)}$ and $\mathfrak{m}_{(\Gamma_1,\Gamma_2,\tau,\phi)}$,
    where we understand $\Delta_\pm$ and $\Delta_{\Gamma_i}$ as subsystems of the affine root system $\Delta$ of $\fg[x,x^{-1}]$.
    The subalgebra $\mathfrak{m}_{(\Gamma_1,\Gamma_2,\tau,\phi)}$ obtained in this way will in general be a subalgebra of $\fg[x,x^{-1}] \times \fg[x]$ complementary to the diagonal embedding of $\fg[x,x^{-1}]$. 
    To obtain a weakly regular subalgebra of $L_1$ it is enough to project $\mathfrak{m}_{(\Gamma_1,\Gamma_2,\tau,\phi)}$ onto $\fg[x,x^{-1}] \times \fg$ and then add $x^{-1}\fg[x^{-1}] \times \{ 0 \}$.
\end{remark}

\begin{example}
    Assume $\Delta$ is of exceptional type $G_2$.
    There are two simple roots $\{\alpha, \beta \}$
    and, without loss of generality, we assume that $\alpha$ is the shortest root.
    Put $\Gamma_1 = \{ \alpha \}$ and $\Gamma_2 = \{ \beta \}$.
    The bijection $\tau$ given by  $\tau(\alpha) = \beta$ trivially satisfies conditions 1.\ and 2.\ above. 
    Consequently, $(\Gamma_1, \Gamma_2, \tau)$ is a generalized Belavin-Drinfeld triple.
    Let $\{ X_{\pm \alpha}, X_{\pm \beta}, H_\alpha, H_\beta \}$ be the Chevalley basis for $\fg(\Delta)$, then
    \begin{align*}
        \fh_{\Gamma_1}^\perp = \textnormal{span}_F \{ \underbrace{3H_\alpha + 2H_\beta}_{\eqcolon H_1} \}, \ \
        \fh_{\Gamma_2}^\perp = \textnormal{span}_F \{ \underbrace{2H_\alpha + H_\beta}_{\eqcolon H_2}  \}.
    \end{align*}
    % Define a map $\phi \colon 'fh  $
    % by $\phi(H_1) = \lambda H_2$ for any $\lambda \in F^\times$.
    % It is clear that $\phi$ is compatible with $(\Gamma_1, \Gamma_2, \tau)$.
    Consequently, by choosing the trivial $\psi$ and a non-zero $\lambda \in F$ we get the following weakly regular subalgebra of $L_1$:
    \begin{align*}
        W \coloneqq & \{ 
            X_{\beta}, X_{\alpha + \beta},
            X_{2\alpha + \beta}, X_{3\alpha + \beta},
            X_{3\alpha + 2\beta}
        \} \times \{ 0 \} \\ 
        & \oplus 
        \{ 0 \} \times
        \{X_{-\alpha}, X_{-\alpha - \beta},
            X_{-2\alpha - \beta}, X_{-3\alpha - \beta},
            X_{-3\alpha - 2\beta}\} \\
        & \oplus 
        \{ (X_\alpha, X_\beta), (H_\alpha, H_\beta), (H_1, \lambda H_2)\} \\
        & \oplus
        x^{-1}\fg[x^{-1}] \times \{ 0 \}. 
    \end{align*}
    We can see that this subalgebra is not regular in our sense, because
    $$
    (H_\alpha, H_\alpha) \cdot (X_\alpha, X_\beta) = (a_{11} X_\alpha, a_{12} X_\beta ) = (2 X_\alpha, -3 X_\beta)
    $$ 
    which is not in $W$.
\end{example}

\begin{example}
    The same construction can be applied to the orthogonal Lie algebra $\mathfrak{o}(5)$. 
    Let $\alpha$ and $\beta$ be two simple roots of $B_2$, such that $\alpha$ is shorter than $\beta$, and $\{ X_{\pm \alpha}, X_{\pm \beta}, H_\alpha, H_\beta \}$ be again the Chevalley basis of $\mathfrak{o}(5)$. 
    Then
    \begin{align*}
        W \coloneqq & \{ 
            X_{\beta}, X_{\alpha + \beta},
            X_{2\alpha + \beta} \} \times \{ 0 \} \\ 
        & \oplus 
        \{ 0 \} \times
        \{X_{-\alpha}, X_{-\alpha - \beta},
            X_{-2\alpha - \beta}\} \\
        & \oplus 
        \{ (X_\alpha, X_\beta), (H_\alpha, H_\beta), (H_\alpha + H_\beta, \lambda (2H_\alpha + H_\beta))\} \\
        & \oplus
        x^{-1}\fg[x^{-1}] \times \{ 0 \}. 
    \end{align*}
    is again a weakly regular subalgebra of $L_1$
    for any non-zero $\lambda \in F$.
\end{example}

\subsection{Regular decomposition \(L_m = \mathfrak{D} \oplus W\) with \(m > 0\) of type \(k > 0\)}%
\label{subsec:regular_L_2_all_types}
The classification of regular subalgebras \(W \subseteq L_m\) of type \(k > 0\) is overly convoluted to be formulated in detail. 
However, we can observe that regular subalgebras of \(L_m\) always admit the following standard form.

\begin{theorem}%
\label{thm:L2_reduction_to_L1}
    Let $L_m = \Delta \oplus W$ be a regular decomposition for \(m > 0\). Then
    \begin{equation}%
    \label{eq:W_standard_form_m>0k>0}
        W = W_\fh \oplus (I_+ \times \{ 0 \}) \oplus (\{ 0 \} \times I_-),
    \end{equation}
    where \(W_\fh \subseteq \fh[x^{-1}] \times \fh[x]/x^m\fh[x]\) is an appropriate subspace and
    \begin{equation*}%
\begin{aligned}
    I_+ \times \{0\} &\coloneqq W \cap (x\fg[x^{-1}] \times \{0 \}), \\
    \{ 0 \} \times I_- &\coloneqq
    W \cap (\{0\} \times \fg[x]/x^{m}\fg[x]).
\end{aligned}
\end{equation*}
\end{theorem}

\begin{proof}
    As it was explained in the beginning of \cref{sec:regular_1_2}, we can restrict our attention to cases $m=1$ and $2$.
    However, for any regular \(W \subseteq L_1\) we can define a regular subalgebra
    $$
    \widetilde{W} \coloneqq W \oplus (\{0\} \times x\fg[x]) \subseteq L_2.
    $$
    If $\widetilde{W}$ has the form \cref{eq:W_standard_form_m>0k>0}, then so does $W$.
    For this reason, it is sufficient to prove the statement for $m=2$.

    Let $W \subseteq L_2$ be a regular subalgebra.
    Define a map $\varphi \colon W_+ / I_+ \to W_- / I_-$ as in 
    \cref{eq:isomorphism_quotionts}.
    It is again a Lie algebra isomorphism having no non-zero fixed points and intertwining the action of the Cartan subalgebra $\fh$.
    Consequently, we can write
    \begin{equation}%
    \label{eq:root_spaces_arbitrary_W_L2}
    \begin{aligned}
        W = 
        W_\fh &\oplus 
        (I_+ \times \{0\}) \oplus
        (\{ 0 \} \times I_-) \\ &\oplus
        \underbrace{\textnormal{span}_F \{(p_\alpha E_\alpha, [q_\alpha] E_\alpha) \mid p_\alpha E_\alpha + I_+ \in W_+ / I_+ \}}_{D \coloneqq}.
    \end{aligned}
    \end{equation}
    \noindent
    Here $W_\fh \subseteq \fh[x^{-1}] \times \fh[x]/x^{2}\fh[x]$, $p_\alpha \in F[x,x^{-1}]$, $q_\alpha \in F[x]$ and the elements in $D$ are glued using $\varphi$, namely $$\varphi(p_\alpha E_\alpha + I_+) = [q_\alpha] E_\alpha + I_-.$$
    The $\fh$-invariance of $W$ implies the $\fh$-invariance of $W_\pm$. 
    Using \cref{lemm:regular_subalgebra_standard_form} we can decompose
    \begin{equation*}
        W_+ = W_{\fh, +} \bigoplus_{\alpha \in \Delta} f_\alpha \fg_{\alpha}[x^{-1}],
    \end{equation*}
    for some non-zero $f_\alpha = a_\alpha x^{-1} + b_\alpha + c_\alpha x$ and $W_{\fh, +} \subseteq \fh[x^{-1}]$.
    Similarly,
    $$
        W_- = W_{\fh,-} \bigoplus_{\alpha \in \Delta} [g_\alpha]E_\alpha,
    $$
    for $g_\alpha \in \{0, 1, x\}$ and
    $W_{\fh,-} \subseteq \fh[x]/x^{2}\fh[x]$.
    The invariance of $W$ under multiplications by $(x^{-1},0)$ and $(0, [x])$ implies that we can choose the representatives $p_\alpha$ and $[q_\alpha]$ in $$(p_\alpha E_\alpha, [q_\alpha]E_\alpha) \in D$$ in a way that (after a re-scaling of $f_\alpha$ and $g_\alpha$ by non-zero constants) we get
    $p_\alpha = f_\alpha$ and $[q_\alpha] = [g_\alpha]$.

    By assumption, $W_+$ is contained in a maximal order $\mathfrak{P}_i$ corresponding to a simple root $\alpha_i$. 
    Decompose 
    $$
    \fg = \fh \oplus \fl \oplus \fc \oplus \fr 
    $$ 
    as described in \cref{eq:l_c_r}.
    Let us consider an arbitrary root vector $E_\alpha \in \fg$.
    Since $W$ is complementary to $\mathfrak{D}$, we can find unique $A_\alpha \in \fg$, $B_\alpha \in \fc$ such that
    $$
        w = (A_\alpha + B_{\alpha}x, A_\alpha + (B_\alpha + E_\alpha)[x]) \in W.
    $$
    For $E_\alpha \not \in \fc$ we have $f_\alpha = a_\alpha x^{-1} + b_\alpha$ and the only way to decompose the element $w$ into a sum \cref{eq:root_spaces_arbitrary_W_L2} is
    $$
    w = (A_\alpha + B_{\alpha}x, A_\alpha + B_{\alpha}[x]) + (0, E_\alpha [x]).
    $$
    As a consequence $A_\alpha = B_\alpha = 0$ and 
    \begin{equation}%
    \label{eq:right_ideal_hrl}
        \{ 0 \} \times (\fh \oplus \fl \oplus \fr)[x] \subseteq W
    \end{equation}
    \noindent
    For an $E_\alpha \in \fc$, the element $w$ above lies in $W$ only if
    $$
        B_\alpha = \lambda_1 E_\alpha \  \text{ and } \  A_\alpha = \lambda_2 E_\alpha,
    $$
for $\lambda_1, \lambda_2 \in F$.
In case $\lambda_2 \neq 0$, we get
$(0, [x])\cdot w = (0, \lambda_2 E_\alpha [x]) \in W$.
If $\lambda_2 \neq 0$ for all $E_\alpha \in \fc$, then $\{ 0\} \times x \fg[x]/x^{2}\fg[x]  \subseteq W$ and we can again reduce everything to $L_1$ and \cref{prop:form_W_L_1}.
Otherwise, there is at least one $E_\alpha \in \fc$ for which $\lambda_2 =0$, $\lambda_1 \neq 0$ and
\begin{equation}%
\label{eq:element_w_for_lambda_2_0}
  (\lambda_1 E_\alpha x, (\lambda_1 + 1) E_\alpha x) \in W.  
\end{equation}

For the negative root $-\alpha \in \fr$ we have 
$$
(x^{-1}E_{-\alpha}, 0) = w' - (g,[g])
$$
for some unique $g \in \fg[\![x]\!]$.
Therefore, 
$$
w' = (x^{-1}E_{-\alpha} + g, g) \in W.
$$
the inclusion $W_+ \subseteq \mathfrak{P}_i$ and the explicit form \cref{eq:max_order_i_roots} of $\mathfrak{P_i}$ implies that $g = 0$ and hence
$$
x^{-1}\fr [x^{-1}] \times \{ 0 \} \subseteq W.
$$
Commuting \cref{eq:element_w_for_lambda_2_0} with $(x^{-1}E_{-\alpha}, 0)$ we obtain the containments:
$
(\lambda_1 H_\alpha, 0)$,  $(\lambda_1^2 x E_\alpha,0) \in \nolinebreak W$.
From this we can conclude that for any $E_\alpha \in \fc$ we have either $$(E_\alpha x, 0) \in W \ \textnormal{ or } \ (0, E_\alpha [x]) \in W.
$$

Take an arbitrary $E_\alpha \in \fg$ and assume
$$
(f_\alpha E_\alpha, [g_\alpha] E_\alpha) \in D.
$$
Then $f_\alpha = a_\alpha x + b_\alpha + c_\alpha x^{-1}$ with $a_\alpha = 0$ for $E_\alpha \not \in \fc$. Moreover, the arguments above show that $[g_\alpha] \in F^\times$.
Indeed, if $E_\alpha \not \in \fc$, then $f_\alpha \in F[x^{-1}]$ by the form of the maximal order and $[g_\alpha] \in F^\times$ due to the inclusion \cref{eq:right_ideal_hrl}.
For $E_\alpha \in \fc$ we have either $(E_\alpha x^{k}, 0) \in W$, $k \le 1$, leading to the contradiction $f_\alpha E_\alpha  \in I_+$ , or 
$(0, E_\alpha [x]) \in W$ and hence $[g_\alpha] \in F^\times$.

Our next step is to prove that the constant $a_\alpha$ vanishes even for $E_\alpha \in \fc$.
Assume $E_\alpha \in \fc$ is such that
$$
(f_\alpha E_\alpha, E_\alpha) = ((ax + b + cx^{-1}) E_\alpha, E_\alpha) \in D
$$
with $a \neq 0$, then by commuting
this element with $(x^{-1}E_{-\alpha}, 0) \in W$ we get
$$
(x^{-1} f_\alpha H_\alpha, 0), (x^{-1} f_\alpha^2 E_\alpha, 0) \in W.
$$
Therefore
$$
((x^{-1}f_\alpha^2 - a f_\alpha)E_\alpha,E_\alpha) \in W,
$$
where the highest power of $x$ in 
$(x^{-1}f_\alpha^2 - af_\alpha)$ is strictly smaller than in $f_\alpha$.
This means that either $$
(x^{-1}f_\alpha^2 - af_\alpha) = 0 \ \text{ or } \ (x^{-1}f_\alpha^2 - af_\alpha) = p f_\alpha $$
for a polynomial $p \in x^{-1}F[x^{-1}]$.
In the first case we immediately have $(0,E_\alpha) \in W$.
In the second case 
$$
(p,0) \cdot (f_\alpha E_\alpha, E_\alpha) = ((x^{-1}f_\alpha^2 - af_\alpha) E_\alpha, 0) \in W
$$
which again forces $(0, E_\alpha) \in W$. 
This contradiction implies that there are no pairs in $D$ of the form
$((ax + b + cx^{-1}) E_\alpha, E_\alpha)$ with $a \neq 0$.
In other words, 
$$
    D \subseteq \fg[x^{-1}] \times \fg.
$$

Let $w \coloneqq ((ax^{-1} - b)E_\alpha, E_\alpha) \in D$ with $a \neq 0$.
Then $(0, E_\alpha [x]) \in W$
and we can find unique $A \in \fg$ and $B \in \fc$ such that
$$
((E_\alpha + A) + Bx, A + B[x]) \in W.
$$
Again, this is possible only if $A = \lambda E_\alpha$ and $B = \mu E_\alpha$ for some $\lambda, \mu \in F$:
$$
((1 + \lambda + \mu x)E_\alpha, (\lambda + \mu [x]) E_\alpha) \in W.
$$
% Subtracting a multiple of $(0, E_\alpha [x]) \in W$ we get
% $$
% ((1 + \lambda + \mu x)E_\alpha, \lambda E_\alpha) \in W.
% $$
Since $f_\alpha = ax^{-1} - b$, we must have $\mu = 0$ and hence
$$
((1 + \lambda) E_\alpha, \lambda E_\alpha) \in W.
$$
For $\lambda = 0$ or $-1$ we get the inclusions $(E_\alpha, 0)$ or $(0, E_\alpha) \in W$. 
Since both inclusions contradict our choice of $w \in D$, we have $\lambda \not\in \{0, -1 \}$
and
$$
(\lambda' E_\alpha, E_\alpha) \in W
$$
for some non-zero $\lambda' \in F$.
subtracting this element from  $w$ we get
$$
((ax^{-1} - b - \lambda') E_\alpha, 0) \in W.
$$
This means that we must be able to find a polynomial $p \in F[x^{-1}]$ such that $(ax^{-1} - b - \lambda') = p (ax^{-1} - b)$.
This is possible only in the case $p = 1$ and $\lambda' = 0$, which is a contradiction.

In other words, we have shown that 
$$
D \subseteq \fg \times \fg.
$$
Now arguing precisely as in the proof of \cref{prop:gxg_splittings} we see that $D = 0$.
\end{proof}

\begin{example}
    The proposition above provides us with a strategy of constructing more examples of regular subalgebras $W \subseteq L_1$ of type $ k \ge 1$.

    Start with a regular subalgebra $V_+ \subseteq \fg(\!(x)\!)$ of type $k \ge 1$, a
    $2$-regular partition $\Delta = S_+ \sqcup S_-$ of the root
    system of $\fg$ and subspaces $\fs_\pm = \ft_\pm \oplus \fr_\pm$ of the Cartan subalgebra $\fh$ subject to the following conditions
    \begin{enumerate}
        \item $\fh = \fs_+ + \fs_-$;
        \item $\ft_+ \cap \ft_- = \{ 0 \}$;
        \item $v_\pm \coloneqq \ft_\pm \bigoplus_{\alpha \in S_\pm} \fg_\alpha$ is a regular subalgebra of $\fg$ and
        \item $[v_+, V_+] \subseteq v_+ \oplus V_+$
    \end{enumerate}
    Then the following subspace is a regular subalgebra of $L_1$ of the same type $k \ge 1$:
    $$
        W \coloneqq W_\phi \oplus \left( (v_+ \oplus V_+) \times \{0\} \right) \oplus \left( \{ 0\} \times v_- \right),
    $$
    where $W_\phi \coloneqq \{ (h, \phi(h)) \mid h \in \fr_+ \}$
    for any linear isomorphism $\phi \colon \fr_+ \to \fr_-$ without non-zero fixed points.
\end{example}

\begin{example}
    Write $\fh = FH_{\alpha_0} \oplus \fh'$, where $\alpha_0$ is the maximal root of $\fg$. 
    Then we can define
    \begin{equation}
    \begin{aligned}
        W_+ &\coloneqq (x^{-1} \fg \oplus \fn_+ \oplus FH_{\alpha_0} \oplus x \fg_{\alpha_0})[x^{-1}] \\ 
        W_- &\coloneqq \fh' \oplus \fn_- \oplus x \left( \fn_- \oplus  \fh \oplus \bigoplus_{\alpha \in \Delta_+ \setminus \{ \alpha_0\}}  \fg_\alpha \right).
    \end{aligned}
    \end{equation}
    The product 
    $W = W_+ \times W_-$ is a regular (bounded) subalgebra of $L_2$.
\end{example}

\begin{example}
    There are also examples of unbounded regular subalgebras of $L_2$. For example, for $\fg = \mathfrak{sl}(3,F)$ we can define the following spaces
    \begin{align*}
        W_+ &\coloneqq \textnormal{span}_F\{ xE_{\alpha+\beta}, x(x^{-1} - a)E_\alpha \} [x^{-1}] \oplus  \textnormal{span}_F\{ E_\beta, H_{\alpha + \beta} \}[x^{-1}] \\ 
        & \oplus \textnormal{span}_F\{(x^{-1} - a)H_\alpha, (x^{-1} - a)E_{-\beta} \} [x^{-1}] \\
        & \oplus \textnormal{span}_F\{ x^{-1}E_{-\alpha}, x^{-1}E_{-\alpha-\beta} \}[x^{-1}],
    \end{align*}
    \begin{align*}
        W_-  \coloneqq \textnormal{span}_F\{ H_{\alpha}, E_{-\alpha}, E_{-\beta}, E_{-\alpha-\beta}, xE_\alpha, xE_\beta, xH_\beta \}[x] / x^2 \fg[x]
    \end{align*}
    and set $W \coloneqq W_+ \times W_-$. Here, \(\pi = \{\alpha,\beta\} \subseteq \Delta\) are the simple roots and \(\alpha + \beta \in \Delta\) is the only non-simple positive root.
    
\end{example}

\section{Connection to Gaudin models}%
\label{sec:gaudin_models}
In this section, \(F = \bC\).
At this point, we have constructed Lie algebra decompositions \(L_m = \mathfrak{D} \oplus W\) with additional compatibility with a Cartan subalgebra \(\fh \subseteq \fg\) and stability under multiplication by \(x^{-1}\).
Moreover, we have given explicit formulas for the associated generalized \(r\)-matrices. It is known that for \(m = 0\) these generalized \(r\)-matrices give rise to Gaudin integrable models; see \cite{skrypnyk_spin_chains}. Moreover, these models are particularly well-behaved if the generalized \(r\)-matrix satisfies the additional compatibility with the Cartan subalgebra; see \cite{skrypnyk_bethe} for \(\fg = \mathfrak{gl}_n(\bC)\).

Let \(r\) be a generalized \(r\)-matrix of the form
\begin{equation}%
\label{eq:rmatrix_general_m}
    r(x,y) = \frac{y^m\Omega}{x-y} + g(x,y)
\end{equation}
which converges to some meromorphic function in some open disc around the origin. Then for any points \(u_1,\dots,u_n\) in the domain of definition of \(r\), we consider
\begin{equation}%
\label{eq:gaudin_hamiltonians}
    H_i \coloneqq \sum_{k \neq i} r(u_k,u_i)^{(ki)} + \frac{1}{2}(g(u_i,u_i)^{(ii)}+\tau(g(u_i,u_i))^{(ii)}) \in U(\fg)^{\otimes n},
\end{equation}
where \((a \otimes b)^{(ij)} \coloneqq \iota_i(a)\iota_j(b)\) for the canonical embedding \(\iota_i \colon U(\fg) \to U(\fg)^{\otimes n}\), which inserts 1 into every tensor factor except the \(i\)-th one where it inserts \(a\).

Elements \(H_i\) commute inside $U(\fg)$ and define a quantum integrable system, which is a generalization of the usual Gaudin models. 
The commutativity of $H_i$'s is proven in \cref{subsec:commutativity_of_H}.
In the general scheme of integrability, \(H_i\) correspond to quadratic invariant functions on \(\fg^{\oplus n}\). 
To obtain integrability of the model, one would need to consider the higher degree invariant functions as well. 

The classical limit of such Gaudin models, called classical Gaudin models, simply replaces \(U(\fg)\) with the symmetric algebra \(S(\fg)\), which can be understood as the space of regular functions on \(\fg^*\). In this way we obtain a classical integrable system. 

\begin{remark}%
\label{rem:special_point}
    Note that in \cref{eq:gaudin_hamiltonians} one of $u_k$'s can be equal $0$. 
    However, such a point is a special point in the sense of \cite{skrypnik_special_points}, because the $r$-matrix becomes degenerate in that point.
    In this case, the corresponding model is not really a Gaudin model, but a ``reduced'' Gaudin model.
\end{remark}

\begin{example}
    Taking \(n = 1\) and \(u = u_1\), we obtain a quantum integrable system defined by one Hamiltonian \(H = \frac{1}{2}m(g(u,u)+\tau(g(u,u)))\), where \(m(a,b) = ab\) is the multiplication map of \(U(\fg)\) (resp.\ \(S(\fg)\) in the classical limit).

    Let us assume \(a_\alpha,b_\alpha,c_\alpha \in \bC\) are chosen in such a way that  \cref{eq:reg_subalgebra_type_i} defines a subalgebra \(W \subseteq \fg(\!(x)\!)\). 
    The corresponding $r$-matrix \cref{eq:rmatrix_type_i} gives rise to the Hamiltonian
    \begin{equation*}
        \begin{split}
        H = \frac{1}{2}&\left(\psi_u(h_i)h_i + h_i\psi_u(h_i) + \sum_{\alpha \in \Delta_+^{<k_i\alpha_i}\cup \Delta_-^{<\alpha_i}}\frac{a_\alpha (E_\alpha E_{-\alpha} + E_{-\alpha}E_\alpha)}{a_\alpha u-1} \right.\\&\left.- \sum_{\alpha \in \Delta_+^{\ge k_i\alpha_i}}\frac{c_\alpha + d_\alpha -2c_\alpha d_\alpha u}{(u-c_\alpha)(u-d_\alpha)}(E_\alpha E_{-\alpha} + E_{-\alpha}E_\alpha)\right),    
        \end{split}
    \end{equation*}
    where \(\psi_u \coloneqq \frac{\phi}{u\phi-1}\).

    Let us explicitly calculate this Hamiltonian for the regular decomposition of 
    $$
    \fg = \mathfrak{sl}_2(\bC)= \textnormal{Span}_\bC\{E,H,F\} = \left(\bC E \oplus \bC H \right) \oplus \mathbb{C}F.
    $$
    For two constants \(a, b \in \bC\) we get
    \begin{align*}
        H &= \frac{a\left(\frac{1}{2}H^2 + EF + FE\right)}{2(au-1)} + \frac{b(EF + FE)}{2(bu-1)} 
        \\&= \frac{a\left(\frac{1}{2}H^2 + (E + F)^2 - (E+F)(E-F)-H\right)}{2(au-1)} + \frac{b((E + F)^2 - (E+F)(E-F)-H)}{2(bu-1)}.
    \end{align*}
\end{example}

\subsection{Commutativity of the generalized Gaudin Hamiltonians}%
\label{subsec:commutativity_of_H}
For completeness, we present a proof that the Hamiltonians \cref{eq:gaudin_hamiltonians} commute. 
This is a coordinate-free rework of the proof from \cite{skrypnyk_spin_chains}, where $m$ is equal to $0$.
Such an approach shows the commutativity of Hamiltonians for any $m \ge 0$ and in all the points, including the special one; see \cref{rem:special_point}.

Fix $1 \le i < j \le n$ and consider
\begin{align*}
    [H_i, H_j] &= \underbrace{\sum_{\substack{k=1 \\ k \neq i}}^n \sum_{\substack{\ell=1 \\ \ell \neq j}}^n [r(u_k, u_i)^{(ki)}, r(u_\ell, u_j)^{(\ell j)}]}_{S_1 \coloneqq} \\
    &+ \underbrace{\frac{1}{2}  [r(u_j, u_i)^{(ji)}, g(u_j, u_j)^{(jj)} + \tau(g(u_j, u_j))^{(jj)}]}_{S_2 \coloneqq} \\
    &+ \underbrace{ \frac{1}{2}  [g(u_i, u_i)^{(ii)} + \tau(g(u_i, u_i))^{(ii)}, r(u_\ell, u_j)^{(i j)}]}_{S_3 \coloneqq} \\
    &+ \underbrace{\frac{1}{4} [g(u_i, u_i)^{(ii)} + \tau(g(u_i, u_i))^{(ii)}, g(u_j, u_j)^{(jj)} + \tau(g(u_j, u_j))^{(jj)}]}_{S_4 \coloneqq}.
\end{align*}
The summand $S_4$ is equal to $0$ because $i \neq j$.
Terms in $S_1$ with $k \neq j, k \neq \ell$ and $\ell \neq i$ are equal $0$. 
The remaining terms are
\begin{equation}
\label{eq:formula_for_S1}
    \begin{split}
        S_1 &= \sum_{\substack{k=1 \\ k \neq i, \, j}}^n [r(u_k, u_i)^{(ki)}, r(u_k, u_j)^{(kj)}] 
    + \sum_{\substack{\ell=1 \\ \ell \neq i, \, j}}^n [r(u_j, u_i)^{(ji)}, r(u_\ell, u_j)^{(\ell j)}] + 
    \sum_{\substack{k=1 \\ k \neq i}}^n [r(u_k, u_i)^{(ki)}, r(u_i, u_j)^{(i j)}] \\
    &= \sum_{\substack{k=1 \\ k \neq i \\ k \neq j}}^n \textnormal{GCYB}^{ijk}(r(u_k, u_i)) + [r(u_j, u_i)^{(ji)}, r(u_i, u_j)^{(ij)}] \\
    &= [r(u_j, u_i)^{(ji)}, r(u_i, u_j)^{(ij)}] \\
&= \frac{1}{2} \left(m_{(0j)}([r(u_j,u_i)^{(0i)},r(u_i,u_j)^{(ij)}] + [r(u_j,u_i)^{(ji)},r(u_i,u_j)^{(i0)}]) \right.\\&\left. \hspace{1cm}  + m_{(0i)}([r(u_j,u_i)^{(ji)},r(u_i,u_j)^{(0j)}] + [r(u_j,u_i)^{(j0)},r(u_i,u_j)^{(ij)}])\right)
\end{split}
\end{equation}
where we added an auxiliary copy of \(U(\fg)\) at tensor factor in position $0$, defined
\begin{align*}
    m_{(0k)}\colon U(\fg)^{\otimes (n+1)} &\longrightarrow U(\fg)^{\otimes n} \\a_0 \otimes a_1 \otimes \dots \otimes a_n &\longmapsto a_1 \otimes \dots \otimes a_{k-1} \otimes a_{0}a_{k} \otimes a_{k} \otimes \dots \otimes a_n
\end{align*} 
and used the identity
\begin{align*}
    &\sum_{k,\ell = 1}^n[(a_k \otimes b_k)^{(ji)},(a_\ell\otimes b_\ell)^{(ij)}] = \sum_{k,\ell = 1}^n([b_k,a_\ell] \otimes a_k b_\ell + a_\ell b_k \otimes [a_k,b_\ell])^{(ij)} \\
&= \frac{1}{2}\sum_{k,\ell = 1}^n ([b_k,a_\ell]   \otimes (a_kb_\ell + b_\ell a_k + [a_k,b_\ell]) 
+ (a_\ell b_k + b_k a_\ell + [a_\ell,b_k])\otimes [a_k,b_\ell])^{(ij)} 
    \\& 
    = \frac{1}{2}\sum_{k,\ell = 1}^n([b_k,a_\ell]\otimes (a_kb_\ell + b_\ell a_k) 
    + (a_\ell b_k + b_k a_\ell)\otimes [a_k,b_\ell])^{(ij)}
    \\&=
    \frac{1}{2}\sum_{k,\ell = 1}^n \left(m_{(0j)}\left((a_k \otimes [b_k,a_\ell] \otimes b_\ell + b_\ell \otimes [b_k,a_\ell] \otimes a_k)^{(0ij)}\right) \right. \\
    &\hspace{2cm}
    \left.+ m_{(0i)}\left((a_\ell \otimes b_k \otimes [a_k,b_\ell]
    + b_k \otimes a_\ell \otimes [a_k,b_\ell])^{(0ij)}\right)\right).
\end{align*}

Now let us again take a look at the GCYBE:
\begin{align*}
    [r(u_0, u_i)^{(0i)}, r(u_0, u_j)^{(0j)}] + 
    [r(u_j, u_i)^{(ji)}, r(u_0, u_j)^{(0j)}] +
    [r(u_0, u_i)^{(0i)}, r(u_i, u_j)^{(ij)}] = 0.
\end{align*}
Using the explicit form of the $r$-matrix and using the invariance of the quadratic Casimir element $\Omega$ we can rewrite the equation above as follows
\begin{align*}
    \left[r(u_0, u_i)^{(0i)} - r(u_j, u_i)^{(0i)}, \frac{u_j^m \Omega^{(0j)}}{u_0 - u_j}\right] 
    &+ [r(u_0, u_i)^{(0i)} + r(u_j, u_i)^{(ji)}, g(u_0, u_j)^{(0j)}] \\
    &+ [r(u_0, u_i)^{(0i)}, r(u_i, u_j)^{(ij)}] = 0.
\end{align*}
Taking the limit $u_0 \to u_j$ we get
\begin{align*}
    u_j^m [\partial_{u_j}r(u_j, u_i)^{(0i)}, \Omega^{(0j)}]
    &+ [r(u_j, u_i)^{(0i)} 
    + r(u_j, u_i)^{(ji)}, g(u_j, u_j)^{(0j)}] 
    + [r(u_j, u_i)^{(0i)}, r(u_i, u_j)^{(ij)}] = 0.
\end{align*}
Swapping factors $j$ and $0$ in the equality above we get
\begin{align*}
    u_j^m [\partial_{u_j}r(u_j, u_i)^{(ji)}, \Omega^{(0j)}]
    &+ [r(u_j, u_i)^{(0i)} 
    + r(u_j, u_i)^{(ji)}, \tau(g(u_j, u_j))^{(0j)}] 
    + [r(u_j, u_i)^{(ji)}, r(u_i, u_j)^{(i0)}] = 0.
\end{align*}
Summing the last two equations and applying the multiplication $m_{(0j)}$ gives 
\begin{equation}\label{eq:formula_S2}
    \begin{split}
        S_2 &= \frac{1}{2}[r(u_j,u_i)^{(ji)},g(u_j,u_j)^{(jj)} + \tau(g(u_j,u_j))^{(jj)}] \\&= -\frac{1}{2}m_{(0j)}([r(u_j,u_i)^{(0i)},r(u_i,u_j)^{(ij)}] + [r(u_j, u_i)^{(ji)}, r(u_i, u_j)^{(i0)}])
    \end{split}
\end{equation}
Similarly, rewriting GCYBE in the form
\begin{align*}
    \left[\frac{u_i^m \Omega^{(0i)}}{u_0-u_i}, r(u_0, u_j)^{(0j)} - r(u_i, u_j)^{(0j)}\right] 
    &+ [g(u_0, u_i)^{(0i)}, r(u_0, u_j)^{(0j)} + r(u_i, u_j)^{(ij)}] \\
    &+ [r(u_j, u_i)^{(ji)}, r(u_0, u_j)^{(0j)}]
    = 0,
\end{align*}
taking the limit $u_0 \to u_i$
\begin{align*}
    u_i^m [\Omega^{(0i)}, \partial_{u_i} r(u_i, u_j)^{(0j)}] 
    &+ [g(u_i, u_i)^{(0i)}, r(u_i, u_j)^{(0j)} + r(u_i, u_j)^{(ij)}] \\
    &+ [r(u_j, u_i)^{(ji)}, r(u_i, u_j)^{(0j)}]
    = 0
\end{align*}
and swapping $i$ and $0$ factors results in
\begin{align*}
u_i^m[\Omega^{(0i)},\partial_{u_i}r(u_i,u_j)^{(ij)}] &+ [\tau(g(u_i,u_j))^{(0i)},r(u_i,u_j)^{(0j)}+r(u_i,u_j)^{(ij)}] \\& [r(u_j,u_i)^{(j0)},r(u_i,u_j)^{(ij)}] = 0.
\end{align*}
Summing these two terms and applying $m_{0i}$ we get
\begin{equation}\label{eq:formula_S3}
    \begin{split}
        S_3 &= \frac{1}{2}[g(u_i,u_i)^{(ii)} + \tau(g(u_i,u_i))^{(ii)},r(u_i,u_j)^{(ij)}] \\&= - \frac{1}{2}m_{0i}([r(u_j,u_i)^{(ji)},r(u_i,u_j)^{(0j)}] + [r(u_j,u_i)^{(j0)},r(u_i,u_j)^{(ij)}]).
    \end{split}
\end{equation}
Combining \cref{eq:formula_for_S1,eq:formula_S2,eq:formula_S3}we conclude
\begin{equation*}
    [H_i,H_j] = S_1 + S_2 + S_3 = 0.
\end{equation*}

%%%%%%%%%%%%%% APPENDICES %%%%%%%%%%%%%%
\appendix
\section{List of notations}%
\label{sec:notations}

\begin{adjustbox}{center}
\renewcommand{\arraystretch}{1.7}
\begin{tabular}[c]{ |c|c|  }
\hline
Symbol & Meaning
\\ \hhline{|=|=|}
\(F\) & Algebraically closed field of characteristic 0
\\ \hline
\(\fg\) & Finite-dimensional simple Lie algebra over \(F\)
\\ \hline
\(\kappa\) & The Killing form on $\fg$
\\ \hline
\(\Omega\) & The quadratic Casimir element in $\fg \ot \fg$
\\ \hline
\(\fh\) & Fixed Cartan subalgebra of \(\fg\)
\\ \hline
\(\Delta = \Delta_+ \sqcup \Delta_-\) & Polarized root system of \(\fg\) with respect to \(\fh\)
\\ \hline
\(H_\alpha, E_{\pm \alpha}\) & \(E_{\pm \alpha} \in \fg_{\pm \alpha}\) for \(\alpha \in \Delta_+\) are chosen such that \(\kappa(E_\alpha,E_{-\alpha}) = 1\) and \(H_\alpha = [E_\alpha,E_{-\alpha}]\)
\\ \hline
\(\pi = \{\alpha_1,\dots,\alpha_n\}\) & Simple roots of \(\Delta = \Delta_+ \sqcup \Delta_-\)
\\ \hline
\(\alpha_0 = \sum_{i = 1}^n k_i\alpha_i\) & Maximal root in \(\Delta\) and its expansion into simple roots
\\ \hline
\(V^*\) & Dual of a vector space \(V\)
\\ \hline 
\(V[x]\) & Polynomials in one variable with coefficients in a vector space \(V\) 
\\ \hline
\(V[\![x]\!]\) & Formal Taylor power series in one variable with coefficients in a vector space \(V\) 
\\ \hline
\(V(\!(x)\!)\) & Formal Laurent power series in one variable with coefficients in a vector space \(V\) 
\\ \hline
\(V^a,V^{a,b}\) & For \(a,b\in F\), \(V^a = (x^{-1}-a)V[x^{-1}]\) and \(V^{a,b} = x(x^{-1}-a)(x^{-1}-b)V[x^{-1}]\)
\\ \hline
\(L_m\) & The Lie algebra \(\fg(\!(x)\!) \times \fg[x]/x^m\fg[x]\)
\\ \hline
\(\mathfrak{d}\) & The subalgebra \(\{(a,a) \mid a \in \fg \} \) of $\fg \times \fg$
\\ \hline
\(\mathfrak{D}\) & The subalgebra \(\{(f,[f])\mid f \in \fg[\![x]\!]\}\) of \(L_m\)
\\ \hline
\(\Delta_\pm^{<m \alpha_i}\) & \(\{
   \alpha \in \Delta_\pm \mid \alpha = \pm \sum_{i=1}^n c_i \alpha_i, \ 0 \le c_i < m\} \subseteq \Delta_\pm\)
\\ \hline
    \(\Delta_{\pm}^{\ge m\alpha_i}\) & \(\{\alpha \in \Delta_\pm \mid \alpha = \pm \sum_{i=1}^n c_i \alpha_i, \ c_i \ge m \} \subseteq \Delta_\pm\)
\\ \hline
    \(\mathfrak{P}_i\) &  Maximal parabolic subalgebra of $\fg(\!(x)\!)$ corresponding to $\alpha_i \in \{\alpha_0, \dots, \alpha_n \}$
\\ \hline
    \(W_\phi\) &  for a linear map $\phi \colon \fh \to \fh$ it is $\{x^{-n}(x^{-1}h - \phi(h)) \mid h \in \fh,n\in\bZ_{\ge 0}\}$
\\ \hline
\end{tabular}
\end{adjustbox}

%%%%%%%%%%%%%% REFERENCES %%%%%%%%%%%%%%%%
\newpage
\pagestyle{plain}
\section*{Statements and Declarations
}
The authors have no competing interests to declare that are relevant to the content of this article.

\printbibliography

\end{document}